%% file: main.tex
\documentclass[11pt, a4paper]{article}  

\usepackage[hidelinks, pdftex]{hyperref}
\usepackage[numbers]{natbib}
\usepackage{algpseudocode}
\usepackage{amsthm, bm}

\usepackage{algorithm}
\usepackage{pifont}
\usepackage{amsmath}
\usepackage[nameinlink]{cleveref}
\usepackage{amssymb}

\usepackage{color}
\usepackage{csvsimple}
\usepackage{enumerate, enumitem}
\usepackage{mathrsfs} 
\usepackage{titlesec}
\usepackage{multicol, multirow, tabularx}
\usepackage{mathtools, calc}
\usepackage{verbatim}
\usepackage{listings}
\usepackage{graphicx}
\usepackage{lscape, textcomp}
\usepackage{tikz, subcaption}
\usepackage{xparse}

\usepackage{url}
\usepackage{enumitem}
\usepackage[tableposition=top]{caption}
\usepackage[toc,page]{appendix}
\usepackage[british]{babel}
\usepackage{tkz-euclide}
\usetikzlibrary{arrows}
\usepackage{fancyhdr} 
\usepackage{array, stackengine}
\usepackage{pbox}
\usepackage{sectsty}
\usepackage{tablefootnote, longtable}
\usepackage{dsfont}

\sectionfont{\fontsize{12}{15}\selectfont}
\subsectionfont{\fontsize{11}{15}\selectfont}
\subsubsectionfont{\fontsize{11}{15}\selectfont}

\newcolumntype{L}[1]{>{\raggedright\let\newline\\\arraybackslash\hspace{0pt}}m{#1}}
\newcolumntype{C}[1]{>{\centering\let\newline\\\arraybackslash\hspace{0pt}}m{#1}}
\newcolumntype{R}[1]{>{\raggedleft\let\newline\\\arraybackslash\hspace{0pt}}m{#1}}

\definecolor{blue}{rgb}{0.0, 0.0,1.0}
\definecolor{gray}{rgb}{0.7, 0.7, 0.7}

\newcommand{\mtx}[1]{\boldsymbol{#1}}
\newcommand{\mvec}[1]{\boldsymbol{#1}}
\newcommand{\wt}[1]{\widetilde{#1}}

\theoremstyle{plain}
\newtheorem{theorem}{Theorem}[section]
\newtheorem{lemma}[theorem]{Lemma}

\newtheorem{corollary}[theorem]{Corollary}

\newtheorem{assump}[theorem]{Assumption}

\theoremstyle{definition}

\newtheorem{definition}[theorem]{Definition}
\newtheorem*{rem*}{Remark}
\newtheorem*{warning*}{Warning}
\newtheorem{rem}[theorem]{Remark}

\DeclareMathOperator{\range}{range}
\DeclareMathOperator{\prob}{\mathbb{P}}
\DeclareMathOperator{\vol}{Vol}

\DeclareMathOperator*{\argmin}{\arg\!\min}
\DeclareMathOperator{\vc}{vec}
\DeclareMathOperator{\erf}{\rm{erf}}
\DeclarePairedDelimiter{\ceil}{\lceil}{\rceil}

\newsavebox\CBox

\makeatletter
\newcommand{\subalign}[2][c]{%
	\if#1c\vcenter\else\vtop\fi{%
		\Let@ \restore@math@cr \default@tag
		\baselineskip\fontdimen10 \scriptfont\tw@
		\advance\baselineskip\fontdimen12 \scriptfont\tw@
		\lineskip\thr@@\fontdimen8 \scriptfont\thr@@
		\lineskiplimit\lineskip
		\ialign{\hfil$\m@th\scriptstyle##$&$\m@th\scriptstyle{}##$\hfil\crcr
			#2\crcr
		}%
	}%
}
\makeatother

\makeatletter
\renewcommand*{\@fnsymbol}[1]{\ensuremath{\ifcase#1\or *\or \ddagger\or \mathsection\or \vee\or \wedge\or \dagger\or
		\mathsection\or \mathparagraph\or \|\or **\or \dagger\dagger
		\or \ddagger\ddagger \else\@ctrerr\fi}}

\numberwithin{equation}{section}

\begin{document}
\title{Constrained global optimization of functions with low effective dimensionality using multiple random embeddings}

\author{
	Coralia Cartis \thanks{The order of the authors is alphabetical; the third author is the primary contributor. Mathematical Institute, University of Oxford, Radcliffe Observatory Quarter, Woodstock Road,
		Oxford, OX2 6GG, UK; \texttt{cartis,massart,otemissov@maths.ox.ac.uk}} \textsuperscript{\normalfont,} \thanks{The Alan Turing Institute, The British Library, London, NW1 2DB, UK. This work was supported by The Alan Turing Institute under The Engineering and Physical Sciences Research Council (EPSRC) grant EP/N510129/1 and under the Turing project scheme.} 
	\and Estelle Massart \footnotemark[1] \textsuperscript{\normalfont,} \thanks{National Physical Laboratory, Hampton Road, Teddington, Middlesex, TW11 0LW, UK. This author's work was supported by the National Physical Laboratory.} 
	\and
	Adilet Otemissov \footnotemark[1] \textsuperscript{\normalfont,}  \footnotemark[2] 
}

\date{\today}
\maketitle
\footnotesep=0.4cm

{\small
	\begin{abstract}
	We consider the bound-constrained global optimization of functions with low effective dimensionality, that are constant along an (unknown) linear subspace and only vary over the effective (complement) subspace. We aim to implicitly explore the intrinsic low dimensionality of the constrained landscape using feasible random embeddings, in order to understand and improve the scalability of algorithms for the global optimization of these special-structure problems. A reduced subproblem formulation is investigated that solves the original problem over a random low-dimensional subspace subject to affine constraints, so as to preserve feasibility with respect to the given domain. Under reasonable assumptions, we show that the probability that the reduced problem is  successful in solving the original, full-dimensional problem is positive. Furthermore, in the case when the objective's
	effective subspace is aligned with the coordinate axes, we provide an asymptotic  bound on this success probability that captures its algebraic dependence on the effective and, surprisingly,  ambient dimensions. 
	 We then  propose  X-REGO,  a generic algorithmic framework that uses multiple random embeddings,  solving the above reduced
	 problem repeatedly, approximately and possibly, adaptively.  Using the success probability of the reduced subproblems,
	   we prove that X-REGO converges globally, with probability one, and linearly in the number of embeddings, to an $\epsilon$-neighbourhood of a constrained global minimizer. Our numerical experiments on special structure functions illustrate our theoretical findings and the improved scalability of X-REGO variants when coupled with state-of-the-art global --- and even local --- optimization solvers for the subproblems.

	\end{abstract}
	
	\bigskip
	
	\begin{center}
	\end{center}
}

\maketitle

\section{Introduction} \label{sec:intro}
\input{intro.tex}


\section{Preliminaries} \label{sec: prelim}
\input{preliminaries.tex}

\section{Estimating the success of the reduced problem}
\label{sec: estim_success}
\input{estim_proba.tex}



\section{The X-REGO algorithm and its global convergence}
\label{sec: Convergence}
\input{convergence_proof.tex}

\section{Numerical experiments}
\label{sec: Numerics}
\input{numerics.tex}

\section{Conclusions and future work}
\label{sec:conclusions}

We studied a generic global optimization framework, X-REGO, that relies on multiple random embeddings, for bound-constrained global optimization  of functions with low effective dimensionality. For each random subspace, a lower-dimensional bound-constrained subproblem is solved, using a global or even local algorithm. 
Theoretical guarantees of convergence and encouraging numerical experiments are presented, which are particularly quantified in terms of their problem dimension dependence for the case when the effective subspace is aligned with the coordinate axes. We note that the X-REGO algorithmic framework (\Cref{alg: AREGO}) can be applied to a general, continuous objective $f$ in (P) as the effective dimensionality assumption is not used; furthermore, our main global convergence result continues to hold under some assumptions (see Remark \ref{rem:generalXREGO}).

Our analysis relies on the assumption that the dimension of the random subspace is larger than the effective dimension. As the latter may be unknown in practice,
 this is a strong prerequisite.  One possibility is to estimate the effective dimension $d_e$ numerically, as in \cite{Sanyang2016}. Otherwise, one may consider extending the theoretical analysis  in this paper to the case $d \leq d_e$. A relevant recent reference is \cite{Kirschner19}, where \citeauthor{Kirschner19}  proved global convergence of an algorithm similar to A-REGO, but using one-dimensional subspaces, within the framework of Bayesian optimization. 

\bibliographystyle{plainnat}
{\footnotesize
	\bibliography{bibliography}}

\appendix
\section{Technical definitions and results}
\input{appendix_tech_def.tex}

\section{Derivation of the probability density function of \texorpdfstring{$\mvec{w}$}{w} }
\label{app:pdf_w}
\input{appendix_distrib_w.tex}

\section{Proof of \texorpdfstring{\Cref{cor:prob[RPXis_succ]=Omega_general_p}}{Theorem 3.8} and \texorpdfstring{\Cref{cor:prob[RPXis_succ]=Omega_p=0}}{Theorem 3.9}}
\label{app:proofs_of_corollaries}
\input{appendix_proofs.tex}

\section{Problem set}  
\label{app: Test set}

\input{appendix_problem_set.tex}

\end{document}

%% file: intro.tex

In this paper, we address the bound-constrained global optimization problem 
\begin{equation}
\tag{P}
\begin{aligned} \label{eq: GO}
f^* := \min_{\mvec{x}\in \mathcal{X}} & \;\; f(\mvec{x}), \\
\end{aligned}
\end{equation}
where $f:  \mathbb{R}^D \rightarrow \mathbb{R}$ is continuous, possibly non-convex and deterministic\footnote{Our analysis would be significantly more involved, but still possible, if $f$ is only well defined on $\mathcal{X}$. Note that in our X-REGO algorithm, we only query $f(\mvec{x})$ at feasible points $\mvec{x}\in\mathcal{X}$.}, and where, without loss of generality, $\mathcal{X} := [-1,1]^D\subseteq  \mathbb{R}^D$.

In an attempt to alleviate the curse of dimensionality of generic global optimization, we focus on objective functions with `\textit{low effective dimensionality}' \cite{Wang2016}, namely, those that only vary over a low-dimensional subspace (which may not necessarily be aligned with standard axes), and remain constant along its orthogonal complement. These functions are also known as objectives with `\textit{active subspaces}' \cite{Constantine2015} or `\textit{multi-ridge}'  \cite{Fornasier2012, Tyagi2014}. They are frequently encountered in  applications, typically when tuning (over)parametrized models and processes, such as in hyper-parameter optimization for neural networks \cite{Bergstra2012}, heuristic algorithms for combinatorial optimization problems \cite{Hutter2014},  complex engineering and physical simulation problems \cite{Constantine2015} as in climate modelling \cite{Knight2007}, and policy search and  dynamical system control \cite{Zhang19, Frohlich2020}. 

When the objective has low effective dimensionality and the effective subspace of variation is known, it is straightforward to cast \eqref{eq: GO} into a lower-dimensional problem which has the same global minimum $f^*$ by restricting it to and solving \eqref{eq: GO} only within  this important subspace.
Typically, however, the effective subspace is  unknown, and
random embeddings have been proposed to reduce 
the size of \eqref{eq: GO} and hence the cost of its solution, while attempting to preserve the problem's (original) global minimum values. 
In this paper, we investigate the following 
feasible formulation of the reduced randomised problem,
\begin{equation} \label{eq: AREGO}
\tag{RP$\mathcal{X}$}
\begin{aligned}
\min_{\mvec{y}}  &\;\; f(\mtx{A}\mvec{y}+\mvec{p})  \\
\text{subject to} &\;\; \mtx{A}\mvec{y} + \mvec{p} \in \mathcal{X},
\end{aligned}
\end{equation}
where $\mtx{A}$ is a $D\times d$ Gaussian random matrix (see \Cref{def: Gaussian_matrix}) with $d \ll D$,
 and where $\mvec{p} \in \mathcal{X}$ is user-defined and provides additional flexibility that we exploit algorithmically. Our approach needs the following clarification.

\begin{definition} \label{def: successful_AREGO}
	We say that \eqref{eq: AREGO} is \textit{successful} if there exists $\mvec{y}^* \in \mathbb{R}^d$ such that $f(\mtx{A}\mvec{y}^*+\mvec{p}) = f^*$ and $\mtx{A}\mvec{y}^* + \mvec{p} \in \mathcal{X}$.
\end{definition}
We derive a lower bound on the probability that  \eqref{eq: AREGO} is successful in the case when $d$ is equal to or larger than the effective dimension. We show that this success probability is positive and that it depends  on both the effective subspace and the ambient dimensions\footnote{A brief description, without proofs, of the main results of this paper has appeared as a four-page conference proceedings paper in the ICML Workshop ``Beyond first order methods in ML systems'' (2020), see \url{https://drive.google.com/file/d/1JxQc9rSK8GYchKnDp0dhwEa4f3AeyNeb/view}.}. However, in the case when the effective subspace is aligned with the coordinate axes, we show that the dependence on $D$ in this lower bound is at worst algebraic.
 We then propose X-REGO ($\mathcal{X}$ - Random Embeddings for Global Optimization), a generic algorithmic framework for solving \eqref{eq: GO} using multiple random embeddings. Namely, X-REGO solves \eqref{eq: AREGO} repeatedly with different $\mtx{A}$ and possibly different $\mvec{p}$, and can use  any global optimization algorithm for solving the reduced problem \eqref{eq: AREGO}.
Using the computed  lower bound on the probability of success of \eqref{eq: AREGO}, we derive a global convergence result for X-REGO, showing that as the number of random embeddings increases, X-REGO converges linearly, with probability one, to  an
$\epsilon$-neighbourhood of a global minimizer of (P).

\paragraph{Existing relevant literature.} 
Optimization of functions with low effective dimensionality has been recently studied primarily as an attempt to remedy the scalability challenges of Bayesian Optimization (BO),
such as in \cite{Djolonga2013, Wang2016, Garnett2014, Li2016, Eriksson2018}.
Investigations of these special-structure problems have been extended beyond BO,  to derivative-free optimization \cite{QianHuYu2016}, multi-objective optimization \cite{Qian2017} and evolutionary methods \cite{Sanyang2016, Demo2020}. 
 As the effective subspace is generally unknown, some existing approaches learn the effective subspace beforehand \cite{Fornasier2012,Tyagi2014, Djolonga2013, Eriksson2018}, while others estimate it during the optimization, updating the estimate as new information becomes available on the objective function \cite{Garnett2014, Zhang19, Chen2020, Demo2020}. We focus here on an alternative approach, bypassing the subspace learning phase, and optimizing directly over random low-dimensional subspaces, as proposed in \cite{Wang2016, Binois2014, Binois2017, Kirschner19}.

 \citeauthor{Wang2016} \cite{Wang2016} propose the REMBO algorithm, that solves,
 using Bayesian methods, a single reduced subproblem,
 \begin{equation} \label{eq: REGO}
\tag{RP}
\begin{aligned}
\min_{\mvec{y}}  &\;\; f(\mtx{A}\mvec{y})  \\
\text{subject to} &\;\; \mvec{y} \in \mathcal{Y} = [-\delta, \delta]^d,
\end{aligned}
\end{equation}
where $\mtx{A}$ is as above, and  $\delta > 0$. They evaluate the probability that the solution of \eqref{eq: REGO} corresponds to a solution of the original problem (P)  in the case when the effective subspace is aligned with coordinate axes and when $d=d_e$, where $d_e$ denotes the dimension of the effective subspace; they show that this  probability of success of \eqref{eq: REGO} depends on the parameter $\delta$ (the size of the $\mathcal{Y}$ box), and it decreases as $\delta$ shrinks. Conversely, setting $\delta$ large may result in large computational costs to solve \eqref{eq: REGO}. Thus, a careful calibration of $\delta$ is needed for good algorithmic performance.
The theoretical analysis in \cite{Wang2016}  has been extended by \citeauthor{Sanyang2016} \cite{Sanyang2016}, where the probability of success of \eqref{eq: REGO} is quantified in the case $d \geq d_e$;  an algorithm, called REMEDA, is also proposed in \cite{Sanyang2016} that uses Gaussian random embeddings in the framework of evolutionary methods for high-dimensional unconstrained global optimization. 
 
 In the recent preprint \cite{Cartis2020}, we further extend these analyses to arbitrary effective subspaces (i.e., not necessarily aligned with the coordinate axes) and random embeddings of dimension $d\geq d_e$, and consider the wider framework of generic unconstrained high-dimensional global optimization. We propose the REGO algorithm, that replaces the high-dimensional problem \eqref{eq: GO} (with $\mathcal{X} = \mathbb{R}^{D})$, by a {\it single} reduced problem \eqref{eq: REGO}, and solves \eqref{eq: REGO} using any global optimization algorithm. Instead of  estimating solely the norm of an optimal solution of \eqref{eq: REGO}, as in \cite{Wang2016, Sanyang2016}, we derive its exact probability distribution. Furthermore, we show that its squared Euclidean norm (when appropriately scaled) follows an inverse chi-squared distribution with $d - d_e +1$ degrees of freedom, and use a tail bound on the chi-squared distribution to get a lower bound on the probability of success of \eqref{eq: REGO}. Our theory and numerical experiments indicate that, under suitable assumptions, the success of \eqref{eq: REGO} is essentially independent on $D$, but depends mainly on two factors: the gap between the subspace dimension $d$ and the effective dimension $d_e$, and the ratio between $\delta$ (the size of the low-dimensional domain), and the distance from the origin (the centre of the original domain $\mathcal{X}$) to the closest affine subspace of global minimizers.
 
 In contrast to \cite{Sanyang2016} and \cite{Cartis2020}, the present case of the {\it constrained} problem (P) poses a new challenge: a solution $\mvec{y^*}$ of \eqref{eq: REGO} is not necessarily feasible for the full-dimensional problem \eqref{eq: GO} (i.e., $\mtx{A} \mvec{y}^* \notin \mathcal{X}$). To remedy this, \citeauthor{Wang2016} \cite{Wang2016} endow REMBO with an additional step that projects $\mtx{A} \mvec{y}^*$ onto $\mathcal{X}$. However, they observe that using a classical kernel (such as the squared exponential kernel) directly on the low-dimensional domain $\mathcal{Y}$ may lead to an over-exploration of the regions on which the projection map onto $\mathcal{X}$ is not injective. The design of kernels avoiding this over-exploration has been tackled in \cite{Binois2014, Binois2017}.
 \citeauthor{Binois2017} \cite{Binois2017} further advances the discussion regarding the choice of the low-dimensional domain $\mathcal{Y}$ in \eqref{eq: REGO} and computes an `optimal' set $\mathcal{Y}^* \subset \mathbb{R}^d$, i.e., a set that has minimum (here, infimum) volume among all the sets $\mathcal{Y} \subset \mathbb{R}^d$ for which the image of the mapping $\mathcal{Y} \to \mathcal{X}: \mvec{y} \mapsto p_\mathcal{X}(\mtx{A} \mvec{y})$ contains the `maximal embedded set' $\{ p_\mathcal{X}(\mtx{A} \mvec{y}) : \mvec{y} \in \mathbb{R}^d \}$, where $p_{\mathcal{X}}(\mvec{x})$ is the classical Euclidean projection of $\mvec{x}$ on $\mathcal{X}$. They show that $\mathcal{Y^*}$ has an intricate representation when the dimension of the full-dimensional problem is large, and propose to replace the Euclidean projection map $p_{\mathcal{X}}$ suggested by \citeauthor{Wang2016} \cite{Wang2016} by an alternative mapping for which an `optimal' low-dimensional domain has nicer properties.
 \citeauthor{Nayebi2019} \cite{Nayebi2019} circumvent the projection step by replacing the Gaussian random embeddings of \eqref{eq: REGO} by random embeddings defined using hashing matrices, and choose $\mathcal{Y} = [-1,1]^d$. This choice guarantees that any solution of the low-dimensional problem provides an admissible solution for the full-dimensional problem in the case $\mathcal{X} = [-1,1]^D$.
 
 The need to combine optimization algorithms that rely on random Gaussian embeddings with a projection step has also been recently discussed in \cite{Letham2020}, where it is suggested to replace the formulation \eqref{eq: REGO} by  \eqref{eq: AREGO}, that we also consider in this paper. However, \citeauthor{Letham2020} \cite{Letham2020} do not provide analytical estimates of the probability of success of this new formulation, solely evaluating it numerically using Monte-Carlo simulations; they also do not use multiple random embeddings.  Our proposed X-REGO algorithmic framework (and more precisely, the adaptive variant A-REGO described in Section \ref{sec: Numerics}) is closely related to the sequential algorithm proposed by~\citeauthor{QianHuYu2016} \cite{QianHuYu2016}, in the framework of unconstrained derivative-free optimization of functions with approximate low-effective dimensionality, and to the algorithm proposed in \cite{Kirschner19} for constrained Bayesian optimization of functions with low-effective dimension, using one-dimensional random embeddings. However,  our results rely on the assumption that the subspace dimension $d$ is larger than the effective dimension $d_e$, and so our approach significantly differs from \cite{Kirschner19}. Very recently, \citeauthor{TranThe2020} \cite{TranThe2020} have proposed an algorithm that uses several low-dimensional (deterministic) embeddings in parallel for Bayesian optimization of high-dimensional functions. 

Randomized subspace methods have recently attracted much interest for local or convex optimization problems; see for example, \cite{Nesterov2017, kozak2019, gower2019, hanzely2020}; no low effective dimensionality assumption
is made in these works. Finally, we note that the main step in our convergence analysis consists in deriving a lower bound on the probability that a random subspace of given dimension intersects a given set (the set of approximate global minimizers), which is an important problem in stochastic geometry, see, e.g., the extensive discussion by \citeauthor{Oymak2017} \cite{Oymak2017}. Unlike the results presented in \cite{Oymak2017}, our results do not involve statistical dimensions of sets, which are unknown and, in our case, problem dependent.


\paragraph{Our contributions.}
 Here we investigate a general random embedding framework for the \emph{bound-constrained} global optimization of functions with low effective dimensionality. This framework replaces the original, potentially high-dimensional problem \eqref{eq: GO} with several reduced and randomized subproblems of the form  \eqref{eq: AREGO}, which directly ensures feasibility of the iterates with respect to the constraints.
 
Using various properties of Gaussian matrices and a useful result from \cite{Cartis2020},
we derive a lower bound on the probability of success of \eqref{eq: AREGO} when $d\geq d_e$. 
To achieve this, we provide a sufficient condition for the  success of \eqref{eq: AREGO} that depends on a random vector $\mvec{w}$, which in turn, is a function of the embedding matrix $\mtx{A}$, the parameter $\mvec{p}$ of \eqref{eq: AREGO} and an arbitrary global minimizer $\mvec{x}^*$ of \eqref{eq: GO}. We show that $\mvec{w}$ follows a $(D-d_e)-$dimensional $t$-distribution with $d - d_e +1$ degrees of freedom, and provide a lower bound on the probability of success of \eqref{eq: AREGO} in terms of the integral of the probability density function of $\mvec{w}$ over a given closed domain. In the case when the effective subspace is aligned with the coordinate axes, the closed domain simplifies to a $(D-d_e)-$dimensional box, and we provide an asymptotic expansion of the integral of the probability density function over the box, when $D\rightarrow \infty$ (and $d$ and $d_e$ are fixed).
Our theoretical analysis, backed by numerical testing, indicates that the probability of success of \eqref{eq: AREGO} decreases with the dimension $D$ of the original problem \eqref{eq: GO}. However, in the case when the effective subspace is aligned with the coordinate axes, we show that it decreases at most algebraically with the ambient dimension $D$ for some useful choices of $\mvec{p}$.

 We also propose the X-REGO algorithm, a generic framework for the constrained global optimization problem (P) that sequentially or in parallel solves multiple subproblems \eqref{eq: AREGO}, varying $\mtx{A}$ and also possibly $\mvec{p}$.
  We prove global convergence of X-REGO to a set of approximate global minimizers of (P) with probability one, with linear rate in terms of the number of subproblems solved.   This result requires mild assumptions   on  problem (P) ($f$ is Lipschitz continuous and (P) admits a strictly feasible solution) and on the algorithm used to solve the reduced problem (namely, it must solve \eqref{eq: AREGO} globally and approximately, to required accuracy), and allows a diverse set of possible choices of $\mvec{p}$ (random, fixed, adaptive, deterministic). Our convergence proof crucially uses our result that the probability of success of \eqref{eq: AREGO} is positive and uniformly bounded away from zero with respect to the choice of $\mvec{p}$, and hence, assumes that $d\geq d_e$.
  
 
 We provide an extensive numerical comparison of several variants of X-REGO on a set of test problems with low effective dimensionality, using three different solvers for \eqref{eq: AREGO}, namely, BARON \cite{Sahinidis2014}, DIRECT \cite{Finkel2003} and (global and local) KNITRO \cite{Byrd2006}. We find that X-REGO variants show significantly improved scalability with most solvers, as the ambient problem dimension grows, compared to directly using the respective solvers on the test set. Notable efficiency was obtained in particular when local KNITRO was used to solve the subproblems and the points $\mvec{p}$ were updated to the `best' point (with the smallest value of $f$) found so far.

\paragraph{Paper outline.}
In \Cref{sec: prelim}, we recall the definition of functions with low effective dimensionality and some existing results that we will use in our analysis. \Cref{sec: estim_success} derives lower bounds for the probability of success of \eqref{eq: AREGO}. The X-REGO algorithm and its global convergence are then presented in \Cref{sec: Convergence}, while  in \Cref{sec: Numerics}, different X-REGO variants are compared numerically on benchmark problems using three optimization solvers (DIRECT, BARON and KNITRO) for the subproblems. Our conclusions are drawn in \Cref{sec:conclusions}.

\paragraph{Notation.}
We use bold capital letters for matrices ($\mtx{A}$) and bold lowercase letters ($\mvec{a}$) for vectors. In particular, $\mtx{I}_D$ is the $D \times D$ identity matrix and $\mvec{0}_D$, $\mvec{1}_D$ (or simply $\mvec{0}$, $\mvec{1}$) are the $D$-dimensional vectors of zeros and ones, respectively. We write $a_i$ to denote the $i$th entry of $\mvec{a}$ and write $\mvec{a}_{i:j}$, $i<j$, for the vector $(a_i \; a_{i+1} \cdots a_{j})^T$. We let  $\range(\mtx{A})$ denote the linear subspace spanned in $\mathbb{R}^D$ by the columns of $\mtx{A} \in \mathbb{R}^{D \times d}$. We write $\langle \cdot , \cdot \rangle$, $\| \cdot \|$ and $\| \cdot \|_{\infty}$ for the usual Euclidean inner product, the Euclidean norm and the infinity norm, respectively. Where emphasis is needed, for the Euclidean norm we also use $\| \cdot \|_2$. 

Given two random variables (vectors) $x$ and $y$ ($\mvec{x}$ and $\mvec{y}$), the expression $x \stackrel{law}{=} y$ ($\mvec{x} \stackrel{law}{=} \mvec{y}$) means that $x$ and $y$ ($\mvec{x}$ and $\mvec{y}$) have the same distribution. We reserve the letter $\mtx{A}$ for a $D\times d$ Gaussian random matrix (see \Cref{def: Gaussian_matrix}) and write $\chi^2_n$ to denote a chi-squared random variable with $n$ degrees of freedom (see \Cref{def:chi-squared_rv}). 

Given a point $\mvec{a} \in \mathbb{R}^D$ and a set $S$ of points in $\mathbb{R}^D$, we write $\mvec{a}+S$ to denote the set $\{ \mvec{a}+\mvec{s}: \mvec{s} \in S \}$. Given functions $f(x):\mathbb{R}\rightarrow \mathbb{R}$ and $g(x):\mathbb{R}\rightarrow \mathbb{R}^+$, we write $f(x) = \Theta(g(x))$ as $x \rightarrow \infty$ to denote the fact that there exist positive reals $M_1,M_2$ and a real number $x_0$ such that, for all $x \geq x_0$, $M_1g(x)\leq|f(x)| \leq M_2g(x)$.

%% file: preliminaries.tex
\subsection{Functions with low effective dimensionality}

\begin{definition}[Functions with low effective dimensionality
\cite{Wang2016}] \label{def: simple_function}
	A function $f : \mathbb{R}^D \rightarrow \mathbb{R}$ has effective dimension $d_e$ if there exists a linear subspace $\mathcal{T}$ of dimension $d_e$ such that for all vectors $\mvec{x}_{\top}$ in $\mathcal{T}$ and $\mvec{x}_{\perp}$ in $\mathcal{T}^{\perp}$ (the orthogonal complement of $\mathcal{T}$), we have
	\begin{equation}\label{eq: def_fun_eff_dim}
		f(\mvec{x}_{\top} + \mvec{x}_{\perp}) = f(\mvec{x}_{\top}),
	\end{equation}
	and $d_e$ is the smallest integer satisfying \eqref{eq: def_fun_eff_dim}.
\end{definition}

The linear subspaces $\mathcal{T}$ and $\mathcal{T}^{\perp}$ are called the \textit{effective} and \textit{constant} subspaces of $f$, respectively.  In this paper, we make the following assumption on the function $f$. \begin{assump}\label{ass:AREGO_fun_eff_dim}
	The function $f:\mathbb{R}^D \rightarrow \mathbb{R}$ is continuous and has effective dimensionality $d_e$ such that $d_e < D$ and $d_e \leq d$, with effective subspace\footnote{Note that $\mathcal{T}$ in \Cref{ass:AREGO_fun_eff_dim} may not be aligned with the standard axes. } $\mathcal{T}$ and constant subspace $\mathcal{T}^{\perp}$ spanned by the columns of the orthonormal matrices $\mtx{U} \in \mathbb{R}^{D \times d_e}$ and $\mtx{V} \in \mathbb{R}^{D\times(D-d_e)}$, respectively. We let $\mvec{x}_\top = \mtx{U} \mtx{U}^T \mvec{x}$ and $\mvec{x}_\perp = \mtx{V} \mtx{V}^T \mvec{x}$, the unique Euclidean projections of any vector $\mvec{x} \in \mathbb{R}^D$ onto $\mathcal{T}$ and $\mathcal{T}^\perp$, respectively.
\end{assump}

 We define the set of feasible global minimizers of problem \eqref{eq: GO},
\begin{equation}\label{def:G}
	G := \{\mvec{x} \in \mathcal{X}: f(\mvec{x}) = f^* \}.
\end{equation}
Note that, for any $\mvec{x}^* \in G$ with Euclidean projection $\mtx{x}_\top^*$ on the effective subspace $\mathcal{T}$, and for any $\tilde{\mvec{x}} \in \mathcal{T}^\perp$, we have
\begin{equation} \label{eq: f^*=f(x_top^*)}
f^* = f(\mvec{x}^*) = f(\mvec{x}^*_{\top} + \tilde{\mvec{x}}) = f(\mvec{x}_{\top}^*).
\end{equation}
The minimizer $\mvec{x}_{\top}^*$ may lie outside $\mathcal{X}$, and furthermore, there may be multiple points $\mvec{x}_{\top}^*$ in $\mathcal{T}$ satisfying $f^* = f(\mvec{x}^*_\top)$ as illustrated in \cite[Example 1.1]{Cartis2020}. 
Thus, the set $G$ is (generally)\footnote{Except in degenerate cases, see \Cref{def:G^*_degenerate}.} a union of (possibly infinitely many) ($D-d_e$)-dimensional simply-connected polyhedral sets, each corresponding to a particular $\mvec{x}_{\top}^*$. If $\mvec{x}_{\top}^*$ is unique, i.e., every global minimizer $\mvec{x}^* \in G$ has the same Euclidean projection $\mvec{x}_\top^*$ on the effective subspace, then $G$ is the $(D-d_e)$-dimensional set $\{ \mvec{x} \in \mathcal{X} : \mvec{x} \in \mvec{x}_{\top}^* + \mathcal{T}^{\perp} \}$.

\begin{definition} \label{def:AREGO_choice_of_x^*}
	Suppose \Cref{ass:AREGO_fun_eff_dim} holds. For any global minimizer $\mvec{x}^* \in G$, let $G^* := \{ \mvec{x} \in \mathcal{X} : \mvec{x} \in \mvec{x}_{\top}^* + \mathcal{T}^{\perp} \}$ be the simply connected subset of $G$ that contains $\mvec{x}_{\top}^* = \mtx{U} \mtx{U}^T \mvec{x}^*$, and $\mathcal{G}^* := \{ \mvec{x} \in \mathbb{R}^D: \mvec{x} \in \mvec{x}_{\top}^* + \mathcal{T}^{\perp} \}$, the ($D-d_e$)-dimensional affine subspace that contains $G^*$.
\end{definition}


We can express $G^* = \mathcal{G}^* \cap \mathcal{X} = \{\mvec{x}_{\top}^*+\mtx{V} \mvec{g}:  -\mvec{1} \leq \mvec{x}_{\top}^* + \mtx{V} \mvec{g} \leq \mvec{1}, \mvec{g} \in \mathbb{R}^{D-d_e} \}$, where $\mtx{V}$ is defined in \Cref{ass:AREGO_fun_eff_dim}. For each $G^*$, we define the corresponding set of ``admissible'' $(D-d_e)$-dimensional vectors as
\begin{equation} \label{eq: barGstar}
    \bar{G}^* := \{\mvec{g} \in \mathbb{R}^{D-d_e} :  \mvec{x}_{\top}^* + \mtx{V} \mvec{g} \in G^*\}.
\end{equation}
Note that the set $G^*$ is $(D-d_e)$-dimensional if and only if the volume of the set $\bar{G}^*$ in $\mathbb{R}^{D-d_e}$, denoted by  $\vol(\bar{G}^*)$, is non-zero. In some particular cases, when the global minimizer $\mvec{x}^*$ in \Cref{def:AREGO_choice_of_x^*} is on the boundary of $\mathcal{X}$, the corresponding simply connected component $G^*$ may be of dimension strictly lower than $(D-d_e)$ and, hence, $\vol(\bar{G}^*) = 0$; a case we need to sometimes exclude from our analysis.
\begin{definition}\label{def:G^*_degenerate}
Let $G^*$ and $\bar{G}^*$ be defined as in \Cref{def:AREGO_choice_of_x^*} and \eqref{eq: barGstar}, respectively. We say that $G^*$ is non-degenerate if $\vol(\bar{G}^*) > 0$.
\end{definition}

The definitions and assumptions introduced in this section are  illustrated next in \Cref{fig:mapping_Y_to_X_affine}.

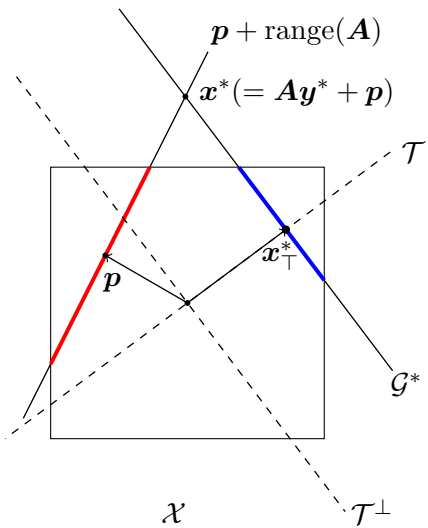
\begin{figure}[!t]
	\centering
	\begin{tikzpicture}[scale = 0.9]
	\draw (-2,-2) rectangle (2,2);
	\draw[line width = 0.5pt] (-1,13/3) -- (3,-1);
	\draw[line width = 0.5pt, ->] (0,0) -- (36/25,27/25);
	\draw[line width = 0.5pt, ->] (0,0) -- (-1.2,0.7);
	\draw[line width = 0.5pt] (-2.4,-1.7) -- (0.3,3.7);
	\draw[line width = 1.5pt, color = red] (-0.55,2) -- (-2,-0.9);
	\draw[line width = 0.5pt, dashed] (80/27, 20/9) -- (-72/27,-2);
	\draw[line width = 0.5pt, dashed] (-90/36, 90/27) -- (83/36,-83/27);
	\draw[line width = 1.5pt, color = blue] (2,1/3) -- (3/4,2);
	
	\draw[fill] (0,0) circle(1pt);
	\draw[fill] (36/25,27/25) circle(1.5pt);
	\draw[fill] (-0.03, 3.04) circle(1pt);
	\draw[fill] (-1.2, 0.7) circle(1pt);
	
	\node at (3.3,20/9) {$\mathcal{T}$};
	\node at (2.7,-83/27) {$\mathcal{T}^{\perp}$};
	\node at (34/25,0.7) {$\mvec{x}_{\top}^*$};
	\node at (-0.2,-3.1) {$\mathcal{X}$};
	\node at (3.2, -1.2) {$\mathcal{G}^*$};
	\node at (-1.1,0.3) {$\mvec{p}$};
	\node at (1.6,4) {$\mvec{p} + \range(\mtx{A})$};
	\node at (1.6,3.1) {$\mvec{x}^* (= \mtx{A}\mvec{y}^* + \mvec{p})$};
	
	\end{tikzpicture}
	\caption{Abstract illustration of the embedding of an affine $d$-dimensional subspace $\mvec{p}+\range(\mtx{A})$ into $\mathbb{R}^D$. The red line represents the feasible set of solutions along $\mvec{p}+\range(\mtx{A})$ and the blue line represents the set $G^*$. The random subspace intersects $\mathcal{G}^*$ at $\mvec{x}^*$, which is infeasible. }
	\label{fig:mapping_Y_to_X_affine}
\end{figure}

\paragraph*{Geometric description of the problem.}
\Cref{fig:mapping_Y_to_X_affine} sketches the linear mapping $ \mvec{y} \rightarrow \mtx{A}\mvec{y} + \mvec{p}$ that maps points from $\mathbb{R}^d$ to points in the affine subspace $\mvec{p}+\range(\mtx{A})$ in $\mathbb{R}^D$. This figure also illustrates the case of a non-degenerate 
simply-connected component $G^*$ of global minimizers (blue line; \Cref{def:AREGO_choice_of_x^*}), which here has dimension $D-d_e=1$. Degeneracy of 
$G^*$ (\Cref{def:G^*_degenerate}) would occur if $\mvec{x}^*_{\top}$ was a vertex of the domain $\mathcal{X}$, in which case the corresponding $G^*$ would be a singleton.

For \eqref{eq: AREGO} to be successful in solving the original problem \eqref{eq: GO}, \Cref{fig:mapping_Y_to_X_affine} illustrates it is sufficient that the red line segment  (the feasible set of (reduced) solutions in $\mathbb{R}^d$ mapped to $\mathbb{R}^D$)
intersects the blue line segment (the set $G^*$)\footnote{If $G^* = G$, this sufficient condition is also necessary; else, we need to check the other simply connected components of $G$ to decide whether \eqref{eq: AREGO} is successful or not.}. The blue and red line segments do not intersect in \Cref{fig:mapping_Y_to_X_affine},
but their prolongations outside $\mathcal{X}$ ($\mathcal{G}^*$ and $\mvec{p}+\range(\mtx{A})$) do\footnote{This is related to \cite[Theorem 2]{Wang2016}, which says that if the dimension of the embedded subspace ($d$) is greater than the effective dimension ($d_e$) of $f$ then $\mathcal{G}^*$ and $\mvec{p}+\range(\mtx{A})$ intersect with probability one. \citeauthor{Wang2016} \cite{Wang2016} have shown this result for the case $\mvec{p} = \mvec{0}$, but it can easily be generalized to arbitrary $\mvec{p}$.}. In \Cref{sec: charact_prelim}, we  review an existing characterization for a reduced minimizer ($\mvec{y}^*$ in \Cref{fig:mapping_Y_to_X_affine}), thus quantifying a specific intersection 
between the random subspace and $\mathcal{G}^*$. We then use this characterization
in \Cref{sec: estim_success} to derive a lower bound  on the probability of  $\mtx{A} \mvec{y}^* + \mvec{p}$ to belong to $G^*$, namely, to be feasible for the original problem \eqref{eq: GO}.

\subsection{Characterization of (unconstrained) minimizers in the reduced space}
\label{sec: charact_prelim}

This section summarizes results from \cite{Cartis2020} that characterize the distribution of a random reduced minimizer $\mvec{y}^*$ such that  $\mtx{A}\mvec{y}^*+\mvec{p}$ is an unconstrained minimizer of $f$.

Let $\mathcal{S}^* := \{ \mvec{y}^* \in \mathbb{R}^d : \mtx{A}\mvec{y}^* + \mvec{p} \in \mathcal{G}^* \}$, with $\mathcal{G}^*$ defined in \Cref{def:AREGO_choice_of_x^*}, be a subset of points $\mvec{y}^*$ corresponding to solutions of  minimizing $f$ over the entire $\mathbb{R}^D$.  With probability one, $\mathcal{S}^*$ is a singleton  if $d = d_e$ and has infinitely many points  if $d>d_e$ \cite[Corollary 3.3]{Cartis2020}. 
It is sufficient to find one of the reduced minimizers in $\mathcal{S}^*$, ideally one that  is easy to analyse, and that is close to the origin (i.e., the centre of the domain $\mathcal{X}$) in some norm so as to encourage the feasibility with respect to $\mathcal{X}$ of its image through $\mtx{A}$. An obvious candidate is the minimal Euclidean norm solution,
	\begin{equation} \label{eq: minimal_2_norm}
	\begin{aligned}
	\mvec{y}_2^{*}  = \argmin_{\mvec{y} \in \mathbb{R}^d} \;\;& \| \mvec{y} \|_2 \\
	\text{s.t.}\;\; & \mvec{y} \in \mathcal{S}^*.
	\end{aligned}
	\end{equation}

\begin{theorem}\cite[Theorem 3.1]{Cartis2020}\label{thm: By=z}
Suppose \Cref{ass:AREGO_fun_eff_dim} holds. 
Let $\mvec{x}^*$ be any global minimizer of \eqref{eq: GO} with Euclidean projection $\mvec{x}_\top^*$  on the effective subspace, and $\mvec{p}\in \mathcal{X}$, a given vector.
Let $\mtx{A}$ be a $D \times d$ Gaussian matrix. Then $\mvec{y}_2^{*}$ defined in \eqref{eq: minimal_2_norm} is given by
	\begin{equation} \label{eq: Euclidean_norm_sol_explicit_formula}
		\mvec{y}_2^* := \mtx{B}^T(\mtx{B}\mtx{B}^T)^{-1}\mvec{z}^*,
	\end{equation}
	which is the minimum Euclidean norm solution to the system  
	\begin{equation} \label{eq: By=z}
	\mtx{B} \mvec{y}^* = \mvec{z}^*,
	\end{equation}
	where  $ \mtx{B} = \mtx{U}^T \mtx{A}$ and  $\mvec{z}^* \in \mathbb{R}^{d_e} $ is uniquely defined by
	\begin{equation}\label{z*}
	\mtx{U}\mvec{z}^* = \mvec{x}_{\top}^* - \mvec{p}_{\top}, \,\,\text{with $\mvec{p}_{\top} = \mtx{U}\mtx{U}^T\mvec{p}$.}
	\end{equation}
\end{theorem}

\begin{proof}
See \Cref{app:proof_of_By=z}.
\end{proof}
\begin{rem} \label{rem: defined_B_and_z}
Note that $\mtx{B} = \mtx{U}^T\mtx{A}$ is a $d_e \times d$ Gaussian matrix, since $\mtx{U}$ has orthonormal columns (see \Cref{thm: orthog_inv_of_Gaussian_matrices}).
Also, \eqref{z*} implies $\| \mvec{z}^* \|_2 = \| \mvec{x}_{\top}^* - \mvec{p}_{\top}\|_2$.
\end{rem}

Using  \eqref{eq: Euclidean_norm_sol_explicit_formula} and various properties of Gaussian matrices,  \cite{Cartis2020} shows that the squared Euclidean norm of $\mvec{y}_2^*$ follows the (appropriately scaled) inverse chi-squared distribution.
\begin{theorem} (\cite[Theorem 3.7]{Cartis2020}) \label{thm: x^*_T/y^*_2_follow_chi_square}
Suppose \Cref{ass:AREGO_fun_eff_dim} holds. Let $\mvec{x}^*$ be any global minimizer of \eqref{eq: GO} and $\mvec{p}\in \mathcal{X}$ a given vector, with respective projections  $\mvec{x}_\top^*$ and
$\mvec{p}_\top$ on the effective subspace. Let $\mtx{A}$ be a $D \times d$ Gaussian matrix.
 Then,  $\mvec{y}^*_2$ defined in \eqref{eq: minimal_2_norm} satisfies
	$$ \frac{\| \mvec{x}_{\top}^* - \mvec{p}_{\top} \|^2_2}{\| \mvec{y}^*_2 \|^2_2} \sim \chi^2_{d-d_e+1} \quad \text{if } \mvec{x}_{\top}^* \neq \mvec{p}_{\top}.$$
If $\mvec{x}_{\top}^* = \mvec{p}_{\top}$, then $\mvec{y}^*_2 = \mvec{0}$ .
\end{theorem}

%% file: estim_proba.tex
This section derives lower bounds on the probability of success of \eqref{eq: AREGO}.  \Cref{lemma:AREGO_successful_y2star} lower bounds this probability  by that of a non-empty intersection between the random subspace $\mvec{p}+\range(\mtx{A})$ and an arbitrary simply-connected component $G^*$ of the set of global minimizers (\Cref{def:AREGO_choice_of_x^*}). This probability is further expressed in \Cref{thm:prob[success]>bound} in terms of a random vector $\mvec{w}$ that follows a multivariate $t$-distribution. From \Cref{sec: proba_success_pos} onwards, we derive positive and/or quantifiable lower bounds on the probability of success of \eqref{eq: AREGO}, while also trying to eliminate, wherever possible, the dependency of the lower bounds on the choice of $\mvec{p}$ and $G^*$.

\begin{lemma} \label{lemma:AREGO_successful_y2star}
Suppose  \Cref{ass:AREGO_fun_eff_dim} holds. Let $\mvec{x}^*$ be a(ny) global minimizer of \eqref{eq: GO}, $\mvec{p}\in \mathcal{X}$, a given vector, and $\mtx{A}$, a $D \times d$ Gaussian matrix.
Let $\mvec{y}_2^*$ be defined in \eqref{eq: minimal_2_norm}.
 The reduced problem \eqref{eq: AREGO} is successful in the sense of \Cref{def: successful_AREGO} if  $\mtx{A}\mvec{y}_2^*+\mvec{p} \in \mathcal{X}$, namely
\begin{equation}\label{eq:prob_there_exists_y}
    \prob[\text{\eqref{eq: AREGO} is successful}] \geq \prob[-\mvec{1} \leq \mtx{A}\mvec{y}_2^* + \mvec{p}\leq \mvec{1}].
\end{equation}
\end{lemma}
\begin{proof}
This is an immediate consequence of \Cref{def: successful_AREGO} and \eqref{eq: minimal_2_norm}, as the latter implies $\mtx{A} \mvec{y}_2^* + \mvec{p} \in \mathcal{G}^*$ and so $f(\mtx{A} \mvec{y}_2^* + \mvec{p}) = f^*$. 
\end{proof}

Let us further express \eqref{eq:prob_there_exists_y} as follows. Let $\mtx{Q} = (\mtx{U} \; \mtx{V})$, where $\mtx{U}$ and $\mtx{V}$ are defined in \Cref{ass:AREGO_fun_eff_dim}. Since $\mtx{Q}$ is orthogonal, we have
\begin{equation}\label{eq:Ay=QQ^TAy}
    \mtx{A}\mvec{y}_2^*= \mtx{Q}\mtx{Q}^T\mtx{A}\mvec{y}_2^*= \mtx{Q} \begin{pmatrix} \mtx{U}^T \\ \mtx{V}^T \end{pmatrix} \mtx{A} \mvec{y}_2^*.
\end{equation}
Using  \eqref{eq: Euclidean_norm_sol_explicit_formula}, we get $\mtx{U}^T\mtx{A} \mvec{y}_2^* = \mvec{z}^*$. Letting 
\begin{equation}\label{w-def}
\mvec{w} := \mtx{V}^T\mtx{A} \mvec{y}_2^*,
\end{equation}
we get


\begin{equation} \label{eq:Ay_2^*=x+Vw}
    \mtx{A}\mvec{y}_2^* = \mtx{Q} \begin{pmatrix} \mvec{z}^* \\ \mvec{w} \end{pmatrix} = \begin{pmatrix} \mtx{U} & \mtx{V} \end{pmatrix} \begin{pmatrix} \mvec{z}^* \\ \mvec{w} \end{pmatrix} = \mtx{U}\mvec{z}^* + \mtx{V}\mvec{w} = \mvec{x}_{\top}^* - \mvec{p}_{\top} + \mtx{V}\mvec{w},
\end{equation}
where in the last equality, we used \eqref{z*}.
By substituting $\mvec{p} = \mvec{p}_{\top}+\mvec{p}_{\perp}$ and \eqref{eq:Ay_2^*=x+Vw} in \eqref{eq:prob_there_exists_y}, we obtain
\begin{equation}\label{eq:prob_Ay^2*_is_between_-1_and_1}
\begin{aligned} 
   \prob[\text{\eqref{eq: AREGO} is successful}] & \geq \prob[-\mvec{1}\leq \mtx{A}\mvec{y}_2^* + \mvec{p} \leq \mvec{1}] \\
   & = \prob[-\mvec{1} \leq \mvec{x}_{\top}^* -\mvec{p}_{\top} + \mtx{V}\mvec{w} + \mvec{p}_{\top}+\mvec{p}_{\perp} \leq \mvec{1}] \\
   & = \prob[-\mvec{1} \leq \mvec{x}_{\top}^* + \mvec{p}_{\perp} + \mtx{V}\mvec{w} \leq \mvec{1}]. 
\end{aligned}
\end{equation}

According to this derivation, all the randomness within the lower bound \eqref{eq:prob_Ay^2*_is_between_-1_and_1} is contained in the random vector $\mvec{w}$. The next theorem, derived in \Cref{app:pdf_w}, provides the probability density function of this random vector.

\begin{rem} 
Suppose that \Cref{ass:AREGO_fun_eff_dim} holds and recall \eqref{def:G}. If there exists $\mvec{x}^*\in G$  such that $\mvec{x}_{\top}^* = \mvec{p}_{\top}$, where the subscript represents the respective Euclidean projections on the effective subspace, then $f(\mvec{p}) = f(\mvec{p}_{\top}+\mvec{p}_{\perp}) = f(\mvec{x}_{\top}^* + \mvec{p}_{\perp}) = f^*$, where $\mvec{p}_{\perp}$ is the Euclidean projection of $\mvec{p}$ on the constant subspace $\mathcal{T}^{\perp}$ of the objective function. Thus  $\mvec{p}\in G$  so that, for any embedding $\mtx{A}$,  \eqref{eq: AREGO} is successful with the trivial solution $\mvec{y}^* = \mvec{0}$. Therefore, in our next result, without loss of generality, we  make the assumption $\mvec{x}_{\top}^* \neq \mvec{p}_{\top}$.
\end{rem}

\begin{theorem}[The p.d.f.~of $\mvec{w}$] \label{thm:pdf_of_w}
    Suppose that \Cref{ass:AREGO_fun_eff_dim} holds. 
    Let $\mvec{x}^*$ be a(ny) global minimizer of \eqref{eq: GO}, $\mvec{p}\in \mathcal{X}$, a given vector, and $\mtx{A}$, a $D \times d$ Gaussian matrix. Assume that $\mvec{p}_\top \neq \mvec{x}_\top^*$, where the subscript represents the Euclidean projection on the effective subspace. The random vector $\mvec{w}$ defined in \eqref{w-def} follows a  $(D-d_e)$-dimensional $t$-distribution with parameters $d-d_e+1$ and $\frac{\| \mvec{x}_{\top}^* - \mvec{p}_{\top}\|^2}{d-d_e+1}\mtx{I}$, and with p.d.f. $g(\mvec{\bar{w}})$ given by
    \begin{equation}\label{eq:pdf_of_w}
        g(\mvec{\bar{w}}) = \frac{1}{(\sqrt{\pi} \| \mvec{x}_{\top}^* - \mvec{p}_{\top}\|)^m} \left[ \frac{\Gamma(\frac{m+n}{2})}{\Gamma(\frac{n}{2})} \right] \left( 1 + \frac{\mvec{\bar{w}}^T\mvec{\bar{w}}}{\| \mvec{x}_{\top}^* - \mvec{p}_{\top}\|^2} \right)^{-(m+n)/2},
    \end{equation}
    where $m = D-d_e$ and $n = d-d_e+1$.
\end{theorem}
\begin{proof} See \Cref{app:pdf_w}.
    
\end{proof}

The remainder of this section aims at answering the two following questions: \emph{Is the probability of success of \eqref{eq: AREGO} positive for any $\mvec{p}$? If yes, can we derive a positive lower bound on the probability of success of \eqref{eq: AREGO} that does not depend on $\mvec{p}$?}  We  show that both questions can be answered positively, and use this extensively in our global convergence analysis in \Cref{sec: Convergence}.

\subsection{Positive probability of success of the reduced problem \texorpdfstring{\eqref{eq: AREGO}}{(RPX)} } \label{sec: proba_success_pos}
We first summarize the above analysis in the following corollary.
\begin{corollary}\label{thm:prob[success]>bound}
   Suppose that \Cref{ass:AREGO_fun_eff_dim} holds. 
    Let $\mvec{x}^*$ be a(ny) global minimizer of \eqref{eq: GO}, $\mvec{p}\in \mathcal{X}$, a given vector, and $\mtx{A}$, a $D \times d$ Gaussian matrix. Assume that $\mvec{p}_\top \neq \mvec{x}_\top^*$, where the subscript represents the Euclidean projection on the effective subspace. Then
    \begin{equation} \label{eq:prob[success]=>prob[-1<x+Vw<1]}
        \prob[\text{\eqref{eq: AREGO} is successful}\,] \geq \prob(-\mvec{1} \leq \mvec{x}_{\top}^* + \mvec{p}_{\perp} + \mtx{V}\mvec{w} \leq \mvec{1}),
    \end{equation}
    where $\mvec{w}$ is a random vector that follows a  $(D-d_e)$-dimensional $t$-distribution with parameters $d-d_e+1$ and $\frac{\| \mvec{x}_{\top}^* - \mvec{p}_{\top}\|^2}{d-d_e+1}\mtx{I}$.
\end{corollary}
\begin{proof}
The result follows from derivations \eqref{eq:prob_there_exists_y}--\eqref{eq:prob_Ay^2*_is_between_-1_and_1} and \Cref{thm:pdf_of_w}.
\end{proof}


We need the following additional assumption.

\begin{assump}\label{ass:G^*_is_nondegenerate}
    Assume that \Cref{ass:AREGO_fun_eff_dim} holds,  and that there is  a set $G^*$ defined in \Cref{def:AREGO_choice_of_x^*} that is non-degenerate according to  \Cref{def:G^*_degenerate}. 
\end{assump}

\begin{theorem}\label{thm: positiveprobasuccess}
Suppose that \Cref{ass:G^*_is_nondegenerate} holds, and let $\mtx{A}$ be a $D \times d$ Gaussian matrix. Then, for any $\mvec{p} \in \mathcal{X}$, 
\begin{equation}\label{eq:prob_RPX_succ>0}
    \prob[\text{\eqref{eq: AREGO} is successful}\,] > 0.
\end{equation}
\end{theorem} 
\begin{proof}
We consider two cases,  $\mvec{p}\in G$ and $\mvec{p} \in \mathcal{X}\setminus G$. Firstly, assume that $\mvec{p} \in G$. Then, $ \prob[\text{\eqref{eq: AREGO} is successful}\,] = 1$ since taking $\mvec{y} = \mvec{0}$ in \eqref{eq: AREGO} yields $f(\mvec{p}) = f^*$.

Assume now that $\mvec{p}\in \mathcal{X}\setminus G$. \Cref{ass:G^*_is_nondegenerate} implies that there exists a global minimizer $\mvec{x}^*$ and associated $G^*$ for which $\vol(\bar{G}^*) > 0$,
where $G^*$ and $\bar{G}^*$ are defined in \Cref{def:AREGO_choice_of_x^*} and \eqref{eq: barGstar}, respectively.
Using \eqref{eq:prob[success]=>prob[-1<x+Vw<1]} with this particular $\mvec{x}^*$ and noting that $\mvec{p}_{\perp} = \mtx{V}\mtx{V}^T\mvec{p}$ gives us
\begin{equation}
    \begin{aligned} 
        \prob[\text{\eqref{eq: AREGO} is successful}] & \geq \prob[-\mvec{1} \leq \mvec{x}_{\top}^* + \mtx{V}(\mtx{V}^T\mvec{p} + \mvec{w}) \leq \mvec{1}] \\
        & = \prob[\mtx{V}^T\mvec{p} + \mvec{w} \in \{\mvec{g} \in \mathbb{R}^{D-d_e} : -\mvec{1} \leq \mvec{x}_{\top}^* + \mtx{V}\mvec{g} \leq \mvec{1} \}] \\
        & = \prob[\mtx{V}^T\mvec{p} + \mvec{w} \in \bar{G}^*] \\
        & = \prob[\mvec{w} \in -\mtx{V}^T \mvec{p} + \bar{G}^*] \\
        & = \int_{-\mtx{V}^T \mvec{p} + \bar{G}^*} g(\mvec{\bar{w}}) d\mvec{\bar{w}}, \label{eq:in_to_bound}
    \end{aligned}
\end{equation}
where $g(\mvec{\bar{w}})$ is the p.d.f. of $\mvec{w}$ given in \eqref{eq:pdf_of_w}. The latter integral is positive since $g(\mvec{\bar{w}}) > 0$ for any $\mvec{\bar{w}} \in \mathbb{R}^{D-d_e}$ and since $\vol(-\mtx{V}^T\mvec{p}+\bar{G}^*) = \vol(\bar{G}^*) > 0$ (invariance of volumes under translations) by \Cref{ass:G^*_is_nondegenerate}. 
\end{proof}

Note that the proof of \Cref{thm: positiveprobasuccess}  illustrates that the success probability of 
\eqref{eq: AREGO}, though positive, depends on the choice of $\mvec{p}$\footnote{When $\|\mvec{x}_\top^* - \mvec{p}_\top\| \rightarrow 0$, the multivariate $t$-distribution in \Cref{thm:prob[success]>bound} becomes degenerate. 
Thus it is challenging to derive a lower bound on the integral \eqref{eq:in_to_bound} that is uniformly bounded away from zero with respect to $\mvec{p}$.}.
Next, under additional problem assumptions, we derive  lower bounds on the success probability of \eqref{eq: AREGO} that are independent of $\mvec{p}$ and/or quantifiable.

\subsection{Quantifying the success probability of \texorpdfstring{\eqref{eq: AREGO}}{(RPX)} in the special case of coordinate-aligned effective subspace} \label{sec: aligned}

Provided the effective subspace $\mathcal{T}$ is aligned with coordinate axes and without loss of generality, we can write the orthonormal matrices $\mtx{U}$ and $\mtx{V}$, whose columns span $\mathcal{T}$ and $\mathcal{T}^\perp$, as $\mtx{U} = [\mtx{I}_{d_e} \; \mtx{0}]^T$ and $\mtx{V} = [\mtx{0} \; \mtx{I}_{D-d_e}]^T$.

\begin{theorem}\label{thm:RP_succ_special_case}
Let \Cref{ass:AREGO_fun_eff_dim} hold with $\mtx{U} = [\mtx{I}_{d_e} \; \mtx{0}]^T$ and $\mtx{V} = [\mtx{0} \; \mtx{I}_{D-d_e}]^T$. 
  Let $\mvec{x}^*$ be a(ny) global minimizer of \eqref{eq: GO}, $\mvec{p}\in \mathcal{X}$, a given vector, and $\mtx{A}$, a $D \times d$ Gaussian matrix. Assume that $\mvec{p}_\top \neq \mvec{x}_\top^*$, where the subscript represents the Euclidean projection on the effective subspace.  Then
\begin{equation} \label{eq:prob[success]=>prob[-1-p<w<1-p]}
    \prob[\text{\eqref{eq: AREGO} is successful}\,] \geq \prob[-\mvec{1}- \mvec{p}_{d_e+1:D} \leq \mvec{w} \leq \mvec{1}- \mvec{p}_{d_e+1:D}],
\end{equation}
where $\mvec{w}$ is a random vector that follows a  $(D-d_e)$-dimensional $t$-distribution with parameters $d-d_e+1$ and $\frac{\| \mvec{x}_{\top}^* - \mvec{p}_{\top}\|^2}{d-d_e+1}\mtx{I}$.
\end{theorem}
\begin{proof}
For $\mvec{x^*}\in G^*$, we have
\begin{equation}\label{eq:x_top^*=x_1:d_e}
    \mvec{x}_{\top}^* = \mtx{U} \mtx{U}^T \mvec{x}^* = \begin{pmatrix}
    \mtx{I} & \mtx{0} \\
    \mtx{0} & \mtx{0}
    \end{pmatrix} \mvec{x}^* = \begin{pmatrix} \mvec{x}^*_{1:d_e} \\ \mvec{0} \end{pmatrix}.
\end{equation}
Furthermore,
\begin{equation*}
    \mvec{p}_{\perp} = \mtx{V} \mtx{V}^T \mvec{p}  = \begin{pmatrix}
    \mtx{0} & \mtx{0} \\
    \mtx{0} & \mtx{I}
    \end{pmatrix} \mvec{p} = \begin{pmatrix} \mvec{0} \\ \mvec{p}_{d_e+1:D} \end{pmatrix}.
\end{equation*}
Note that $\mvec{x}^* \in [-1,1]^{D}$ implies that $\mvec{x}^*_{1:d_e} \in [-1,1]^{d_e}$.
\Cref{thm:prob[success]>bound} then yields
\begin{alignat*}{2}
    \prob[\text{\eqref{eq: AREGO} is successful}] & \geq \prob(-\mvec{1} \leq \mvec{x}_{\top}^* + \mvec{p}_{\perp}+ \mtx{V} \mvec{w} \leq \mvec{1}) \\
    & = \prob\left[\begin{pmatrix} -\mvec{1} \\ -\mvec{1} \end{pmatrix} \leq \begin{pmatrix} \mvec{x}^*_{1:d_e} \\ \mvec{0} \end{pmatrix} + 
    \begin{pmatrix} \mvec{0} \\ \mvec{p}_{d_e+1:D} \end{pmatrix} +
    \begin{pmatrix} \mtx{0} \\ \mtx{I} \end{pmatrix} \mvec{w} \leq \begin{pmatrix} \mvec{1} \\ \mvec{1} \end{pmatrix}\right]  \\
    (\text{since $\mvec{x}^*_{1:d_e} \in [-1,1]^{d_e}$}) & = \prob[-\mvec{1} \leq \mvec{p}_{d_e+1:D} + \mvec{w} \leq \mvec{1}], 
\end{alignat*}
which immediately gives \eqref{eq:prob[success]=>prob[-1-p<w<1-p]}.
\end{proof}
Note that the right-hand side of \eqref{eq:prob[success]=>prob[-1-p<w<1-p]} can be written as the integral of the p.d.f.~of $\mvec{w}$ over the hyperrectangular region 
$-\mvec{1} - \mvec{p}_{d_e+1:D} \leq  \mvec{w} \leq \mvec{1} - \mvec{p}_{d_e+1:D} $. Instead of directly computing this integral, we analyse its asymptotic behaviour for large $D$, assuming that $d_e$ and $d$ are fixed. We obtain the following main result, with its proof provided in \Cref{app:proofs_of_corollaries}.

\begin{theorem}\label{cor:prob[RPXis_succ]=Omega_general_p}
Let \Cref{ass:AREGO_fun_eff_dim} hold with  $\mtx{U} = [\mtx{I}_{d_e} \; \mtx{0}]^T$ and $\mtx{V} = [\mtx{0} \; \mtx{I}_{D-d_e}]^T$. Let $d_e$ and $d$ be fixed, and let
$\mtx{A}$ be a $D \times d$ Gaussian matrix. For all $\mvec{p} \in \mathcal{X}$, we have
\begin{equation} \label{eq:prob[success]=>tau_in_corollary}
    \prob[\text{\eqref{eq: AREGO} is successful}\,] \geq \tau > 0,
\end{equation}
where $\tau$ satisfies
    \begin{equation}\label{eq:asym_exp_general_p}
        \text{$\tau = \Theta\left(\frac{\log(D-d_e+1)^{\frac{d-1}{2}}}{2^{D-d_e} \cdot(D-d_e+1)^{d_e}}\right)$ as $D \rightarrow \infty$,}
    \end{equation}
    and the constants in $\Theta (\cdot)$ depend only on $d_e$ and $d$.
\end{theorem}
\begin{proof}
See \Cref{app:proofs_of_corollaries}. 
\end{proof}
    
The next result shows that, in the particular case when $\mvec{p} = \mvec{0}$, the center of the full-dimensional domain $\mathcal{X}$, the probability of success decreases at worst algebraically\footnote{This simplification is due to the fact that when $\mvec{p}=\mvec{0}$, the factor $2^{D-d_e}$ in the denominator of \eqref{eq:asym_exp_general_p} disappears.} with the ambient dimension $D$. 
\begin{theorem}\label{cor:prob[RPXis_succ]=Omega_p=0}
Let \Cref{ass:AREGO_fun_eff_dim} hold with  $\mtx{U} = [\mtx{I}_{d_e} \; \mtx{0}]^T$ and $\mtx{V} = [\mtx{0} \; \mtx{I}_{D-d_e}]^T$. Let $d_e$ and $d$ be fixed, and let
$\mtx{A}$ be a $D \times d$ Gaussian matrix. Let $\mvec{p} = \mvec{0}$. Then 
\begin{equation} \label{eq:prob[success]=>tau_in_corollary0}
    \prob[\text{\eqref{eq: AREGO} is successful}\,] \geq \tau_{\mvec{0}} > 0,
\end{equation}
where
    \begin{equation}\label{eq:asym_exp_p=0}
        \text{$\tau_{\mvec{0}} = \Theta\left(\frac{\log(D-d_e+1)^{\frac{d-1}{2}}}{(D-d_e+1)^{d_e}}\right)$ as $D \rightarrow \infty$,}
    \end{equation}
    and where the constants in $\Theta (\cdot)$ depend only on $d_e$ and $d$.
\end{theorem}
\begin{proof}
See \Cref{app:proofs_of_corollaries}. 
\end{proof}

\begin{rem}
    Unlike \Cref{thm: positiveprobasuccess}, the above result does not require \Cref{ass:G^*_is_nondegenerate}. In this specific case, as the effective subspace is aligned with the coordinate axes, \Cref{ass:G^*_is_nondegenerate} is satisfied. The latter follows from $\bar G^* = \{ \mvec{g} \in \mathbb{R}^{D-d_e} : -\mvec{1} \leq \mvec{x}_\top^* + \mtx{V} \mvec{g} \leq \mvec{1}\} = \{ \mvec{g} \in [-1,1]^{D-d_e}\}$, as $\mtx{V} = [\mtx{0} \; \mtx{I}_{D-d_e}]^T$ and the last $D-d_e$ components of the vector $\mvec{x}_\top^*$ are zero; see the proof of \Cref{thm:RP_succ_special_case}. 
\end{rem}
\begin{rem}
The lower bounds on the probability of success of the reduced problem derived here and in the previous section are reasonably tight. We note for example that 
\eqref{eq:prob[success]=>prob[-1-p<w<1-p]} holds with equality if $d=d_e$ and $G = G^*$. Our numerical experiments 
in \Cref{sec: Numerics} also clearly illustrate that the success probability decreases with growing problem dimension $D$. 
\end{rem}

\begin{rem}
  
    Our particular choice of asymptotic framework here is due to its practicality as well as to the ready-at-hand analysis of a similar integral to \eqref{def:I(p,Delta)} in \cite{Wong2001}. The scenario ($d_e$ and $d$ fixed, $D$ large) is a familiar one in practice, where commonly, $d_e$ is small compared to $D$, and $d$ is limited by computational resources available to solve the reduced subproblem. 
    Other asymptotic frameworks that could be considered in the future are $d_e = O(1)$, $d = O(\log(D))$ or $d_e = O(1)$, $d = \beta D$ where $\beta$ is fixed. For more details on how to obtain asymptotic expansions similar to \eqref{eq:asym_exp_general_p} and \eqref{eq:asym_exp_p=0} for such choices of $d_e$ and $d$, refer to \cite{Temme2014, Wong2001}.
\end{rem}

\subsection{Uniformly positive lower bound on the success probability  of \texorpdfstring{\eqref{eq: AREGO}}{(RPX)}  in the general case}

As mentioned in the last paragraph of \Cref{sec: proba_success_pos}, it is difficult to derive a uniformly positive lower bound on the probability of success of \eqref{eq: AREGO} that does not depend on $\mvec{p}$. However, assuming Lipschitz continuity of the objective function, we are able to achieve such a guarantee for  \eqref{eq: AREGO} to be {\it approximately} successful, a weaker notion that is  defined as follows.

\begin{definition} \label{def: eps_successful_AREGO}
	For a(ny) $\epsilon>0$, we say that \eqref{eq: AREGO} is $\epsilon$-\textit{successful} if there exists $\mvec{y}^* \in \mathbb{R}^d$ such that $f(\mtx{A}\mvec{y}^*+\mvec{p}) \leq f^* + \epsilon$ and $\mtx{A}\mvec{y}^* + \mvec{p} \in \mathcal{X}$.
\end{definition}

\noindent Let 
\begin{equation} \label{eq: G_epsilon}
    G_\epsilon := \{ \mvec{x} \in \mathcal{X} : f(\mvec{x}) \leq f^* + \epsilon\}
\end{equation}
be the set of feasible $\epsilon$-minimizers. The reduced problem \eqref{eq: AREGO} is thus $\epsilon$-successful if it contains a feasible $\epsilon$-minimizer.

\begin{assump}\label{ass:f_is_Lipschitz}
   The objective function $f : \mathbb{R}^D \rightarrow \mathbb{R}$ is Lipschitz continuous with Lipschitz constant $L$, that is, $|f(\mvec{x}) - f(\mvec{y})| \leq L\| \mvec{x} - \mvec{y} \|_2$ for all $\mvec{x}$ and $\mvec{y}$ in $\mathbb{R}^D$.
\end{assump}

The next theorem shows that the probability that $\eqref{eq: AREGO}$ is $\epsilon$-successful is uniformly bounded away from zero for all $\mvec{p} \in \mathcal{X}$.

\begin{theorem} \label{thm:prob[I_S=1]>tau} Suppose that \Cref{ass:G^*_is_nondegenerate} and  \Cref{ass:f_is_Lipschitz} hold, and let $\mtx{A}$ be a $D \times d$ Gaussian matrix and $\epsilon>0$, an accuracy tolerance.  Then there exists a constant $\tau_{\epsilon} > 0$ such that, for all $\mvec{p} \in \mathcal{X}$,
\begin{equation} \label{eq:prob[success]=>tau}
    \prob[\eqref{eq: AREGO} \ \text{is $\epsilon$-successful}] \geq \tau_{\epsilon}.
\end{equation}
\end{theorem}
\begin{proof}
\Cref{ass:G^*_is_nondegenerate} implies that there exists a global minimizer $\mvec{x}^* \in \mathcal{X}$, with corresponding sets $G^*$ (\Cref{def:AREGO_choice_of_x^*}) and $\bar{G}^*$ in \eqref{eq: barGstar}
such that $\vol(\bar{G}^*) > 0$. Let $N_{\eta}(G^*):= \{ \mvec{x} \in \mathcal{X}: \| \mvec{x}_{\top}^* - \mtx{U}\mtx{U}^T \mvec{x}\|_2 \leq \eta \}$ be a neighbourhood of $G^*$ in $\mathcal{X}$, for some $\eta > 0$, where as usual, $\mvec{x}_{\top}^* = \mtx{U} \mtx{U}^T \mvec{x}^*$ is the Euclidean projection of $\mvec{x}^*$ on the effective subspace.

Firstly,  assume that $\mvec{p} \in N_{\epsilon/L}(G^*)$. Then, 
$\| \mvec{x}_{\top}^* - \mvec{p}_{\top} \| \leq \epsilon/L$, and by \Cref{ass:f_is_Lipschitz}, $|f(\mvec{p}) - f^*| = |f(\mvec{p}_{\top}) - f(\mvec{x}_{\top}^*)| \leq L\| \mvec{x}_{\top}^* - \mvec{p}_{\top} \| \leq \epsilon$. Thus $\mvec{p} \in G_{\epsilon}$ and, hence, $\prob[\eqref{eq: AREGO} \ \text{is $\epsilon$-successful}] = 1$. 

Otherwise, $\mvec{p}\in \mathcal{X}\setminus N_{\epsilon/L}(G^*)$. Using the proof of \Cref{thm: positiveprobasuccess}, we have
\begin{equation}\label{eq: epssuccess<success}
    \begin{aligned} 
        \prob[\text{$\eqref{eq: AREGO}$ is $\epsilon$-successful}] \geq \prob[\text{\eqref{eq: AREGO} is successful}] & \geq \int_{-\mtx{V}^T \mvec{p} + \bar{G}^*} g(\mvec{\bar{w}}) d\mvec{\bar{w}},
    \end{aligned}
\end{equation}
where $g(\mvec{\bar{w}})$ is the p.d.f. of $\mvec{w}$ given by \eqref{eq:pdf_of_w}, and where the first inequality is due to the fact that  \eqref{eq: AREGO} being successful implies that \eqref{eq: AREGO} is $\epsilon$-successful (by letting $\epsilon := 0$ in \Cref{def: eps_successful_AREGO}).
To prove \eqref{eq:prob[success]=>tau}, it is thus sufficient to 
lower bound $g(\mvec{\bar{w}})$  by a positive constant, independent of $\mvec{p}$. Since, $\mvec{p} \notin N_{\epsilon/L}(G^*)$, we have 
\begin{equation}\label{eq:norm_x_top-p_top<2sqrt_of_D}
    \frac{\epsilon}{L} < \| \mvec{x}^*_{\top} - \mvec{p}_{\top}\|_2 = \| \mtx{U}\mtx{U}^T(\mvec{x}^* - \mvec{p}) \|_2 \leq \|\mtx{U}\mtx{U}^T\|_2\cdot \|\mvec{x}^* - \mvec{p}\|_2 \leq 2\sqrt{D}, 
\end{equation}
where the last inequality follows from $\|\mtx{U}\mtx{U}^T\|_2 = 1$, since $\mtx{U}$ has orthonormal columns, and from $-\mvec{2} \leq \mvec{x}^* - \mvec{p} \leq \mvec{2}$ since $\mvec{x}^*,\mvec{p} \in [-1,1]^D$. Furthermore, note that, for any $\mvec{\bar{w}} \in -\mtx{V}^T \mvec{p} + \bar{G}^*$, we have
$$ -\mvec{1}-\mvec{x}_{\top}^* - \mvec{p}_{\perp} \leq \mtx{V}\mvec{\bar{w}} \leq \mvec{1} - \mvec{x}_{\top}^* - \mvec{p}_{\perp}, $$
and, hence,
\begin{equation*}
    \begin{aligned}
       \| \mtx{V}\mvec{\bar{w}} \|_{\infty}  & \leq \max(\|-\mvec{1}-\mvec{x}_{\top}^* - \mvec{p}_{\perp}\|_{\infty},\|\mvec{1}-\mvec{x}_{\top}^* - \mvec{p}_{\perp}\|_{\infty}) \\
       & \leq \|\mvec{1}\|_{\infty} + \| \mvec{x}_{\top}^* \|_{\infty} + \|\mvec{p}_{\perp}\|_{\infty} \\
       & \leq 1 + \| \mvec{x}_{\top}^* \|_{2} + \|\mvec{p}_{\perp}\|_{2} \\
       & = 1 + \| \mtx{U} \mtx{U}^T \mvec{x}^* \|_{2} + \|\mtx{V} \mtx{V}^T\mvec{p}\|_{2} \\
       & \leq 1 + \| \mtx{U} \mtx{U}^T \|_2 \cdot \|\mvec{x}^* \|_{2} + \|\mtx{V} \mtx{V}^T\|_2 \cdot \|\mvec{p}\|_{2} \\
       & \leq 1+2\sqrt{D},
    \end{aligned}
\end{equation*}
where the last inequality follows from $\|\mtx{U}\mtx{U}^T\|_2 = 1$ and $\|\mtx{V}\mtx{V}^T\|_2 = 1$ (as $\mtx{U}$ and $\mtx{V}$ are orthonormal) and from $\mvec{x}^*, \mvec{p} \in [-1,1]^D$. Thus,
\begin{equation}\label{eq:norm_w<3D}
    \|\mvec{\bar{w}}\|_2 = \| \mtx{V}\mvec{\bar{w}} \|_2 \leq \sqrt{D}\| \mtx{V}\mvec{\bar{w}} \|_{\infty} \leq \sqrt{D} (1+2\sqrt{D}) \leq 3D.
\end{equation}
By combining \eqref{eq:pdf_of_w},  \eqref{eq:norm_x_top-p_top<2sqrt_of_D} and \eqref{eq:norm_w<3D}, we finally obtain
\begin{equation*}
\begin{aligned}
    \int_{-\mtx{V}^T \mvec{p} + \bar{G}^*} g(\mvec{\bar{w}}) d\mvec{\bar{w}} & = C(m,n) \int_{-\mtx{V}^T \mvec{p} + \bar{G}^*}  \frac{1}{\| \mvec{x}_{\top}^* - \mvec{p}_{\top}\|^m} \left( 1 + \frac{\|\mvec{\bar{w}}\|^2}{\| \mvec{x}_{\top}^* - \mvec{p}_{\top}\|^2} \right)^{-(m+n)/2} d\mvec{\bar{w}} \\
    & > C(m,n) (2\sqrt{D})^{-m} (1+9D^2L^2/\epsilon^2)^{-(m+n)/2}  \int_{-\mtx{V}^T \mvec{p} + \bar{G}^*} d \mvec{\bar{w}} \\
    & = C(m,n) (2\sqrt{D})^{-m} (1+9D^2L^2/\epsilon^2)^{-(m+n)/2} \vol(-\mtx{V}^T \mvec{p} + \bar{G}^*) \\
    & = C(m,n) (2\sqrt{D})^{-m} (1+9D^2L^2/\epsilon^2)^{-(m+n)/2} \vol(\bar{G}^*),
\end{aligned}
\end{equation*}
where $C(m,n) = \Gamma((m+n)/2)/(\pi^{m/2}\Gamma(n/2))$ and where in the last equality we used the fact $\vol(-\mtx{V}^T\mvec{p}+\bar{G}^*) = \vol(\bar{G}^*)$ for any $\mvec{p} \in \mathbb{R}^D$ (invariance of volumes under translations). The result follows from the assumption that $\vol(\bar{G}^*)~>~0$.
\end{proof}

%% file: convergence_proof.tex

In the case of random embeddings for unconstrained global optimization \cite{Cartis2020},
 the success probability of the reduced problem is independent  of the ambient dimension \cite{Cartis2020}. However, in the constrained case of problem \eqref{eq: GO}, the analysis in \Cref{sec: estim_success} shows that the probability of success of the reduced problem \eqref{eq: AREGO} decreases with $D$. 
 It is thus imperative in any algorithm that uses feasible random embeddings in order to solve \eqref{eq: GO} to allow multiple such subspaces to be explored, and it is practically important to find out what are efficient and theoretically-sound ways to choose these subspaces iteratively. This is the aim of our generic and flexible algorithmic framework, X-REGO (\Cref{alg: AREGO}). Furthermore, as an additional level of generality and practicality, we allow the reduced, random subproblem to be solved stochastically, so that a sufficiently accurate global solution of this problem is only guaranteed with a certain probability. This covers the obvious case when a (convergent) stochastic global optimization algorithm would be employed to solve the reduced subproblem, but also when a deterministic global solver is used but may sometimes fail to find the required solution due to a limited computational budget, processor failure and so on. 
 
In X-REGO, for $k\geq 1$, the $k$th embedding is determined by a realization $\tilde{\mtx{A}}^k = \mtx{A}^k(\mvec{\omega}^k)$ of the random Gaussian matrix $\mtx{A}^k$, and it is drawn at the point $\tilde{\mvec{p}}^{k-1} = \mvec{p}^{k-1}(\mvec{\omega}^{k-1}) \in \mathcal{X}$, a realization of the  random variable $\mvec{p}^{k-1}$ (which, without loss of generality, includes the case of deterministic choices 
by writing $\mvec{p}^{k-1}$ as a random variable with support equal to a singleton). 

\begin{algorithm}[H]
	\caption{$\mathcal{X}$-Random Embeddings for Global Optimization (X-REGO) applied to~\eqref{eq: GO}}
	\label{alg: AREGO}
	\begin{algorithmic}[1]
		\State Initialize $d$ and $\mvec{p}^0 \in \mathcal{X}$
		\For{\text{$k \geq 1$ until termination}} \label{termination}
		\State Draw $\tilde{\mtx{A}}^k$, a realization of the $D\times d$ Gaussian matrix $\mtx{A}$
		\State Calculate $\tilde{\mvec{y}}^k$ by solving approximately and possibly, probabilistically,
		\begin{equation}\label{prob: AREGO_subproblem_re}
		\tag{$\widetilde{\text{RP}\mathcal{X}^k}$}
		\begin{aligned} 
		\tilde{f}^k_{min} = \min_{\mvec{y}\in\mathbb{R}^d} & \; f(\tilde{\mtx{A}}^k \mvec{y} + \tilde{\mvec{p}}^{k-1}) \\
		\text{subject to} & \; \tilde{\mtx{A}}^k \mvec{y} + \tilde{\mvec{p}}^{k-1} \in \mathcal{X} 
		\end{aligned}
		\end{equation}
		\State Let
		\begin{equation} \label{eq: xck}
		\tilde{\mvec{x}}^k := \tilde{\mtx{A}}^k \tilde{\mvec{y}}^k + \tilde{\mvec{p}}^{k-1}
		\end{equation}
		\State Choose (deterministically or randomly) $\tilde{\mvec{p}}^k \in \mathcal{X}$
		\EndFor
	\end{algorithmic}
\end{algorithm}

X-REGO can be seen as a stochastic process, so that in addition to $\tilde{\mvec{p}}^k$ and 
$\tilde{\mtx{A}}^k$, each algorithm realization provides sequences $\tilde{\mvec{x}}^k = \mvec{x}^k(\mvec{\omega}^k)$, $\tilde{\mvec{y}}^k = \mvec{y}^k(\mvec{\omega}^k)$ and $\tilde{f}_{min}^k = f_{min}^k(\mvec{\omega}^k)$, for $k \geq 1$, that are realizations of the random variables $\mvec{x}^k$, $\mvec{y}^k$ and $f_{min}^k$, respectively. Each iteration of X-REGO solves -- approximately and possibly, with a certain probability -- a realization \eqref{prob: AREGO_subproblem_re} of the random problem 
\begin{equation} \label{prob: AREGO_subproblem}
\tag{$\text{RP}\mathcal{X}^k$}
    	\begin{aligned} 
		f^k_{min} = \min_{\mvec{y}} & \; f(\mtx{A}^k \mvec{y} + \mvec{p}^{k-1}) \\
		\text{subject to} & \; \mtx{A}^k \mvec{y} + \mvec{p}^{k-1} \in \mathcal{X}. 
		\end{aligned}
\end{equation}
To calculate $\tilde{\mvec{y}}^k$, \eqref{prob: AREGO_subproblem_re} may be solved to some required accuracy using a deterministic global optimization algorithm that is allowed to fail with a certain probability; or employing a stochastic algorithm, so that 
$\tilde{\mvec{y}}^k$ is only guaranteed to be an approximate global minimizer of  \eqref{prob: AREGO_subproblem_re}  (at least) with a certain probability. 

 Several variants of X-REGO can be obtained by specific choices of the random variable $\mvec{p}^k$ (assumed throughout the paper to have support contained in $\mathcal{X}$). A first possibility consists in simply defining $\mvec{p}^k$ as a random variable with support $\{\mvec{0}\}$, so that $\tilde{\mvec{p}}^{k} = \mvec{0}$ for all $k$. It is also possible to preserve the progress achieved so far by defining $\mvec{p}^k = \mvec{x}_{opt}^k$, where 
 \begin{equation} \label{eq: xoptk}
     \mvec{x}_{opt}^k := \arg \min \{ f(\mvec{x}^1), f(\mvec{x}^2), \dots, f(\mvec{x}^k)\},
 \end{equation}
 the random variable corresponding to the best point found over the $k$ first embeddings. We compare numerically several choices of $\mvec{p}$ on benchmark functions in \Cref{sec: Numerics}. 

The termination in \Cref{termination} could be set to  a given maximum number of embeddings, or could check that no significant progress in decreasing the objective function has been achieved over the last few embeddings, compared to the value $f(\tilde{\mvec{x}}^k_{opt})$. For generality, we leave it unspecified for now.

\subsection{Global convergence of the X-REGO algorithm to the set of global \texorpdfstring{$\epsilon$}{eps}-minimizers}
\label{sec:globalX-REGO}

For a(ny) given tolerance $\epsilon >0$, let $G_\epsilon$ be the set of approximate global minimizers of \eqref{eq: GO} defined in \eqref{eq: G_epsilon}.  We show that $\mvec{x}_{opt}^k$ 
in \eqref{eq: xoptk} converges to $G_{\epsilon}$ almost surely as $k \rightarrow \infty$
(see \Cref{thm: glconv}). 


Intuitively, our proof relies on the fact that any vector $\tilde{\mvec{x}}^k$ defined in \eqref{eq: xck} belongs to $G_\epsilon$ if the following two conditions hold simultaneously: 
(a) the reduced problem \eqref{prob: AREGO_subproblem} is $(\epsilon - \lambda)$-successful in the sense of \Cref{def: eps_successful_AREGO}\footnote{
The reader may expect us to simply require that \eqref{prob: AREGO_subproblem} is $\epsilon$-successful. However, in order to ensure convergence of X-REGO to the set of $\epsilon$-minimizers, we need to be slightly more demanding on the success requirements for \eqref{prob: AREGO_subproblem} so that we  allow inexact solutions (up to accuracy $\lambda$) of the reduced problem \eqref{prob: AREGO_subproblem_re}.}, namely, 
\begin{equation}\label{eq:succ-red}
f_{min}^k \leq f^* + \epsilon-\lambda;
\end{equation}
(b) 
the reduced problem \eqref{prob: AREGO_subproblem_re} is solved (by a deterministic/stochastic algorithm) to an accuracy $\lambda\in (0,\epsilon)$ in the objective function value, namely,
\begin{equation}\label{eq:approxf}
f(\mtx{A}^k \mvec{y}^k + \mvec{p}^{k-1}) \leq  f_{min}^k + \lambda
\end{equation}
holds (at least) with a certain probability.
We introduce  two additional random variables that capture the conditions in (a) and (b) above,
\begin{align}
        R^k  &= \mathds{1}\{\text{\eqref{prob: AREGO_subproblem} is 
        $(\epsilon-\lambda)$-successful in the sense of \eqref{eq:succ-red}}\}, \label{eq: Rk} \\ 
        S^k  &= \mathds{1}\{\text{\eqref{prob: AREGO_subproblem} is solved to accuracy $\lambda$ in the sense of \eqref{eq:approxf}}\}, \label{eq: Sk}
\end{align}
where $\mathds{1}$ is the usual indicator function for an event. 


Let $\mathcal{F}^k = \sigma(\mtx{A}^1, \dots, \mtx{A}^k, \mvec{y}^1, \dots, \mvec{y}^k, \mvec{p}^0, \dots, \mvec{p}^k)$ be the $\sigma$-algebra generated by the random variables $\mtx{A}^1, \dots, \mtx{A}^k, \mvec{y}^1, \dots, \mvec{y}^k, \mvec{p}^0, \dots, \mvec{p}^k$ (a mathematical concept that represents the history of the  X-REGO algorithm as well as its randomness
until the $k$th embedding)\footnote{A similar setup for random iterates of probabilistic models can be found in \cite{Bandeira2014, Cartis2018}.}, with  $\mathcal{F}^0 = \sigma(\mvec{p}^0)$. 
We also construct an `intermediate' $\sigma$-algebra, namely,
$$\mathcal{F}^{k-1/2} = \sigma(\mtx{A}^1, \dots, \mtx{A}^{k-1}, \mtx{A}^{k}, \mvec{y}^1, \dots, \mvec{y}^{k-1}, \mvec{p}^0, \dots, \mvec{p}^{k-1}), $$
with  $\mathcal{F}^{1/2} = \sigma(\mvec{p}^0, \mtx{A}^{1})$.
Note that $\mvec{x}^k$, $R^k$ and $S^k$ are $\mathcal{F}^{k}$-measurable\footnote{It would be possible to restrict the definition of the $\sigma$-algebra $\mathcal{F}^k$
so that it contains strictly the randomness of the embeddings $\mtx{A}^i$ and $\mvec{p}^i$ for $i\leq k$; then we would need to assume that $\mvec{y}^k$  
is $\mathcal{F}^k$-measurable, 
which would imply that $R^k$, $S^k$ and $\mvec{x}^k$ are also $\mathcal{F}^k$-measurable. Similar comments apply to the definition of 
$\mathcal{F}^{k-1/2}$.}, and 
$R^k$ is also $\mathcal{F}^{k-1/2}$-measurable;
thus they are well-defined random variables.




\begin{rem}
	The random variables $\mtx{A}^1, \dots, \mtx{A}^k$, $\mvec{y}^1, \dots, \mvec{y}^k$, $\mvec{x}^1, \dots, \mvec{x}^k$, $\mvec{p}^0, \mvec{p}^1, \dots, \mvec{p}^k$, $R^1$, $\dots$, $R^k$, $S^1, \dots, S^{k}$ are  $\mathcal{F}^{k}$-measurable since $\mathcal{F}^0 \subseteq \mathcal{F}^1 \subseteq \cdots \subseteq \mathcal{F}^{k}$. Also, $\mtx{A}^1, \dots, \mtx{A}^k$, $\mvec{y}^1, \dots, \mvec{y}^{k-1}$, $\mvec{x}^1, \dots, \mvec{x}^{k-1}$, $\mvec{p}^0, \mvec{p}^1, \dots, \mvec{p}^{k-1}$, $R^1$, $\dots$, $R^k$, $S^1, \dots, S^{k-1}$ are  $\mathcal{F}^{k-1/2}$-measurable since $\mathcal{F}^0 \subseteq \mathcal{F}^{1/2} \subseteq \mathcal{F}^1 \subseteq \cdots \subseteq \mathcal{F}^{k-1} \subseteq \mathcal{F}^{k-1/2}$.
	
\end{rem}



A weak assumption is given next, that is satisfied by reasonable techniques for the subproblems; namely, the  reduced problem \eqref{prob: AREGO_subproblem} needs to be solved to required accuracy with some positive probability.

\begin{assump}\label{assump: prob_of_I_R>rho}
	There exists $\rho \in (0,1]$ such that, for all $k \geq 1$,\footnote{The equality in the displayed equation follows from $\mathbb{E}[S^k | \mathcal{F}^{k-1}] = 1 \cdot \prob[ S^k = 1 | \mathcal{F}^{k-1}] + 0 \cdot \prob[ S^k = 0 | \mathcal{F}^{k-1} ]$.}
	$$  \prob[ S^k = 1 | \mathcal{F}^{k-1/2}]=\mathbb{E}[S^k | \mathcal{F}^{k-1/2}]  \geq \rho, $$ 
	i.e., with (conditional) probability at least $\rho > 0$, the solution $\mvec{y}^k$ of \eqref{prob: AREGO_subproblem} 
	satisfies \eqref{eq:approxf}.
\end{assump}

\begin{rem}
 If a deterministic (global optimization) algorithm is used to solve \eqref{prob: AREGO_subproblem_re}, then $S^k$ is  always $\mathcal{F}_k^{k-1/2}$-measurable and \Cref{assump: prob_of_I_R>rho} is equivalent to $S^k\geq \rho$. Since $S^k$ is an indicator function, this further implies that $S^k\equiv 1$, provided a sufficiently large computational budget is available.
\end{rem}


The results of \Cref{sec: estim_success} provide  a lower bound on the (conditional) probability of the reduced problem \eqref{prob: AREGO_subproblem} to be $(\epsilon-\lambda)$-successful, with the consequence given in the first part of the next Corollary. 

\begin{corollary}\label{corr: lowerbdRPK}
	If Assumptions  \ref{ass:G^*_is_nondegenerate} and \ref{ass:f_is_Lipschitz} hold, then 
	\begin{equation}\label{ineq: cond_exp>tau}
\mathbb{E}[R^k | \mathcal{F}^{k-1}]  \geq \tau, \quad {\rm for}\quad k\geq 1.
\end{equation}

If Assumption  \ref{assump: prob_of_I_R>rho} 
holds, then 
\begin{equation}\label{ineq: exp_IS_IR>tau*rho}
\mathbb{E}[R^k S^k | \mathcal{F}^{k-1/2}]  \geq  \rho R^k, \quad {\rm for}\quad k\geq 1.
\end{equation}
\end{corollary}
\begin{proof}
 Recall that the support of the random variable $\mvec{p}^k$ is contained in $\mathcal{X}$. For each embedding, we apply \Cref{thm:prob[I_S=1]>tau} (setting $\mvec{p} = \tilde{\mvec{p}}^{k-1}$ and replacing $\epsilon$ by $\epsilon-\lambda$) to deduce that there exists $\tau \in (0,1]$ such that
$\prob[R^k = 1 | \mathcal{F}^{k-1} ] \geq \tau$, for $k\geq 1$.
Then, in terms of conditional expectation, we have
$\mathbb{E}[R^k | \mathcal{F}^{k-1}] = 1 \cdot \prob[ R^k = 1 | \mathcal{F}^{k-1}] + 0 \cdot \prob[ R^k = 0 | \mathcal{F}^{k-1} ] \geq \tau$.

If Assumption  \ref{assump: prob_of_I_R>rho} holds, then
$\mathbb{E}[R^k S^k | \mathcal{F}^{k-1/2}] = R^k  \mathbb{E}[ S^k | \mathcal{F}^{k-1/2}] \geq  \rho R^k$,
where the equality follows from the fact that $R^k$ is $\mathcal{F}^{k-1/2}$-measurable (see \cite[Theorem 4.1.14]{Durrett2019}).
\end{proof}

\subsubsection{Global convergence proof}

A useful property is given next. 
\begin{lemma}\label{lemma: lim_of_prob_is_1}
	Let Assumptions  \ref{ass:G^*_is_nondegenerate}, \ref{ass:f_is_Lipschitz} and \ref{assump: prob_of_I_R>rho}  hold. Then, for $K\geq 1$, we have
	$$ \prob\Big[ \bigcup_{k=1}^K \left\{ \{R^k = 1\} \cap \{ S^k = 1 \} \right\} \Big] \geq 1 - (1-\tau \rho)^K. $$
\end{lemma}
\begin{proof}
	We define an auxiliary random variable,
$	 J^K :=  \mathds{1} \left(\bigcup_{k=1}^K \left\{  \{R^k = 1\} \cap \{ S^k = 1 \} \right\} \right).  $
	Note that $J^K = 1- \prod_{k=1}^{K} (1-R^k S^k)$. We have
\begin{align*}
	\prob\Big[ \bigcup_{k=1}^K \left\{ \{R^k = 1\} \cap \{ S^k = 1 \} \right\} \Big] &= \mathbb{E}[J^K]  = 1 - \mathbb{E}\Big[\prod_{k=1}^{K} (1-R^k S^k)\Big] \\
	& \stackrel{(*)}{=} 1 - \mathbb{E}\Big[\mathbb{E}\Big[ \prod_{k=1}^{K} (1-R^k S^k) \Big| \mathcal{F}^{K-1/2} \Big]\Big] \\
	& \stackrel{(\circ)}{=} 1 - \mathbb{E}\Big[ \prod_{k=1}^{K-1} (1-R^k S^k) \cdot \mathbb{E}\big[ 1 - R^K S^K | \mathcal{F}^{K-1/2} \big]\Big] \\
	& \geq 1 - \mathbb{E}\Big[(1-\rho R^K)\cdot \prod_{k=1}^{K-1} (1-R^k S^k)  \Big] \end{align*}
	\begin{align*}
\hspace*{4cm}	& \stackrel{(*)}{=} 1 - \mathbb{E}\Big[\mathbb{E}\Big[(1-\rho R^K)\cdot \prod_{k=1}^{K-1} (1-R^k S^k)   \Big| \mathcal{F}^{K-1}\Big]\Big] \\
	& \stackrel{(\circ)}{=} 1 - \mathbb{E}\Big[ \prod_{k=1}^{K-1} (1-R^k S^k) \cdot \mathbb{E}\big[ 1 - \rho R^K | \mathcal{F}^{K-1} \big]\Big] \\
	& \geq 1 - (1-\tau \rho) \cdot \mathbb{E}\Big[ \prod_{k=1}^{K-1} (1-R^kS^k)\Big],
	\end{align*}
where 
\begin{itemize}
\item[-]
$(*)$ follow from the tower property of conditional expectation (see (4.1.5) in \cite{Durrett2019}), 
\item[-]
$(\circ)$ is due to the fact that $R^1, \dots, R^{K-1}$ and $S^1,\dots,S^{K-1}$ are $\mathcal{F}^{K-1/2}$-  \,and $\mathcal{F}^{K-1}$-measurable (see Theorem 4.1.14 in \cite{Durrett2019}), 
\item[-]
 the inequalities follow from \eqref{ineq: exp_IS_IR>tau*rho} and \eqref{ineq: cond_exp>tau}, respectively. 
\end{itemize}
We repeatedly expand the expectation of the product for $K-1$, $\ldots$, $1$, in exactly the same manner as above, to obtain the desired result.
\end{proof}


In the next lemma, we show that if 
\eqref{prob: AREGO_subproblem} is $(\epsilon-\lambda)$-successful and  is solved to accuracy $\lambda$ in objective value, then the solution $\mvec{x}^k$ must be inside $G_{\epsilon}$; thus proving our intuitive 
statements (a) and (b) at the start of Section \ref{sec:globalX-REGO}.

\begin{lemma}\label{lemma: if WcapG then_x in G_epsilon}
		Suppose Assumptions \ref{ass:G^*_is_nondegenerate}, \ref{ass:f_is_Lipschitz} and  \ref{assump: prob_of_I_R>rho}  hold. Then
		$$
		\{R^k = 1\}\cap \{S^k = 1\} \subseteq \{\mvec{x}^{k}  \in G_{\epsilon}\}.
		$$
\end{lemma}
\begin{proof}
	By \Cref{def: eps_successful_AREGO}, if \eqref{prob: AREGO_subproblem} is $(\epsilon-\lambda)$-successful, then there exists $\mvec{y}^k_{int} \in \mathbb{R}^d$ such that $\mtx{A}^k \mvec{y}^k_{int} + \mvec{p}^{k-1} \in \mathcal{X}$ and 
	\begin{equation} \label{ineq: asym_conv_ineq1}
	f(\mtx{A}^k\mvec{y}^k_{int} + \mvec{p}^{k-1}) \leq f^* + \epsilon - \lambda.
	\end{equation}
	Since $\mvec{y}^k_{int}$ is in the feasible set of \eqref{prob: AREGO_subproblem} and $ f^k_{min}$ is the global minimum of \eqref{prob: AREGO_subproblem}, we have
	\begin{equation} \label{ineq: asym_conv_ineq3}
	 f^k_{min} \leq f(\mtx{A}^k\mvec{y}^k_{int} + \mvec{p}^{k-1}).
	\end{equation}
	Then, for $\mvec{x}^k$, \eqref{eq:approxf} gives the first inequality below,
	$$ f(\mvec{x}^k) \leq f^k_{min} + \lambda \leq f(\mtx{A}^k\mvec{y}^k_{int} + \mvec{p}^{k-1}) + \lambda \leq f^* + \epsilon, $$
	where the second and third inequalities follow from \eqref{ineq: asym_conv_ineq3} and \eqref{ineq: asym_conv_ineq1}, respectively. This shows that $\mvec{x}^k \in G_\epsilon$. 
\end{proof}

\begin{theorem}[Global convergence]\label{thm: glconv}
	Suppose Assumptions \ref{ass:G^*_is_nondegenerate}, \ref{ass:f_is_Lipschitz} and \ref{assump: prob_of_I_R>rho}  hold. Then
	$$\lim_{k\rightarrow \infty} \prob[\mvec{x}^k_{opt} \in G_{\epsilon}]=\lim_{k\rightarrow \infty} \prob[f(\mvec{x}^k_{opt}) \leq f^* + \epsilon] = 1$$
	where $\mvec{x}^k_{opt}$ and $G_{\epsilon}$ are defined in \eqref{eq: xoptk} and   \eqref{eq: G_epsilon}, respectively.
	
	Furthermore, for any $\xi \in (0,1)$, 
	\begin{equation}\label{eq:prob[x_opt^k_in_G_eps]>alpha}
\text{$\prob[ \mvec{x}^k_{opt} \in G_{\epsilon} ]= \prob[f(\mvec{x}^k_{opt}) \leq f^* + \epsilon]\geq \xi$ for all $k \geq K_{\xi}$,}
\end{equation}
where $K_\xi:= \displaystyle\ceil*{\frac{|\log(1-\xi)|}{\tau \rho}}$.
\end{theorem}
\begin{proof}
Lemma \ref{lemma: if WcapG then_x in G_epsilon} and the definition of $\mvec{x}^k_{opt}$ in \eqref{eq: xoptk} provide 
	$$ \{ R^k = 1 \} \cap \{ S^k = 1 \} \subseteq \{ \mvec{x}^k \in G_{\epsilon} \} \subseteq \{ \mvec{x}_{opt}^k \in G_{\epsilon} \} $$
	for $k = 1, 2,\dots, K$ and for any integer $K\geq 1$. Hence, 
	\begin{equation}\label{rel: cup_I is in cup_X}
	\bigcup_{k=1}^K \{ R^k = 1 \} \cap \{ S^k = 1 \} \subseteq \bigcup_{k=1}^K \{ \mvec{x}^k_{opt} \in G_{\epsilon} \}.
	\end{equation}
	Note that  the sequence $\{ f(\mvec{x}^1_{opt}), f(\mvec{x}^2_{opt}), \dots, f(\mvec{x}^K_{opt})\}$ is monotonically decreasing. Therefore, if $\mvec{x}^k_{opt} \in G_{\epsilon}$ for some $k \leq K$ then $\mvec{x}^i_{opt} \in G_{\epsilon}$ for all $i = k, \dots, K$; and so the sequence $(\{ \mvec{x}^k_{opt} \in G_{\epsilon} \})_{k = 1}^K$ is an increasing sequence of events. Hence,
	\begin{equation}\label{eq:cup_x_opt_in_G=x_opt_in_G}
	    \bigcup_{k=1}^K \{ \mvec{x}^k_{opt} \in G_{\epsilon} \} = \{ \mvec{x}^K_{opt} \in G_{\epsilon} \}.
	\end{equation}
	From \eqref{eq:cup_x_opt_in_G=x_opt_in_G} and \eqref{rel: cup_I is in cup_X}, we have for all $K\geq 1$,
	\begin{equation}\label{eq:prob[x_opt_in_G_eps>1-(1-tr)^K]}
	\prob[\{ \mvec{x}^K_{opt} \in G_{\epsilon} \}]  \geq \prob\Big[\bigcup_{k=1}^K \{ R^k = 1 \} \cap \{ S^k = 1 \} \Big]  \geq 1 - (1- \tau \rho)^K,
	\end{equation}
		where the second inequality follows from \Cref{lemma: lim_of_prob_is_1}.
Finally, passing to the limit with $K$ in \eqref{eq:prob[x_opt_in_G_eps>1-(1-tr)^K]}, we deduce
$	1 \geq \lim_{K \rightarrow \infty} \prob[\{ \mvec{x}^K_{opt} \in G_{\epsilon} \}] \geq \lim_{K \rightarrow \infty} \left[1 - (1- \tau \rho)^K\right] = 1$, as required.

 Note that if 
    \begin{equation} \label{eq:1-(1-tau rho)^K>alpha}
        1-(1-\tau \rho)^k \geq \xi
    \end{equation}
    then \eqref{eq:prob[x_opt_in_G_eps>1-(1-tr)^K]} implies $\prob[ \mvec{x}^k_{opt} \in G_{\epsilon} ] \geq \xi$. Since \eqref{eq:1-(1-tau rho)^K>alpha} is equivalent to
    $ k \geq \displaystyle\frac{\log(1-\xi)}{\log(1-\tau \rho)}$, 
     \eqref{eq:1-(1-tau rho)^K>alpha} holds for all $k\geq K_\xi$ since 
    $K_\xi \geq \displaystyle\frac{\log(1-\xi)}{\log(1-\tau \rho)}$.
\end{proof}

\begin{rem}
\label{rem:generalXREGO}
 Crucially, we note that X-REGO (\Cref{alg: AREGO}) is a generic framework that can be applied to a general, continuous objective $f$ in (P). Furthermore, the convergence result in \Cref{thm: glconv} also continues to hold in this general case provided \eqref{ineq: cond_exp>tau} can be shown to hold; this is  where we crucially use the special structure of low effective dimensionality of the objective that we investigate in this paper.
\end{rem}

\begin{rem}
If $f$ is a convex function (and known a priori to be so), then clearly, a local (deterministic or stochastic) optimization algorithm may be used to solve \eqref{prob: AREGO_subproblem_re}
and achieve \eqref{eq:approxf}. Apart from this important speed-up and simplification, it is difficult to exploit this additional special structure of $f$ in our analysis, in order to improve the success bounds and convergence.
\end{rem}

\paragraph*{Quantifiable rates of convergence when the effective subspace is aligned with coordinate axes}
Using the estimates for $\tau$ in \Cref{cor:prob[RPXis_succ]=Omega_general_p}, we can estimate precisely the rate of convergence of X-REGO as a function of problem dimension, assuming that $\mathcal{T}$ is aligned with coordinate axes.

\begin{theorem} 
Suppose \Cref{ass:AREGO_fun_eff_dim} holds with $\mtx{U} = [\mtx{I}_{d_e} \; \mtx{0}]^T$ and $\mtx{V} = [\mtx{0} \; \mtx{I}_{D-d_e}]^T$, as well as  Assumption \ref{assump: prob_of_I_R>rho}. Let $\xi \in (0,1)$, and $d_e$ and $d$ be fixed. Then \eqref{eq:prob[x_opt^k_in_G_eps]>alpha}
holds with 
    \begin{equation} \label{eq:K_bound_any_choice_of_p}
        \text{$K_\xi= \frac{\left|\log(1-\xi)\right|}{\rho} O\left(\frac{2^{D-d_e} \cdot(D-d_e+1)^{d_e}}{\log(D-d_e+1)^{\frac{d-1}{2}}}\right)$ as $D \rightarrow \infty$.}
    \end{equation}
If $\mvec{p}^k = \mvec{0}$ for $k\geq 0$, then \eqref{eq:prob[x_opt^k_in_G_eps]>alpha} holds  with
    \begin{equation}\label{eq:K_bound_p=0}
        \text{$K_\xi = \frac{\left|\log(1-\xi)\right|}{\rho} O\left(\frac{ (D-d_e+1)^{d_e}}{\log(D-d_e+1)^{\frac{d-1}{2}}}\right)$ as $D \rightarrow \infty$.}
    \end{equation}
\end{theorem}

\begin{proof}
Firstly, note our remark regarding assumptions below. 
    The result follows from \Cref{thm: glconv}, \eqref{eq:asym_exp_general_p} and \eqref{eq:asym_exp_p=0}. 
\end{proof}

\begin{rem}
 Assumptions \ref{ass:G^*_is_nondegenerate} and \ref{ass:f_is_Lipschitz} were required to prove \Cref{thm:prob[I_S=1]>tau} and, consequently, \eqref{ineq: cond_exp>tau}. If the effective subspace is aligned with coordinate axes, we no longer need Assumptions \ref{ass:f_is_Lipschitz} and \ref{ass:G^*_is_nondegenerate} to prove \eqref{ineq: cond_exp>tau}. In this case, \eqref{ineq: cond_exp>tau} follows from \Cref{cor:prob[RPXis_succ]=Omega_general_p}, together with the fact that \eqref{prob: AREGO_subproblem} being successful implies \eqref{prob: AREGO_subproblem} is $\epsilon$-successful for any $\epsilon \geq 0$. 
\end{rem}



%% file: numerics.tex
\subsection{Setup}
\paragraph{Algorithms.}  We test different variants of \Cref{alg: AREGO}  against the \textit{no-embedding} framework, in which \eqref{eq: GO} is solved directly without using random embeddings and with no explicit exploitation of its special structure. Each variant of X-REGO corresponds to a specific choice of  $\mvec{p}^k$, $k \geq 0$:

\begin{itemize}
    \item[-] Adaptive X-REGO (A-REGO). In X-REGO, the point $\mvec{p}^k$ is chosen as the best point found up to the $k$th embedding: if $f(\mtx{A}^k \mvec{y}^k+\mvec{p}^{k-1}) < f(\mvec{p}^{k-1})$ then  $\mvec{p}^{k} := \mtx{A}^k \mvec{y}^k+\mvec{p}^{k-1}$, otherwise, $\mvec{p}^k := \mvec{p}^{k-1}$.
    \item[-] Local Adaptive X-REGO (LA-REGO). In X-REGO, we solve \eqref{prob: AREGO_subproblem_re} using a local solver (instead of a global one as in N-REGO). Then, if $|f(\mtx{A}^k \mvec{y}^k+\mvec{p}^{k-1}) - f(\mvec{p}^{k-1})| > \gamma$ for some small $\gamma$ (here, $\gamma = 10^{-5}$), we let $\mvec{p}^{k} := \mtx{A}^k \mvec{y}^k+\mvec{p}^{k-1}$, otherwise, $\mvec{p}^k$ is chosen uniformly at random in $\mathcal{X}$. 
    \item[-] Nonadaptive X-REGO (N-REGO). In X-REGO, all the random subspaces are drawn at the origin:  $\mvec{p}^k := \mvec{0}$ for all $k$. 
    \item[-] Local Nonadaptive X-REGO (LN-REGO). In X-REGO, the low-dimensional problem \eqref{prob: AREGO_subproblem_re} is solved using a local solver, and the point $\mvec{p}^k$ is chosen uniformly at random in $\mathcal{X}$ for all $k$. 
\end{itemize}


\paragraph{Solvers.} We test the aforementioned X-REGO variants using three solvers for solving the reduced problem \eqref{prob: AREGO_subproblem_re} (or the original problem \eqref{eq: GO} in the no-embedding case), namely, DIRECT (\cite{Finkel2003,Gablonsky2001,Jones1993}), BARON (\cite{Sahinidis2014,Sahinidis2005}) and KNITRO (\cite{Byrd2006}).

DIRECT(\cite{Gablonsky2001, Jones1993, Finkel2003}) version 4.0 (DIviding RECTangles) is a deterministic\footnote{Here, we refer to the predictable behaviour of the solver given a fixed set of parameters.} global optimization solver, that does not require information about the gradient nor about the Lipschitz constant.

BARON(\cite{Sahinidis2014, Sahinidis2005}) version 17.10.10 (Branch-And-Reduce Optimization Navigator) is a state-of-the-art branch- and-bound type global optimization solver for nonlinear and mixed-integer programs, that is highly competitive \cite{Neumaier2005}. However, it accepts only a few (general) classes of functions (e.g., no trigonometric functions, no black box functions).

KNITRO(\cite{Byrd2006}) version 10.3.0 is a large-scale nonlinear local optimization solver that makes use of objective derivatives.  KNITRO has a multi-start feature, referred here as mKNITRO, allowing it to aim for global minimizers. 

We refer to \cite{Cartis2020} for a detailed description of the solvers. We test A-REGO and N-REGO using DIRECT, BARON and mKNITRO  and test LA-REGO and LN-REGO using only local KNITRO, with no multi-start.

\paragraph{Test set.} The methodology of these constructions is given in \cite{Cartis2020, Wang2016} and summarized here in Appendix \ref{app: Test set}. Our synthetic test set contains 19  $D$-dimensional functions with low effective dimension, with $D = 10,100$ and $1000$. We construct these high-dimensional functions from 19 global optimization problems (\Cref{table: Test set}, of dimensions 2--6) with known global minima \cite{Ernesto2005, AMPGO, SFAwebsite}, some of which are in the Dixon-Szego test set \cite{Dixon-Szego1975}. The construction process consists in artificially adding  coordinates to the original functions, and then applying a rotation to ensure that the effective subspace is not aligned with the coordinate axes.

\paragraph{Experimental setup.} 
For each version of X-REGO and its paired solvers, we solve the entire test set 5 times to estimate the average performance of the algorithms. Let $f$ be a function from the test set with the global minimum $f^*$. When applying any version of X-REGO to minimize $f$, we terminate either after $K = 100$ embeddings, or earlier, as soon as\footnote{We acknowledge that the use of the true global minimum $f^*$, or a sufficiently close lower bound, in our numerical testing is not practical. But we note that our aim here is to test both  `no-embedding' and X-REGO in similar, even if idealized, settings.}
\begin{equation} \label{eq: stopCrit}
    f(\wt{\mtx{A}}^k \tilde{\mvec{y}}^k + \tilde{\mvec{p}}^{k-1})-f^* \leq \epsilon = 10^{-3}.
\end{equation}
We then record the computational cost, which we measure in terms of either function evaluations or CPU time in seconds. To compare with `no-embedding', we solve the full-dimensional problem \eqref{eq: GO} directly with DIRECT, BARON and mKNITRO with no use of random embeddings. The budget and termination criteria used for each solver to solve \eqref{prob: AREGO_subproblem_re} within X-REGO or to solve \eqref{eq: GO} in the `no-embedding' framework are outlined in \Cref{table: solvers_descriptions}.
\begin{rem}
	The experiments are done not to compare solvers but to contrast `no-embedding' with the X-REGO variants. All the experiments were run in MATLAB on the 16 cores (2$\times$8 Intel with hyper-threading) Linux machines with 256GB RAM and 3300 MHz speed. 
\end{rem}

We compare the results using performance profiles (Dolan and Mor\'{e}, \cite{Dolan2002}), which measure the proportion of problems solved by the algorithm in less than a given budget defined based on the best performance among the algorithms considered. More precisely, for each solver (BARON, DIRECT and KNITRO), and for each algorithm $\mathcal{A}$ (the above-mentioned variants of X-REGO and `no-embedding'), we record $\mathcal{N}_p(\mathcal{A})$, the computational cost (see \Cref{table: solvers_descriptions}) of running algorithm $\mathcal{A}$ to solve problem $p$ within accuracy $\epsilon$. Let $\mathcal{N}^*_p$ be the minimum computational cost required for problem $p$ by any algorithm $\mathcal{A}$. The performance (probability) of algorithm $\mathcal{A}$ on the problem set $\mathcal{P}$ is defined as 
\begin{equation*} \label{eq: performanceProfile}
    \pi_{\mathcal{A}}(\alpha) = \frac{| \{p \in \mathcal{P} : \mathcal{N}_p(\mathcal{A}) \leq \alpha \mathcal{N}^*_p\} |}{|\mathcal{P}|},
\end{equation*}
with performance ratio $\alpha \geq 1$. As each experiment involving random embeddings is repeated five times, we obtain five curves for the corresponding algorithm-solver pairs.

\begin{table}[!t]
	\centering
	\footnotesize
	\caption{The table outlines the experimental setup for the solvers, used both in the `no-embedding' algorithm and for solving the low-dimensional problem \eqref{prob: AREGO_subproblem_re} (as usually, $f$ denotes the $D$-dimensional function to minimize, $f^*$ is its global minimum, and  
	$\epsilon$ in \eqref{eq: stopCrit} is set to $10^{-3}$). At each internal iteration, DIRECT stores $f_D^*$ --- the minimal value of $f$ found so far, while BARON stores $f_B^U$ and $f_B^L$ --- the smallest upper bound and largest lower bound found found so far. Note that, for BARON, $f_B^U = f(\mvec{x}^k)$ in \eqref{prob: AREGO_subproblem_re}.}
	\label{table: solvers_descriptions}
	\begin{tabular}{p{2.42cm}|p{2.75cm}|p{2.85cm}|p{3cm}|p{2.83cm}}
		& \multicolumn{1}{c | }{DIRECT} & \multicolumn{1}{c |}{BARON} & \multicolumn{1}{c |}{mKNITRO}& \multicolumn{1}{c}{KNITRO}\\ \hline
		\mbox{Measure of} computational cost& function evaluations & CPU seconds & function evaluations & function evaluations \\ \hline
		 \mbox{Max. budget} to solve \eqref{prob: AREGO_subproblem_re} & 3000 function evaluations & 5 CPU seconds & 5 starting points & 1 starting point \\ 
		 & & & &\\
		 
		 \mbox{Max. budget} to solve \eqref{eq: GO} & 60000 function evaluations & 1000 CPU seconds & 100 starting points & Not applicable \\ \hline
		 
		 Termination  for \eqref{prob: AREGO_subproblem_re} & Terminate either on budget or if \mbox{$f_D^* \leq f^* + \epsilon$} &  Terminate either on budget or if $f_B^U$ and $f_B^L$ satisfy $f_B^U \leq f_B^L + \epsilon$ & Default options (unless overwritten by additional options) & Default options (unless overwritten by additional options)\\ 
		 
		 Termination for \eqref{eq: GO} & Same as above &  Terminate either on budget or if $f_B^U$ satisfies $f_B^U \leq f^* + \epsilon$ & Same as above & Not applicable\\ \hline
		 
		 Additional \hspace{0.2cm} options for \eqref{prob: AREGO_subproblem_re} & \verb|testflag|=1 
		 \verb|maxits|=Inf \verb|globalmin|=$f^*$ & Default options  & \verb|ms_enable|=1
		 \verb|fstopval|=$f^*+\epsilon$ & \verb|fstopval|=$f^*+\epsilon$ \\ 
		 & & & & \\
		 Additional \hspace{0.2cm} options for \eqref{eq: GO} & Same as above & Same as above  & Same as above & Not applicable \\ \hline
	\end{tabular}
\end{table}

\subsection{Numerical results}
\begin{figure}[!t]
    \centering
    \includegraphics[scale=0.6]{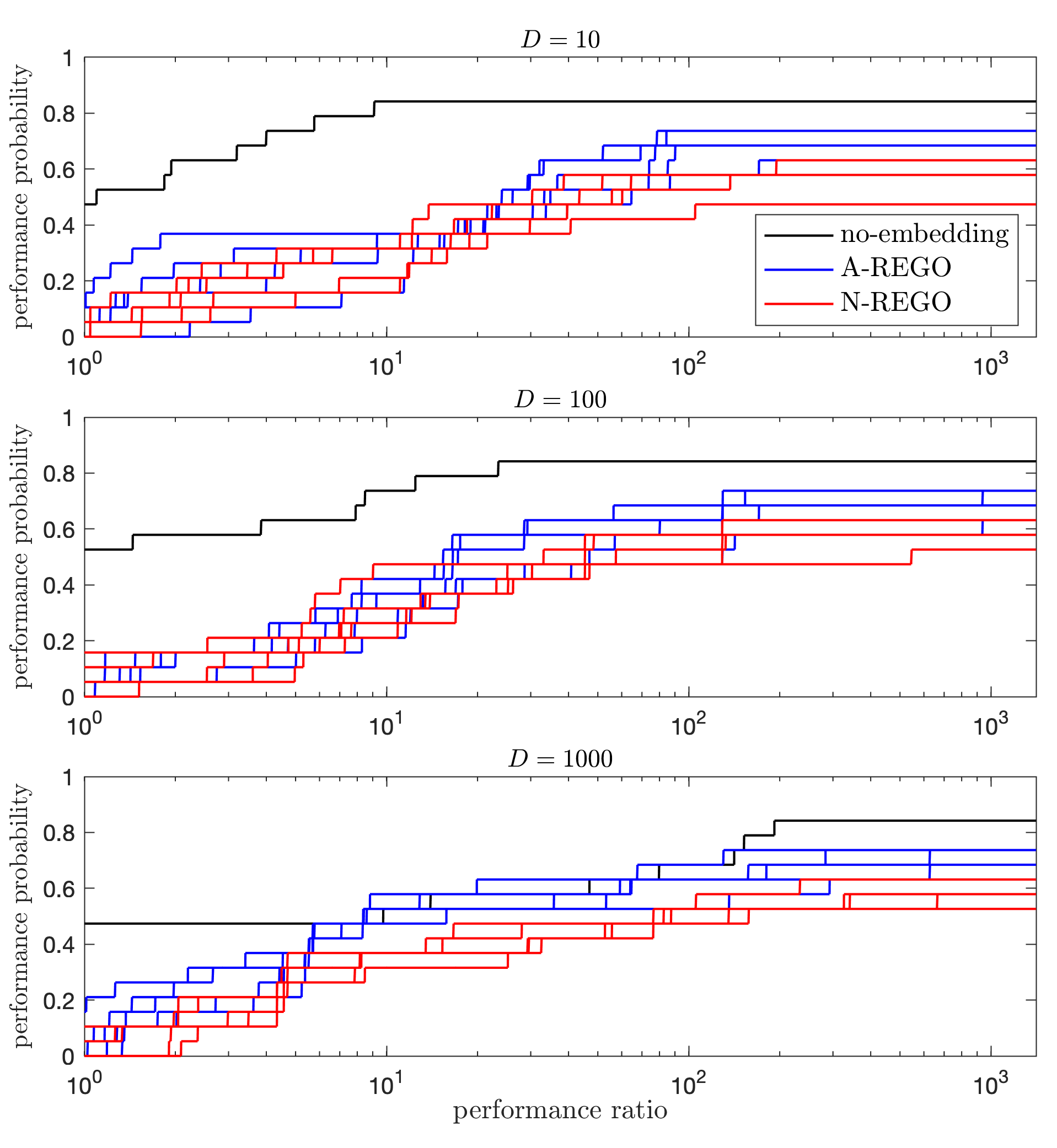}
    \caption{Comparison between  X-REGO variants and `no-embedding', using DIRECT to solve the subproblem \eqref{prob: AREGO_subproblem_re}.}
    \label{fig:DIRECT}
\end{figure}
\paragraph{DIRECT:} \Cref{fig:DIRECT} compares the adaptive and non-adaptive random embedding algorithms (A-REGO and N-REGO) to the no-embedding framework, when using the DIRECT solver for the reduced problem \eqref{prob: AREGO_subproblem_re} (and  for the full-dimensional problem in the case of the no-embedding framework). We find that the no-embedding framework outperforms the two X-REGO variants. We also note that this behaviour is more pronounced when the dimension of the problem \eqref{eq: GO} is small. In that regime, it is also difficult to determine which version of X-REGO performs the best. When $D$ is large, the no-embedding framework still outperforms the two variants of X-REGO, but among these two, the adaptive one (A-REGO) performs generally better than N-REGO. The median number of function evaluations required by the algorithms, measured over the five repetitions of the experiment, is given in \Cref{table: med_num_f_eval}. 


\paragraph{BARON:}  \Cref{fig:BARON} compares A-REGO and N-REGO to the no-embedding framework, when using BARON to solve the reduced problem \eqref{prob: AREGO_subproblem_re}. We find that the no-embedding framework is clearly outperformed by the two variants of X-REGO in the large-dimensional setting. Then, it is also clear that the adaptive variant of X-REGO outperforms the non-adaptive one. \Cref{table: med_num_f_eval} also indicates that the CPU time used by the different algorithms increases with the dimension of the problem, and that the increase is most rapid for  'no-embedding'.

\paragraph{KNITRO:} The comparison between the X-REGO variants, using (m)KNITRO to solve \eqref{prob: AREGO_subproblem_re}, is given in \Cref{fig:KNITRO}. Here, we also compare the local variants of X-REGO (namely, LA-REGO and LN-REGO), for which the reduced problem is solved using local KNITRO, with no multi-start feature. We find that the local variants outperform the global ones, and the no-embedding framework when the dimension of the problem is sufficiently large. \Cref{fig:KNITRO} also indicates that the local non-adaptive variant (LN-REGO) outperforms the adaptive one in this high-dimensional setting. This behaviour can also be observed in \Cref{table: med_num_f_eval}, which indicates that the median number of function evaluations increases significantly for LA-REGO while for LN-REGO,  it actually decreases.

\begin{figure}[!t]
    \centering
    \includegraphics[scale=0.6]{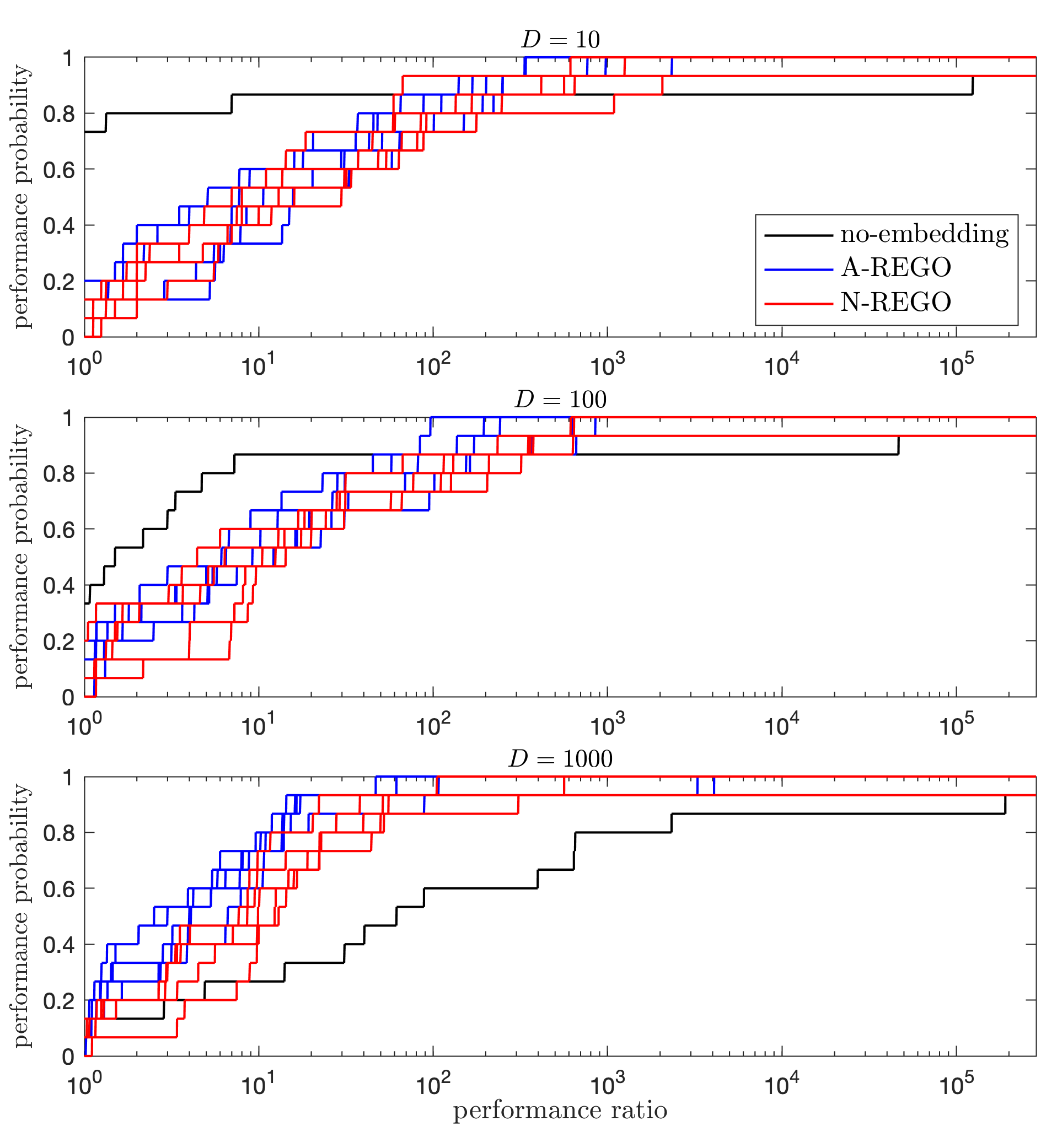}
    \caption{Comparison between  X-REGO variants and `no-embedding', using BARON to solve the subproblem \eqref{prob: AREGO_subproblem_re}.}
    \label{fig:BARON}
\end{figure}
\begin{figure}[!t]
    \centering
    \includegraphics[scale=0.6]{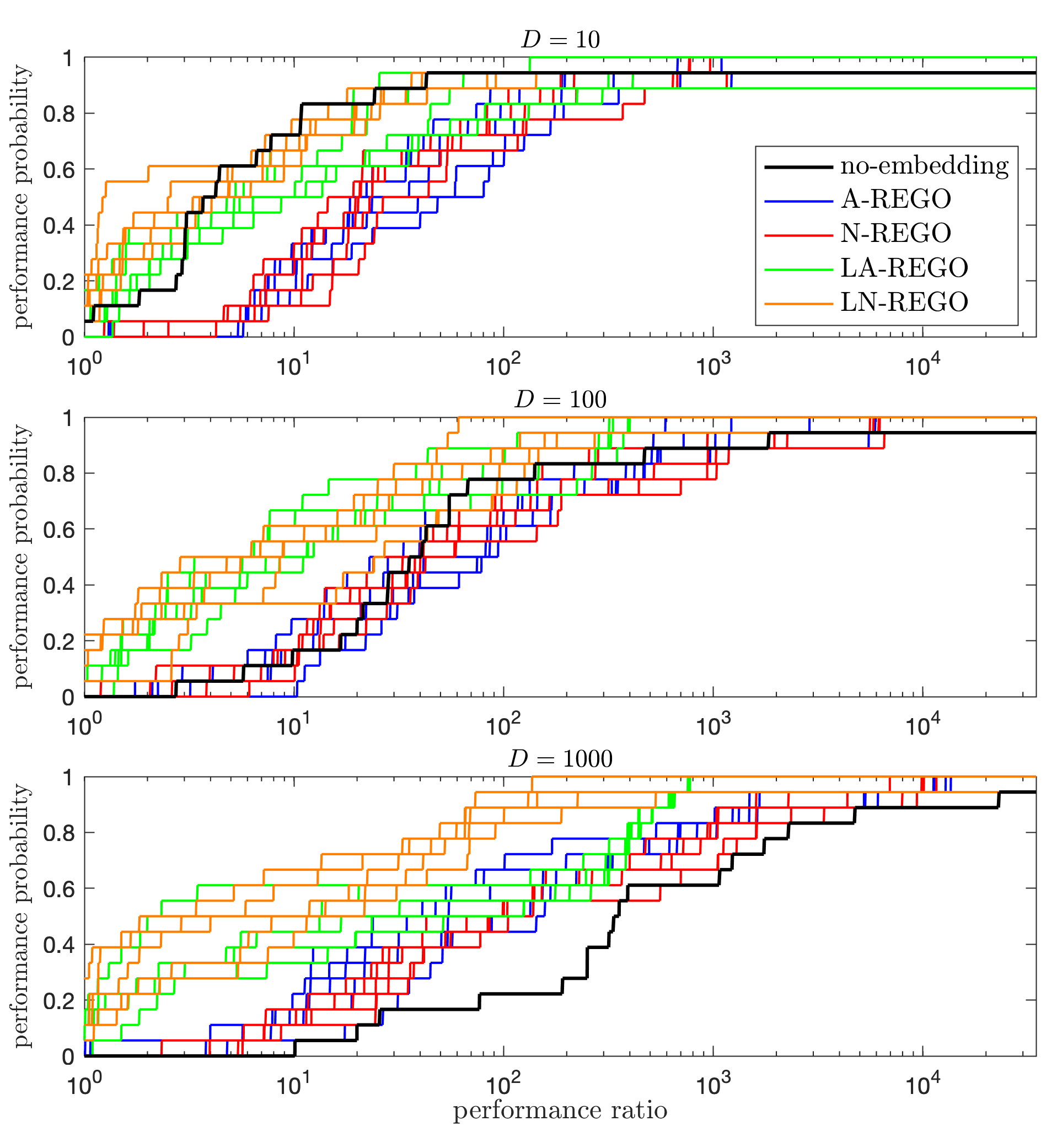}
    \caption{Comparison between  X-REGO variants and `no-embedding', using KNITRO to solve the subproblem \eqref{prob: AREGO_subproblem_re}.}
    \label{fig:KNITRO}
\end{figure}

\begin{table}[!t]
	\centering
	\footnotesize
	\caption{Median number of function evaluations or CPU time spent by each algorithm-solver pair.}
	\label{table: med_num_f_eval}
	\begin{tabular}{p{1.85cm}|p{1.05cm} p{1.05cm} p{1.05cm}|p{1.05cm} p{1.05cm} p{1.05cm}|p{1.05cm} p{1.05cm}p{1.05cm}|}
		&\multicolumn{3}{c |}{DIRECT (fun. evals)} & \multicolumn{3}{c |}{BARON (CPU time)}  & \multicolumn{3}{c |}{KNITRO (fun. evals)} \\
        & \mbox{$D = 10$} & \mbox{$D = 10^2$} &  \mbox{$D = 10^3$} & \mbox{$D = 10$} &  \mbox{$D = 10^2$} & \mbox{$D = 10^3$} & \mbox{$D = 10$} & \mbox{$D = 10^2$} & \mbox{$D = 10^3$}  \\ \hline
    \mbox{no-embedding} & \multicolumn{1}{r}{1261}  &\multicolumn{1}{r}{16933}   & \multicolumn{1}{r|}{63795}  &\multicolumn{1}{r}{0.08}  &\multicolumn{1}{r}{0.50}  &\multicolumn{1}{r|}{155.20}  &\multicolumn{1}{r}{220}  &\multicolumn{1}{r}{1425}  &\multicolumn{1}{r|}{11542} \\
    A-REGO & \multicolumn{1}{r}{24569}   &\multicolumn{1}{r}{300348}  &\multicolumn{1}{r|}{300276}  &\multicolumn{1}{r}{0.63}  &\multicolumn{1}{r}{1.93}  &\multicolumn{1}{r|}{15.66}  &\multicolumn{1}{r}{1534}  &\multicolumn{1}{r}{3992}  &\multicolumn{1}{r|}{5346}   \\
    N-REGO & \multicolumn{1}{r}{63093}  &\multicolumn{1}{r}{300484}  &\multicolumn{1}{r|}{300532}  &\multicolumn{1}{r}{0.82}  &\multicolumn{1}{r}{3.00}  &\multicolumn{1}{r|}{21.51}  &\multicolumn{1}{r}{1582}  &\multicolumn{1}{r}{3606}  &\multicolumn{1}{r|}{8766}  \\
    LA-REGO & \multicolumn{1}{r}{--} & \multicolumn{1}{r}{--} & \multicolumn{1}{r|}{--} & \multicolumn{1}{r}{--} & \multicolumn{1}{r}{--} & \multicolumn{1}{r|}{--} &\multicolumn{1}{r}{368}  &\multicolumn{1}{r}{631}  &\multicolumn{1}{r|}{2564}  \\
    LN-REGO & \multicolumn{1}{r}{--} & \multicolumn{1}{r}{--} & \multicolumn{1}{r|}{--} & \multicolumn{1}{r}{--} & \multicolumn{1}{r}{--} & \multicolumn{1}{r|}{--} &\multicolumn{1}{r}{220}  &\multicolumn{1}{r}{763}  &\multicolumn{1}{r|}{704}  \\
	\end{tabular}
\end{table}

\paragraph{Conclusions to numerical experiments}
The numerical experiments presented in this paper indicate that, as expected, the X-REGO algorithm is mostly beneficial  for high-dimensional problems, when $D$ is large. In this setting,  X-REGO variants paired with the  BARON and mKNITRO solvers
outperform the 'no-embedding' approach, of applying these solvers directly to the problems. It is less obvious to decide which variant of X-REGO is best, but it seems that, at least on the problem set considered, the local variants outperform the global ones.

%% file: appendix_tech_def.tex
\subsection{Gaussian random matrices}
\begin{definition} (Gaussian matrix, see \cite[Definition 2.2.1]{Gupta1999}) \label{def: Gaussian_matrix} 
	A Gaussian (random) matrix is a matrix $\mtx{A} = (a_{ij})$, where the entries $a_{ij} \sim \mathcal{N}(0,1)$ are independent (identically distributed) standard normal variables. 
\end{definition}
Gaussian matrices have been well-studied with many results available at hand. Here, we mention a few key properties of Gaussian matrices that we use in the analysis; for a collection of results pertaining to Gaussian matrices and other related distributions refer to \cite{Gupta1999, Vershynin2018}. 

\begin{theorem} \label{thm: orthog_inv_of_Gaussian_matrices} (see \cite[Theorem 2.3.10]{Gupta1999})
	Let $\mtx{A}$ be a $D \times d$ Gaussian random matrix. If $\mtx{U} \in \mathbb{R}^{D \times p}$, $D \geq p $, and $\mtx{V} \in \mathbb{R}^{d \times q}$, $d \geq q$, are orthonormal, then $\mtx{U}^T\mtx{A}\mtx{V}$ is a Gaussian random matrix. 
\end{theorem}

\begin{theorem} \label{thm:indep_of_Gaussian_matrices} (see \cite[Theorem 2.3.15]{Gupta1999}) Let $\mtx{A}$ be a $D \times d$ Gaussian random matrix, and let $\mtx{X} \in \mathbb{R}^{r \times D}$ and $\mtx{Y} \in \mathbb{R}^{q \times D}$ be given matrices. Then, $\mtx{X}\mtx{A}$ and $\mtx{Y}\mtx{A}$ are independent if and only if $\mtx{X}\mtx{Y}^T = \mtx{0}$.
\end{theorem}

\begin{theorem} \label{thm: Wishart_is_nonsingular} (see \cite[Theorem 3.2.1]{Gupta1999})
	Let $\mtx{A}$ be a $D \times d$ Gaussian random matrix with $D \geq d$. Then, the Wishart matrix $\mtx{A}^T\mtx{A}$ is positive definite, and hence nonsingular, with probability one. 
\end{theorem}

\subsection{Other relevant probability distributions}
\begin{definition}[Chi-squared distribution]  \label{def:chi-squared_rv}

	Given a collection $Z_1, Z_2, \dots, Z_n$ of $n$ independent standard normal variables, the random variable $W = Z_1^2+Z_2^2+\cdots Z_n^2$ is said to follow the chi-squared distribution with $n$ degrees of freedom (see \cite[A.2]{Lee2012}). We denote this by $W \sim \chi_n^2$.
\end{definition}

\begin{theorem}\label{thm:Rayleigh_quotient}
(see \cite[Theorem 3.3.12]{Gupta1999}) Let $\mtx{M}$ be an $n \times l$ Gaussian matrix with $n \geq l$, $\mvec{y}$ be an $l \times 1$ random vector distributed independently of $\mtx{M}^T\mtx{M}$, and $\prob[\mvec{y} \neq \mvec{0}] = 1$. Then, 
$$ \frac{\mvec{y}^T \mtx{M}^T\mtx{M}\mvec{y}}{\mvec{y}^T\mvec{y}} \sim \chi^2_n$$
and is independent of $\mvec{y}$.
\end{theorem}

\begin{definition}[Inverse chi-squared distribution] \label{def: inv_chi-squared}
	Given $X \sim \chi_n^2$, the random variable $Y = 1/X$ is said to follow the inverse chi-squared distribution with $n$ degrees of freedom. We denote this by $Y \sim 1/\chi_n^2$ (see \cite[A.5]{Lee2012}).  
\end{definition}

\begin{definition}[Multivariate $t$-distribution]
An $l$-dimensional random variable $\mvec{t}$ is said to have $t$-distribution with parameters $\nu$ and $\mtx{\Sigma}$ if its joint p.d.f. is given by (see \cite[Chapter 4]{Gupta1999})
\begin{equation} \label{eq:pdf_of_t_distr}
    f(\mvec{t}) = \frac{1}{(\pi \nu)^{l/2}}\left[ \frac{\Gamma\left(\frac{l+\nu}{2}\right)}{\Gamma\left(\frac{\nu}{2}\right)} \right] \det(\mtx{\Sigma})^{-1/2} \left(1+\frac{1}{\nu} \mvec{t}^T \mtx{\Sigma}^{-1}\mvec{t} \right)^{-(l+\nu)/2},
\end{equation}
where $\Gamma$ is the usual gamma function.
\end{definition}

\begin{definition}[$F$-distribution] \label{def:F-distribution} 
Let $W_1 \sim \chi_{\nu_1}^2$ and $W_2 \sim \chi_{\nu_2}^2$ be independent. A random variable $X$ is said to follow an $F$-distribution with degrees of freedom $\nu_1$ and $\nu_2$ if 
$$ X \sim \frac{W_1/\nu_1}{W_2/\nu_2}.$$
We denote this by $X \sim F(\nu_1,\nu_2)$.
The p.d.f. of $X$ is given by (see \cite[A.19]{Lee2012})
\begin{equation}\label{eq:pdf_of_F}
\text{$f(x) = \frac{\Gamma(\frac{\nu_1+\nu_2}{2})}{\Gamma(\frac{\nu_1}{2})\Gamma(\frac{\nu_2}{2})} \left( \frac{\nu_1}{\nu_2} \right)^{\nu_1/2} x^{\nu_1/2-1} \left( 1+\frac{\nu_1}{\nu_2}x \right)^{-\frac{\nu_1+\nu_2}{2}}$ for $x>0$.}
\end{equation}
\end{definition}

\subsection{Additional relevant results}

\begin{lemma} \label{lemma: Fang_equally_distributed_random_vectors} \cite[p.~13]{Fang1990}
	Let $\mvec{x}$ and $\mvec{y}$ be random vectors such that $\mvec{x} \stackrel{law}{=} \mvec{y}$ and let $f_i(\cdot)$, $i=1,2,\dots,m$, be measurable functions. Then,
	$$ \begin{pmatrix}
	f_1(\mvec{x}) & f_2(\mvec{x}) & \dots & f_m(\mvec{x}) 
	\end{pmatrix}^T \stackrel{law}{=}  \begin{pmatrix}
	f_1(\mvec{y}) & f_2(\mvec{y}) & \dots & f_m(\mvec{y})
	\end{pmatrix}^T.$$
\end{lemma}

The last results apply to spherical probability distributions, defined as follows (for more details regarding spherical distributions, refer to \cite{Fang1990, Gupta1999, Bernardo2000}).

\begin{definition}\label{def: spherical_distribution}
	An $n\times 1$ random vector $\mtx{x}$ is said to have a spherical distribution if for every orthogonal $n \times n$ matrix $\mtx{U}$,
	$$ \mtx{U}\mvec{x} \stackrel{law}{=} \mvec{x}.$$
\end{definition}



\begin{theorem} \label{thm: pdf_of_x_Gupta} (see \cite[Theorem 2.1.]{Gupta1997})
	Let $\mvec{x} \stackrel{law}{=} r\mvec{u}$ be a spherically distributed $n \times 1$ random vector with $\prob[\mvec{x} = \mvec{0}] = 0$, where $r$ is independent of $\mvec{u}$ with p.d.f. $h(\cdot)$. Then, the p.d.f. $g(\hat{\mvec{x}})$ of $\mvec{x}$ is given by 
	$$ g(\hat{\mvec{x}}) = \frac{\Gamma(n/2)}{2\pi^{n/2}} h(\| \hat{\mvec{x}} \|) \| \hat{\mvec{x}} \|^{1-n}. $$
\end{theorem}

\section{Proof of \texorpdfstring{\Cref{thm: By=z}}{Theorem 2.5} } \label{app:proof_of_By=z}
We prove that $\mvec{y}^* \in S^*$ if and only if $\mtx{B} \mvec{y}^*  = \mvec{z}^*$; \eqref{eq: Euclidean_norm_sol_explicit_formula} then immediately follows from \eqref{eq: minimal_2_norm}.
Let $\mvec{y}^* \in \mathbb{R}^d$ be such that $\mtx{A}\mvec{y}^* + \mvec{p}\in \mathcal{G}^*$. First, we establish that 
	\begin{equation}\label{assert: f(Ay)=f(x)_iff_x=UU^TAy}
	\text{$\mtx{A}\mvec{y}^* + \mvec{p}\in \mathcal{G}^*$ if and only if $\mvec{x}_{\top}^* - \mvec{p}_{\top} = \mtx{U}\mtx{U}^T\mtx{A}\mvec{y}^*$.}
	\end{equation}
	Suppose that $\mtx{A}\mvec{y}^* + \mvec{p} \in \mathcal{G}^*$. Then, using the definition of $\mathcal{G}^*$ (see \Cref{def:AREGO_choice_of_x^*}) we can write $\mtx{A}\mvec{y}^* + \mvec{p} = \mvec{x}_{\top}^* + \tilde{\mvec{x}}$ for some $\tilde{\mvec{x}} \in \mathcal{T}^{\perp}$.
	We have
	$$ \mtx{U}\mtx{U}^T\mtx{A}\mvec{y}^*+\mvec{p}_{\top} =  \mtx{U}\mtx{U}^T(\mtx{A}\mvec{y}^*+\mvec{p}) = \mtx{U}\mtx{U}^T(\mvec{x}_{\top}^* + \tilde{\mvec{x}}) = \mvec{x}_{\top}^*,$$
	where we have used $\mtx{U}\mtx{U}^T \mvec{x}_{\top}^* = \mvec{x}_{\top}^*$ and $\mtx{U}\mtx{U}^T \tilde{\mvec{x}} = \mvec{0}$. Conversely, assume that $\mvec{y}^*$ satisfies 
	\begin{equation}\label{eq: x_T^*=UU^TAy}
	\mvec{x}^*_{\top} - \mvec{p}_{\top} = \mtx{U}\mtx{U}^T \mtx{A}\mvec{y}^*.
	\end{equation}
	Denote by $\mtx{S}$ the $D \times D$ orthogonal matrix $(\mtx{U} \; \mtx{V})$, where $\mtx{V}$ is defined in \Cref{ass:AREGO_fun_eff_dim}. Using \eqref{eq: x_T^*=UU^TAy} and the identity $\mtx{U}\mtx{U}^T + \mtx{V}\mtx{V}^T = \mtx{S} \mtx{S}^T = \mtx{I}_D$, we obtain
	\begin{equation*}
	\begin{aligned}
	    \mtx{A}\mvec{y}^* + \mvec{p} & = (\mtx{U}\mtx{U}^T + \mtx{V}\mtx{V}^T)(\mtx{A}\mvec{y}^*+\mvec{p}) \\
	    & = \mtx{U}\mtx{U}^T\mtx{A}\mvec{y}^* + \mtx{U}\mtx{U}^T\mvec{p}+\mtx{V}\mtx{V}^T(\mtx{A}\mvec{y}^*+\mvec{p})  \\
	    & = \mvec{x}^*_{\top} - \mvec{p}_{\top} + \mvec{p}_{\top} + \mtx{V}\mtx{V}^T(\mtx{A}\mvec{y}^*+\mvec{p}) \\
	    & = \mvec{x}_{\top}^* + \mtx{V}\mtx{V}^T(\mtx{A}\mvec{y}^*+\mvec{p}).
	\end{aligned}
	\end{equation*}
	Note that $\mtx{V}\mtx{V}^T(\mtx{A}\mvec{y}^*+\mvec{p})$ lies on $\mathcal{T}^{\perp}$ as it is the orthogonal projection of $\mvec{A}\mvec{y}^*+\mvec{p}$ onto $\mathcal{T}^{\perp}$, which implies that $\mtx{A}\mvec{y}^*+\mvec{p} \in \mathcal{G}^*$.
	This completes the proof of \eqref{assert: f(Ay)=f(x)_iff_x=UU^TAy}.
	
	Now we show that \eqref{eq: By=z} and \eqref{eq: x_T^*=UU^TAy} are equivalent. We multiply both sides of $\mvec{x}^*_{\top}-\mvec{p}_{\top} = \mtx{U}\mtx{U}^T \mtx{A}\mvec{y}^*$ by $\mtx{S}^T$, and obtain
	\begin{equation} \label{eq: Qx^* = QUU^TAy^*}
	\begin{pmatrix}
	\mtx{U}^T \\
	\mtx{V}^T
	\end{pmatrix} (\mvec{x}_{\top}^*-\mvec{p}_{\top}) = \begin{pmatrix}
	\mtx{U}^T \\
	\mtx{V}^T
	\end{pmatrix} \mtx{U}\mvec{U}^T\mtx{A}\mvec{y}^*.
	\end{equation}
	Since $\mvec{x}_{\top}^*-\mvec{p}_{\top}$ is in the column span of $\mtx{U}$, we can write $\mvec{x}_{\top}^* - \mvec{p}_{\top} = \mtx{U}\mvec{z}^*$
	for some (unique) vector $\mvec{z}^* \in \mathbb{R}^{d_e}$. By substituting the above into \eqref{eq: Qx^* = QUU^TAy^*} we obtain
	$$ \begin{pmatrix}
	\mtx{U}^T\mtx{U}\mvec{z}^* \\
	\mtx{V}^T \mtx{U} \mvec{z}^*
	\end{pmatrix} = \begin{pmatrix}
	\mtx{U}^T\mtx{U} \mtx{U}^T\mtx{A}\mvec{y}^* \\
	\mtx{V}^T \mtx{U} \mtx{U}^T \mtx{A} \mvec{y}^*
	\end{pmatrix}.$$
	This reduces to 
	$$\begin{pmatrix}
	\mvec{z}^* \\
	\mvec{0}
	\end{pmatrix} = \begin{pmatrix}
	\mtx{U}^T\mtx{A}\mvec{y}^* \\
	\mtx{0}
	\end{pmatrix},$$
	where we have used the identities $\mtx{U}^T\mtx{U} = \mtx{I}$ and $\mtx{V}^T\mtx{U} = \mvec{0}$, which follow from \Cref{ass:AREGO_fun_eff_dim}.
	To obtain \eqref{eq: x_T^*=UU^TAy} from \eqref{eq: By=z}, multiply \eqref{eq: By=z} by $\mtx{U}$.

%% file: appendix_distrib_w.tex
We derive the probability density function of the random vector\footnote{For the vector $\mvec{w}$ to be well-defined, we require $d_e<D$ (see \Cref{ass:AREGO_fun_eff_dim}). If $d_e=D$, then $d = D$; letting $\mtx{Q} = \mtx{I}$ and using $\mvec{z}^* = \mvec{x}^* - \mvec{p}$ in \eqref{eq:prob_there_exists_y}--\eqref{eq:prob_Ay^2*_is_between_-1_and_1}, it is straightforward to see that $\prob[\text{\eqref{eq: AREGO} is successful}] = 1$.} $\mvec{w}$ defined in \eqref{w-def} 
following a similar line of argument as in \cite{Cartis2020}: we first derive the distribution of $\| \mvec{w} \|_2^2$ and then show that $\mvec{w}$ follows a spherical distribution, which then allows us to derive the exact distribution of $\mvec{w}$.
\begin{theorem}[Distribution of $\|\mvec{w}\|_2^2$] \label{thm:dist_of_norm_of_w}
Suppose that \Cref{ass:AREGO_fun_eff_dim} holds. Let $\mvec{x}^*$ be a(ny) global minimizer of \eqref{eq: GO}, $\mvec{p}\in \mathcal{X}$, a given vector, and $\mtx{A}$, a $D \times d$ Gaussian matrix. Assume that $\mvec{p}_\top \neq \mvec{x}_\top^*$, where the subscript represents the Euclidean projection on the effective subspace.  Let $\mvec{w}$ be defined in \eqref{w-def}.
Then, 
$$ \left(\frac{1}{\| \mvec{x}_{\top}^* - \mvec{p}_{\top}\|^2_2} \cdot \frac{d-d_e+1}{D-d_e}\right) \|\mvec{w}\|_2^2 \sim  F(D-d_e,d-d_e+1),$$
where $F(v_1,v_2)$ denotes the $F$-distribution with degrees of freedom $v_1$ and $v_2$.  
\end{theorem}
\begin{proof}
We write $\mvec{w}$ as $\mtx{C}\mvec{y}_2^*$, where $\mtx{C} = \mtx{V}^T\mtx{A}$. We first establish three facts: a) $\mtx{B}$ and $\mtx{C}$ are independent; b) $\mvec{y}_2^*$ and $\mtx{C}$ are independent; c) $\prob[\mvec{y}_2^* \neq \mvec{0}] = 1$. 

\begin{enumerate}
    \item[a)] Since $\mtx{V}$ is orthonormal, \Cref{thm: orthog_inv_of_Gaussian_matrices} implies that $\mtx{C}$ is a Gaussian matrix. Moreover, the fact $\mtx{U}^T\mtx{V} = \mtx{0}$ implies that $\mtx{B}$ and $\mtx{C}$ are independent, see \Cref{thm:indep_of_Gaussian_matrices}.
    \item[b)] Since $\mvec{y}_2^*$ is  measurable as a function of $\mtx{B}$
    (see proof \cite[Lemma A.16]{Cartis2020}), $\mvec{y}_2^*$ and $\mtx{C}$ must be independent. 
    \item[c)] We have $\prob[\mvec{y}_2^* \neq \mvec{0}] = 1 - \prob[\mvec{y}_2^* = \mvec{0}] = 1 - \prob[\| \mvec{y}_2^* \|_2^2 = 0] = 1-0,$ where the last equality is due to the fact that $\|\mvec{y}_2^*\|_2^2$ follows the (appropriately scaled) inverse chi-squared distribution (\Cref{thm: x^*_T/y^*_2_follow_chi_square}), which is a continuous distribution. 
\end{enumerate}
Now, we apply \Cref{thm:Rayleigh_quotient} to obtain
\begin{equation}
    \frac{\|\mvec{w}\|_2^2}{\| \mvec{y}_2^* \|_2^2} = \frac{(\mvec{y}_2^*)^T \mtx{C}^T \mtx{C} \mvec{y}_2^*}{\| \mvec{y}_2^* \|_2^2} \sim \chi^2_{D-d_e},
\end{equation}
which together with \Cref{thm: x^*_T/y^*_2_follow_chi_square} yields
\begin{equation}
    \frac{\|\mvec{w}\|_2^2}{\|\mvec{x}_{\top}^* - \mvec{p}_{\top}\|_2^2}  \sim \frac{\chi^2_{D-d_e}}{\chi^2_{d-d_e+1}},
\end{equation}
where $\chi^2_{D-d_e}$ and $\chi^2_{d-d_e+1}$ are independent\footnote{\Cref{thm:Rayleigh_quotient} implies that $\mvec{y}_2^* (= \| \mvec{x}_{\top}^*-\mvec{p}_{\top} \|/\chi^2_{d-d_e+1} )$ and $\chi^2_{D-d_e}$ are independent; hence, $\chi^2_{D-d_e}$ and $\chi^2_{d-d_e+1}$ must also be independent.}. Using the definition of the $F$-distribution (see \Cref{def:F-distribution}), we obtain the desired result. 
\end{proof}

Using \Cref{thm:dist_of_norm_of_w}, it is straightforward to derive the p.d.f of $\|\mvec{w}\|$.
\begin{theorem}[The p.d.f. of $\|\mvec{w}\|$]
Suppose that \Cref{ass:AREGO_fun_eff_dim} holds. Let $\mvec{x}^*$ be a(ny) global minimizer of \eqref{eq: GO}, $\mvec{p}\in \mathcal{X}$, a given vector, and $\mtx{A}$, a $D \times d$ Gaussian matrix. Assume that $\mvec{p}_\top \neq \mvec{x}_\top^*$, where the subscript represents the Euclidean projection on the effective subspace.
The p.d.f. $h(\hat{w})$ of $\|\mvec{w}\|$, with $\mvec{w}$ defined in \eqref{w-def}, is given by
\begin{equation}\label{eq:pdf_of_norm_of_w}
    h(\hat{w}) = \frac{2\Gamma(\frac{m+n}{2})}{\Gamma(\frac{m}{2})\Gamma(\frac{n}{2})} \cdot \frac{\hat{w}^{m-1}}{\|\mvec{x}_{\top}^* - \mvec{p}_{\top}\|^m} \left(1+\frac{\hat{w}^2}{\|\mvec{x}_{\top}^* - \mvec{p}_{\top}\|^2}\right)^{-(m+n)/2},
\end{equation}
where $m = D-d_e$, $n=d-d_e+1$, and where $\Gamma$ is the usual gamma function.
\end{theorem}
\begin{proof}
Let $X \sim F(D-d_e, d-d_e+1)$. \autoref{thm:dist_of_norm_of_w} implies that
\begin{equation}
    \|\mvec{w}\| \stackrel{law}{=} K \sqrt{X},
\end{equation}
where 
\begin{equation}
    K = \| \mvec{x}_{\top}^* - \mvec{p}_{\top} \| \sqrt{\frac{D-d_e}{d-d_e+1}}.
\end{equation}
For the p.d.f. of $\|\mvec{w}\|$, we have
\begin{equation} \label{eq:g(w)=2wf/K^2}
    h(\hat{w}) = \frac{d}{d\hat{w}} \prob[\|\mvec{w}\| < \hat{w}] = \frac{d}{d\hat{w}} \prob[K\sqrt{X} < \hat{w}] = \frac{d}{d\hat{w}} \prob[X < \hat{w}^2/K^2] = \frac{2\hat{w}}{K^2} f(\hat{w}^2/K^2),
\end{equation}
where $f(x)$ denotes the p.d.f of an $F$-distributed random variable with degrees of freedom $m = D-d_e$ and $n = d-d_e+1$.
By substituting \eqref{eq:pdf_of_F} in \eqref{eq:g(w)=2wf/K^2}, we obtain the desired result.
\end{proof}
To derive the p.d.f.~of $\mvec{w}$ we  rely on the fact that $\mvec{w}$ has a spherical distribution (see \Cref{def: spherical_distribution}), as we show next.
\begin{theorem}[$\mvec{w}$ has a spherical distribution] \label{thm:w_has_a_spherical_distr}
Suppose that \Cref{ass:AREGO_fun_eff_dim} holds. Let $\mvec{x}^*$ be a(ny) global minimizer of \eqref{eq: GO}, $\mvec{p}\in \mathcal{X}$, a given vector, and $\mtx{A}$, a $D \times d$ Gaussian matrix. Assume that $\mvec{p}_\top \neq \mvec{x}_\top^*$, where the subscript represents the Euclidean projection on the effective subspace. The random vector $\mvec{w}$, defined in \eqref{w-def}, has a spherical distribution.
\end{theorem}
\begin{proof}
Our proof is similar to the proof of Lemma A.16 in \cite{Cartis2020}.
Let $\mtx{S}$ be any $(D-d_e) \times (D-d_e)$ orthogonal matrix.
To prove that $\mvec{w}$ has a spherical distribution, we need to show that
\begin{equation}
    \mvec{w} \stackrel{law}{=} \mtx{S}\mvec{w}.
\end{equation}
Using \eqref{eq: Euclidean_norm_sol_explicit_formula}, we write $\mvec{w} = \mtx{C}\mtx{B}^T(\mtx{B}\mtx{B}^T)^{-1}\mvec{z}^*$, where $\mtx{C} = \mtx{V}^T\mtx{A}$ and $\mtx{B}=\mtx{U}^T\mtx{A}$ are Gaussian matrices independent of one another by the point a) of the proof of \Cref{thm:dist_of_norm_of_w}. Let $f : \mathbb{R}^{Dd \times 1} \rightarrow \mathbb{R}^{(D-d_e) \times 1} $ be a vector-valued function defined as
\begin{equation}
  f(\vc[\mtx{C}^T \; \mtx{B}^T]) = \mtx{C}\mtx{B}^T(\mtx{B}\mtx{B}^T)^{-1}\mvec{z}^*,
\end{equation}
where $\vc[\mtx{C}^T \; \mtx{B}^T]$ denotes the vector of the concatenated columns of $(\mtx{C}^T \; \mtx{B}^T)$. We can express $f$ as
$$ f(\vc[\mtx{C}^T \; \mtx{B}^T]) = \begin{pmatrix}
	\frac{p_1(\mtx{C},\mtx{B})}{q(\mtx{B})} & \frac{p_2(\mtx{C},\mtx{B})}{q(\mtx{B})} & \dots & \frac{p_{D-d_e}(\mtx{C},\mtx{B})}{q(\mtx{B})} 
	\end{pmatrix}^T, $$
where $p_i(\mtx{C},\mtx{B})$ for $1 \leq i \leq D-d_e$ are some polynomials in the entries of $\mtx{C}$ and $\mtx{B}$ and $q(\mtx{B}) = \det(\mtx{B}\mtx{B}^T)$. 
Since $q$ and $p_i$'s are polynomials in Gaussian random variables, they are all measurable. Furthermore, since $\mtx{B}$ is Gaussian, by \autoref{thm: Wishart_is_nonsingular}, $\prob[q = 0] = 0$; this implies that $p_i/q$ is a measurable function for each $i = 1,2,\dots,D-d_e$ (see \cite[Theorem 4.10]{Wheeden2015}). 

We have
\begin{equation}
    \text{$\mvec{w} = f(\vc[\mtx{C}^T \; \mtx{B}^T])$ and $\mtx{S}\mvec{w} = f(\vc[(\mtx{S}\mtx{C})^T \; \mtx{B}^T])$}.
\end{equation}
From \Cref{thm: orthog_inv_of_Gaussian_matrices} it follows that $\mtx{C} \stackrel{law}{=} \mtx{S}\mtx{C}$; hence $\vc[\mtx{C}^T \; \mtx{B}^T] \stackrel{law}{=} \vc[(\mtx{S}\mtx{C})^T \; \mtx{B}^T]$. We can now apply \Cref{lemma: Fang_equally_distributed_random_vectors} to conclude that
\begin{equation}
    \mvec{w} = f(\vc[\mtx{C}^T \; \mtx{B}^T]) \stackrel{law}{=} f(\vc[(\mtx{S}\mtx{C})^T \; \mtx{B}^T]) = \mtx{S}\mvec{w}. \qedhere
\end{equation} 
\end{proof}

We are now ready to derive the p.d.f. of $\mvec{w}$, and hence prove \Cref{thm:pdf_of_w}.

\paragraph*{Proof of \Cref{thm:pdf_of_w}:} 
We show that the p.d.f. of $\mvec{w}$ is given by \eqref{eq:pdf_of_w}. The identification with the $t$-distribution follows from \eqref{eq:pdf_of_t_distr}. Let us first show that $\prob[\mvec{w} = \mvec{0}] = 0$. Let $X \sim F(D-d_e,d-d_e+1)$. We have
    \begin{equation}\label{eq:prob[w=0]=0}
        \prob[\mvec{w} = \mvec{0}] = \prob[\|\mvec{w}\|^2 = 0] = \prob[X = 0],
    \end{equation}
    where in the last equality we applied \autoref{thm:dist_of_norm_of_w}. Since the $F$-distributed $X$ is a continuous random variable, the last probability in \eqref{eq:prob[w=0]=0} is equal to zero.
    
    Since $\prob[\mvec{w} = \mvec{0}] = 0$ and $\mvec{w}$ has a spherical distribution (\autoref{thm:w_has_a_spherical_distr}), \autoref{thm: pdf_of_x_Gupta} implies that the p.d.f. $g(\mvec{\bar{w}})$ of $\mvec{w}$ satisfies
    \begin{equation}\label{eq:g=pdf_Gupta}
        g(\mvec{\bar{w}}) = \frac{\Gamma(m/2)}{2\pi^{m/2}} h(\| \mvec{\bar{w}} \|) \| \mvec{\bar{w}} \|^{1-m},
    \end{equation}
    where $h(\cdot)$ denotes the p.d.f.~of $\|\mvec{w}\|$. By substituting \eqref{eq:pdf_of_norm_of_w} into \eqref{eq:g=pdf_Gupta}, we obtain the desired result.\hfill$\Box$

%% file: appendix_proofs.tex

A crucial Lemma is given first. 
\begin{lemma} \label{lemma:prob[success]>I}
In the conditions of \Cref{cor:prob[RPXis_succ]=Omega_general_p},  we have
\begin{equation} \label{eq:prob[success]>I}
    \prob[\text{\eqref{eq: AREGO} is successful}\,] \geq I(\mvec{p},\Delta),
\end{equation}
where $\Delta :=  \|\mvec{x}_{\top}^* - \mvec{p}_{\top}\|$ and
\begin{equation}\label{def:I(p,Delta)}
    I(\mvec{p},\Delta) := \frac{1}{2^{\frac{n}{2}-1}\Gamma(\frac{n}{2})} \int_0^{\infty} \left(\prod_{i=d_e+1}^D \frac{1}{\sqrt{2\pi}} \int_{s(-1-p_i)/\Delta}^{s(1-p_i)/\Delta}e^{-x^2/2} dx \right) s^{n-1} e^{-s^2/2} ds.
\end{equation}
\end{lemma}
\begin{proof}
\Cref{thm:RP_succ_special_case} implies that
\[    \prob[\text{\eqref{eq: AREGO} is successful}\,] \geq \prob[-\mvec{1}-\mvec{p}_{d_e+1:D} \leq \mvec{w} \leq \mvec{1}-\mvec{p}_{d_e+1:D}],  \]
where $\mvec{w}$ follows a $(D-d_e)$-dimensional  $t$-distribution with parameters $n= d-d_e+1$ and $\mtx{\Sigma} = (\Delta^2/n)\mtx{I}$. According to \cite[p.~133]{Gupta1999}, 
\begin{equation}\label{eq:w=scaled_Gaussian}
    \mvec{w} \stackrel{law}{=} \frac{\Delta}{s} \begin{pmatrix} Z_1 \\ \vdots \\ Z_m \end{pmatrix},
\end{equation} 
with $s \sim \sqrt{\chi^2_{n}}$, $m = D-d_e$ and $Z_1, \dots, Z_m$  i.i.d standard Gaussian random variables.
Then, \eqref{eq:w=scaled_Gaussian} yields
\begin{multline}\label{eq:I(p,delta)=prob[-1-p<w<1-p]}
    \prob[\text{\eqref{eq: AREGO} is successful}\,] \geq \prob[-\mvec{1}-\mvec{p}_{d_e+1:D} \leq \mvec{w} \leq \mvec{1}-\mvec{p}_{d_e+1:D}] \\
    = \prob\left[\frac{s}{\Delta}(-1-p_{d_e+1}) \leq Z_1 \leq \frac{s}{\Delta}(1-p_{d_e+1}), \dots, \frac{s}{\Delta}(-1-p_D) \leq Z_m \leq \frac{s}{\Delta}(1-p_D)\right],
\end{multline} 
which can be written as (see \cite[p.~1]{Dunnett1955})
\begin{equation}\label{eq:I=int_of_Gh(s)}
   \int_{0}^{\infty} G(\mvec{p},\Delta,s) h(s) ds,
\end{equation}
where

\begin{equation}\label{def:G(s)}
\begin{aligned}
G(\mvec{p},\Delta,s) & = \int_{s(-1-p_{d_e+1})/\Delta}^{s(1-p_{d_e+1})/\Delta} \dots \int_{s(-1-p_D)/\Delta}^{s(1-p_D)/\Delta} \frac{1}{(2\pi)^{m/2}} e^{-\frac{1}{2}(x_1^2 + \cdots + x_m^2)} dx_1 \dots dx_m \\
& = \prod_{i=d_e+1}^D \frac{1}{\sqrt{2\pi}} \int_{s(-1-p_{i})/\Delta}^{s(1-p_{i})/\Delta}e^{-x^2/2} dx,
\end{aligned}
\end{equation}

\noindent and where $h(s)$ is the pdf of $s$ given by
\begin{equation}\label{def:h(s)}
    h(s) = \frac{1}{2^{\frac{n}{2}-1} \Gamma(\frac{n}{2})} s^{n-1} e^{-s^2/2}.
\end{equation}
By combining \eqref{eq:I(p,delta)=prob[-1-p<w<1-p]} -- \eqref{def:h(s)}, we obtain
\eqref{eq:prob[success]>I}--\eqref{def:I(p,Delta)}.
\end{proof}

It is easier to show \Cref{cor:prob[RPXis_succ]=Omega_p=0} first, when 
$\mvec{p} = \mvec{0}$.


\subsection{Proof of \texorpdfstring{\Cref{cor:prob[RPXis_succ]=Omega_p=0}}{Theorem 3.9} }

The next result is  a direct corollary of \Cref{lemma:prob[success]>I} when $\mvec{p} = \mvec{0}$,
allowing us to replace   $I(\mvec{p},\Delta)$ in \eqref{eq:prob[success]>I} with a new integral $J_{m,n}(\Delta)$ that will be easier to manipulate.
\begin{corollary} \label{cor:prob[success]>J}
In the conditions and notation of \Cref{lemma:prob[success]>I}, let $\mvec{p} = \mvec{0}$. Then
\begin{equation}
    \prob[\text{\eqref{eq: AREGO} is successful}\,] \geq J_{m,n}(\|\mvec{x}_{\top}^*\|),
\end{equation}
where 
\begin{equation}\label{eq:J(delta)}
J_{m,n}(\Delta) := \frac{1}{2^{\frac{n}{2}-1}\Gamma(\frac{n}{2})} \int_0^{\infty} \left(\sqrt{\frac{2}{\pi}} \int_{0}^{s/\Delta} e^{-x^2/2} dx \right)^m s^{n-1} e^{-s^2/2} ds.
\end{equation}
\end{corollary} 

\begin{proof}
Let $\mvec{p}=\mvec{0}$. Then \Cref{lemma:prob[success]>I} implies that 
$\prob[\text{\eqref{eq: AREGO} is successful}\,] \geq I(\mvec{0},\|\mvec{x}_{\top}^*\|),$ where 
\begin{equation}\label{eq:I(0,delta)}
\begin{aligned}
     I(\mvec{0},\|\mvec{x}_{\top}^*\|) & = \frac{1}{2^{\frac{n}{2}-1}\Gamma(\frac{n}{2})} \int_0^{\infty} \left(\prod_{i=d_e+1}^D \frac{1}{\sqrt{2\pi}} \int_{-s/\|\mvec{x}_{\top}^*\|}^{s/\|\mvec{x}_{\top}^*\|}e^{-x^2/2} dx \right) s^{n-1} e^{-s^2/2} ds \\
    & = \frac{1}{2^{\frac{n}{2}-1}\Gamma(\frac{n}{2})} \int_0^{\infty} \left(\sqrt{\frac{2}{\pi}} \int_{0}^{s/\|\mvec{x}_{\top}^*\|} e^{-x^2/2} dx \right)^m s^{n-1} e^{-s^2/2} ds \\
    & = J_{m,n}(\|\mvec{x}_{\top}^*\|). \\
\end{aligned}
\end{equation}
\end{proof}

We need to introduce the following three results on the integral $J_{m,n}(\Delta)$  in \eqref{eq:J(delta)}.
\begin{lemma} \label{lemma:J_monot_decreasing}
The integral $J_{m,n}(\Delta)$  in \eqref{eq:J(delta)} is a monotonically decreasing function of $\Delta$.
\end{lemma}
\begin{proof}
Let $\Delta_1, \Delta_2$ be any positive reals that satisfy $\Delta_1 \leq \Delta_2$. We need to show that $J_{m,n}(\Delta_1) \geq J_{m,n}(\Delta_2)$. This relation follows immediately from the observation that, for any $s \geq 0$,
$$ \sqrt{\frac{2}{\pi}} \int_{0}^{s/\Delta_1} e^{-x^2/2} dx \geq \sqrt{\frac{2}{\pi}} \int_{0}^{s/\Delta_2} e^{-x^2/2} dx $$
since the integrand is positive.
\end{proof}

\begin{lemma}\label{lemma:J<=1}
The integral $J_{m,n}(\Delta)$ defined in \eqref{eq:J(delta)} satisfies  $J_{m,n}(\Delta) \leq 1$ for all $\Delta > 0$.
\end{lemma}
\begin{proof}
Note that, for any $s \geq 0$, we have
$$ \sqrt{\frac{2}{\pi}} \int_{0}^{s/\Delta} e^{-x^2/2} dx \leq \sqrt{\frac{2}{\pi}} \int_{0}^{\infty} e^{-x^2/2} dx = 1. $$
Hence,
$$ J_{m,n}(\Delta) \leq  \frac{1}{2^{\frac{n}{2}-1}\Gamma(\frac{n}{2})} \int_0^{\infty} s^{n-1} e^{-s^2/2} ds = 1. $$
\end{proof}

The following theorem provides an asymptotic expansion of $J_{m,n}(\Delta)$ for large $m$, that has algebraic dependence on $m$.
\begin{theorem}\label{thm:expansion_of_J(delta)}
Let $J_{m,n}(\Delta)$ be the integral defined in \eqref{eq:J(delta)}. Let $n$ and $\Delta$ be fixed and let $r = (n+\Delta^2-2)/2$. If $r \neq 0$ then, for large $m$,
\begin{multline}\label{eq:expansion_of_J}
    J_{m,n}(\Delta) =  \frac{C(n,\Delta)}{(m+1)^{\Delta^2}} \left( (\log (m+1))^r - \frac{r}{2} \log (\log (m+1)) \cdot (\log (m+1))^{r-1} \right.\\ \bigg. + O( (\log (m+1))^{r-1}) \bigg),
\end{multline}
where
\begin{equation*}
C(n,\Delta) = \pi^{\frac{\Delta^2}{2}}\Delta^n \frac{ \Gamma(\Delta^2)}{\Gamma(n/2)}.
\end{equation*}
If $r=0$, then $J_{m,n}(\Delta) = J_{m,1}(1) = 1/(m+1)$.
\end{theorem}
\begin{proof}
The proof of this lemma is similar to the derivations in \cite[Section 2, Chapter 2]{Wong2001}, and is deferred to the end of this appendix.
\end{proof}


\paragraph*{Proof of \Cref{cor:prob[RPXis_succ]=Omega_p=0}}
\Cref{cor:prob[success]>J} implies that 
\begin{equation} \label{eq: boundcorparta}
     \prob[\text{\eqref{eq: AREGO} is successful}\,] \geq I(\mvec{0},\| \mvec{x}_{\top}^*\|) \geq J_{m,n}(\| \mvec{x}_{\top}^*\|). 
\end{equation}
By definition of $\mvec{x}_\top^*$, there exists $\mvec{x}^* \in G$ such that $\mvec{x}_\top^* = \mtx{U} \mtx{U}^T \mvec{x}^*$ with $\mtx{U} = [\mtx{I}_{d_e}; \; \mtx{0}]$. Then $\mvec{x}^*_\top = [\mvec{x}^*_{1:d_e}; \mvec{0}]$ which implies $\| \mvec{x}_{\top}^*\| \leq \sqrt{d_e}$.
By monotonic decrease of $J_{m,n}$ (see \Cref{lemma:J_monot_decreasing}), \eqref{eq: boundcorparta} yields
\[ \prob[\text{\eqref{eq: AREGO} is successful}\,] \geq J_{m,n}(\sqrt{d_e}) \]
for all $\mvec{x}^*, \mvec{p} \in \mathcal{X}$ such that $\mvec{x}_{\top}^* \neq \mvec{p}_{\top}$. If $\mvec{x}_{\top}^* = \mvec{p}_{\top}$, then 
$$ \prob[\text{\eqref{eq: AREGO} is successful}\,] = 1 \geq J_{m,n}(\sqrt{d_e}),$$
where the inequality follows from \Cref{lemma:J<=1}.
Thus, \eqref{eq:prob[success]=>tau_in_corollary} is satisfied for $\tau_{\mvec{0}} = J_{m,n}(\sqrt{d_e})$, and \eqref{eq:asym_exp_p=0} follows from \Cref{thm:expansion_of_J(delta)}.\hfill$\Box$



\subsection{Proof of \texorpdfstring{\Cref{cor:prob[RPXis_succ]=Omega_general_p}}{Theorem 3.8} }
Unlike the case $\mvec{p} = \mvec{0}$, we cannot rewrite directly the integral $I(\mvec{p},\Delta)$ in terms if the integral $J_{m,n}(\Delta)$ (i.e., \Cref{cor:prob[success]>J} does not hold) for $\mvec{p} \in \mathcal{X}$ arbitrary. However, we derive a lower bound on $I(\mvec{p},\Delta)$ in terms of the simpler integral $J_{m,n}(\Delta)$ that is valid for all $\mvec{p} \in \mathcal{X}$. 

\begin{lemma} \label{lem: low_bound_JmnDelta2} For any $\mvec{p} \in \mathcal{X}$ and for any $\Delta > 0$, we have
\[ I(\mvec{p},\Delta) \geq \frac{1}{2^m}J_{m,n}(\Delta/2). \]
\end{lemma}

\begin{proof}
Let us define the function 
\begin{equation}
    g(z,\Delta,s) = \frac{1}{\sqrt{2\pi}} \int_{s(-1-z)/\Delta}^{s(1-z)/\Delta}e^{-x^2/2} dx,
\end{equation}
and note that
\begin{equation}\label{eq:I_in_terms_of_g}
    I(\mvec{p},\Delta) = \frac{1}{2^{\frac{n}{2}-1}\Gamma(\frac{n}{2})} \int_0^{\infty} \left(\prod_{i=d_e+1}^D g(p_i,\Delta,s)\right) s^{n-1} e^{-s^2/2} ds.
\end{equation}
Next we find the minimizers of $g(z,\Delta,s)$ over $z \in [-1,1]$. Introducing the notation $l(z,\Delta,s) := s(1-z)/\Delta$, and using Leibniz integral rule, we obtain
\begin{equation}
\begin{aligned}
    \frac{d g(z,\Delta,s)}{d z} & = e^{\frac{-l(z,\Delta,s)^2}{2}}\frac{d (l(z,\Delta,s))}{d z}-e^{\frac{-l(-z,\Delta,s)^2}{2}}\frac{d(-l(-z,\Delta,s))}{d z} \\
    & = e^{\frac{-s^2(1-z)^2}{2\Delta^2}}\left(\frac{-s}{\Delta}\right) - e^{\frac{-s^2(-1-z)^2}{2\Delta^2}}\left(\frac{-s}{\Delta}\right) \\
    & = \frac{s}{\Delta}e^{-\frac{s^2}{2\Delta^2}(1+z^2)}\left(e^{-\frac{s^2z}{\Delta^2}}-e^{\frac{s^2z}{\Delta^2}}\right).
\end{aligned}
\end{equation}
Hence, $d g(z,\Delta,s)/d z$ is equal to zero if and only if 
\begin{equation}
     e^{-\frac{s^2z}{\Delta^2}}-e^{\frac{s^2z}{\Delta^2}} = 0,
\end{equation}
which occurs only at $z = 0$. The sign of $d g(z,\Delta,s)/d z$ changes from negative to positive at $z=0$ implying that the function is concave and so $g(z,\Delta,s)$ attains its maximum at $z=0$ and its minimum at the boundaries. Since $g(z,\Delta,s)$ is symmetric around $z=0$, the minimum is attained at $z = \pm 1$. Thus, for all $z \in [-1,1]$,
\begin{equation}\label{eq:g(p,s)_lower_bound}
    g(z,\Delta,s) \geq g(-1,\Delta, s) = \frac{1}{\sqrt{2\pi}} \int_{-l(1,\Delta,s)}^{l(-1,\Delta,s)}e^{-x^2/2} dx = \frac{1}{\sqrt{2\pi}} \int_{0}^{\frac{2s}{\Delta}}e^{-x^2/2} dx.
\end{equation}
By combining \eqref{eq:g(p,s)_lower_bound} with \eqref{eq:I_in_terms_of_g},  we obtain
\begin{equation}\label{eq:I(p,delta)_lower_bound}
\begin{aligned}
I(\mvec{p},\Delta) & \geq \frac{1}{2^{\frac{n}{2}-1}\Gamma(\frac{n}{2})} \int_0^{\infty} \left(\prod_{i=d_e+1}^D \frac{1}{\sqrt{2\pi}} \int_{0}^{\frac{2s}{\Delta}}e^{-x^2/2} dx \right) s^{n-1} e^{-s^2/2} ds \\
& = \frac{1}{2^m}\cdot\frac{1}{2^{\frac{n}{2}-1}\Gamma(\frac{n}{2})} \int_0^{\infty} \left(\sqrt{\frac{2}{\pi}} \int_{0}^{\frac{2s}{\Delta}} e^{-x^2/2} dx \right)^m s^{n-1} e^{-s^2/2} ds \\
& = \frac{1}{2^m}J_{m,n}(\Delta/2).
\end{aligned}
\end{equation}
 
\end{proof}

\paragraph*{Proof of \Cref{cor:prob[RPXis_succ]=Omega_general_p}.}
\Cref{lemma:prob[success]>I} and \Cref{lem: low_bound_JmnDelta2} provide
\begin{equation}\label{eq:I(p,Delta)>JDelta/2}
    \prob[\text{\eqref{eq: AREGO} is successful}\,] \geq I(\mvec{p},\Delta) \geq \frac{1}{2^m} J_{m,n}(\Delta/2).
\end{equation}
Let us now show that $\Delta \leq 2\sqrt{d_e}$ for all $\mvec{x}^*,\mvec{p}\in [-1,1]^D$. Since $\mtx{U} = [\mtx{I}_{d_e} \; \mtx{0}]^T$, for any global minimizer $\mvec{x}^*$, we have $\mvec{x}_{\top}^* = \mtx{U} \mtx{U}^T \mvec{x}^* = [\mvec{x}_{1:d_e}^*; \; \mvec{0}]$, and for any $\mvec{p}$, we have $\mvec{p}_{\top} = \mtx{U} \mtx{U}^T \mvec{p} = [\mvec{p}_{1:d_e}; \; \mvec{0}]$. 
Since $\mvec{x}^*,\mvec{p} \in [-1,1]^D$, there holds $\|\mvec{x}_{\top}^*\| \leq \sqrt{d_e}$ and $\|\mvec{p}_{\top}\| \leq \sqrt{d_e}$, and hence, $\Delta = \| \mvec{x}_{\top}^* - \mvec{p}_{\top} \| \leq \|\mvec{x}_{\top}^*\| + \|\mvec{p}_{\top}\| \leq 2\sqrt{d_e}.$

Using the fact that $J_{m,n}(\Delta)$ is a monotonically decreasing function (see \Cref{lemma:J_monot_decreasing}), \eqref{eq:I(p,Delta)>JDelta/2} yields 
\begin{equation}\label{eq:I(p,Delta)>J(sqrt(d_e))/2^m}
    \prob[\text{\eqref{eq: AREGO} is successful}\,] \geq \frac{1}{2^m} J_{m,n}(\sqrt{d_e})
\end{equation}
for all $\mvec{x}^*, \mvec{p} \in \mathcal{X}$ such that $\mvec{x}_{\top}^* \neq \mvec{p}_{\top}$. If $\mvec{x}_{\top}^* = \mvec{p}_{\top}$, then 
$$ \prob[\text{\eqref{eq: AREGO} is successful}\,] = 1 \geq \frac{1}{2^m}J_{m,n}(\sqrt{d_e}),$$
where the inequality follows from \Cref{lemma:J<=1}.
Thus, \eqref{eq:prob[success]=>tau_in_corollary} is satisfied for $\tau = J_{m,n}(\sqrt{d_e})/2^m$, and \eqref{eq:asym_exp_general_p} follows from \Cref{thm:expansion_of_J(delta)}.\hfill$\Box$

\subsection{Proof of \texorpdfstring{\Cref{thm:expansion_of_J(delta)}}{Theorem D.5} }
We rewrite $J_{m,n}(\Delta)$ as follows
\begin{equation*}
\begin{aligned}
    J_{m,n}(\Delta) & = \frac{1}{2^{\frac{n}{2}-1}\Gamma(\frac{n}{2})} \int_0^{\infty} \left(\sqrt{\frac{2}{\pi}} \int_{0}^{s/\Delta} e^{-x^2/2} dx \right)^m s^{n-1} e^{-s^2/2} ds \\
    & = \frac{1}{2^{\frac{n}{2}-1}\Gamma(\frac{n}{2})} \int_0^{\infty} \left(\frac{2}{\sqrt{\pi}} \int_{0}^{\frac{s}{\sqrt{2}\Delta}} e^{-x^2} dx \right)^m s^{n-1} e^{-s^2/2} ds \\
    & = \frac{1}{2^{\frac{n}{2}-1}\Gamma(\frac{n}{2})} \int_0^{\infty} \erf^m\left(\frac{s}{\sqrt{2}\Delta}\right) s^{n-1} e^{-s^2/2} ds,
\end{aligned}
\end{equation*}
where $\erf(\cdot)$ denotes the usual error function. After making an appropriate transformation, the integral becomes
\begin{equation*}
    J_{m,n}(\Delta) = \frac{2 \Delta^n}{\Gamma(\frac{n}{2})} \int_0^{\infty} \erf^m(s) s^{n-1} e^{-\Delta^2s^2} ds
\end{equation*}
In \cite[Section 2, Chapter 2]{Wong2001}, \citeauthor{Wong2001} derives an asymptotic expansion of a similar integral; our derivations are based on his method.

As $s$ varies from $0$ to $\infty$, $\erf(s)$ increases monotonically from $0$ to $1$. So, for $m$ large almost all the mass of the integrand is concentrated at $\infty$. We make the substitution $e^{-t} = \erf(s)$ to bring the integral to the form:
\begin{equation}\label{eq:transformed_I_in_t}
    J_{m,n}(\Delta) = \frac{\sqrt{\pi}\Delta^n }{\Gamma(\frac{n}{2})} \int_{0}^{\infty} e^{Ks(t)^2} s(t)^{n-1} e^{-(m+1)t} dt,
\end{equation}
where $K = 1-\Delta^2$ and $s(t) = \erf^{-1}(e^{-t})$. Due to monotonicity of $\erf$, $s(t)$ is uniquely defined for every $t$. As $\erf$ varies from $0$ to $1$, $t$ varies from $\infty$ to $0$. So the mass of the transformed integrand is now concentrated around 0.

We will derive the asymptotic expansion for \eqref{eq:transformed_I_in_t} in three steps:
\begin{enumerate}
    \item First, we will derive the asymptotic expansion of $e^{Ks(t)^2} s(t)^{n-1}$.
    \item Then, we will show that, for any $0<c<1$, the integral
    \begin{equation*}
        \int_c^{\infty} e^{Ks(t)^2} s(t)^{n-1} e^{-(m+1)t} dt
    \end{equation*}
    is exponentially small.
    \item Finally, we will derive the asymptotic expansion of
    \begin{equation*}
        \int_{0}^c e^{Ks(t)^2} s(t)^{n-1} e^{-(m+1)t} dt.
    \end{equation*}
\end{enumerate}

\subsubsection*{Step 1}
\begin{lemma}\label{lemmma:Asymptote_of_s(t)} (see \cite[Lemma 1, p.~67]{Wong2001})
For small positive $t$, $s(t) = \erf^{-1}(e^{-t})$ satisfies 
\begin{equation*}
    s(t)^2 = -\log(t) - \frac{1}{2} \log(-\log(t)) - \log(\sqrt{\pi}) + \frac{\log(-\log(t))}{4(-\log(t))} - \frac{\log(e/\sqrt{\pi})}{2(-\log(t))} + O\left( \frac{\log^2(-\log(t))}{(\log(t))^2} \right).
\end{equation*}
\end{lemma}
\begin{proof}
The asymptotic expansion of $\erf(s)$ at infinity is given by 
\begin{equation*}
    \erf(s) \sim 1 - \frac{e^{-s^2}}{\sqrt{\pi}s}\left( 1- \frac{1}{2s^2} + \frac{3}{4s^4} - \cdots \right)
\end{equation*}
By writing $1 - e^{-t} = 1-\erf(s)$ and using Taylor's expansion for $e^{-t}$ at $0$, we obtain
\begin{equation*}
    t(1+O(t)) = \frac{e^{-s^2}}{\sqrt{\pi}s}\left( 1- \frac{1}{2s^2} + \frac{3}{4s^4} - \cdots \right).
\end{equation*}
By taking logs on both sides and using the Taylor's expansion for $\log(1+x)$, we have
\begin{equation}\label{eq:taylor_exp_exp(-t)_and_erf}
    \log(t) + O(t) = -s^2-\log(\sqrt{\pi}) - \log(s) - \frac{1}{2s^2} + O\left(\frac{1}{s^4}\right).
\end{equation}
The dominant terms are $\log(t)$ and $s^2$, hence
\begin{equation} \label{eq:s^2_sim_-log(t)}
    \text{$s^2 \sim -\log(t)$, as $t \rightarrow 0^{+}$}. 
\end{equation}
To obtain higher order approximations, we write
\begin{equation*} 
    s(t)^2 = -\log(t) + \epsilon_1(t)
\end{equation*}
and substitute this into \eqref{eq:taylor_exp_exp(-t)_and_erf}. We have
\begin{equation} \label{eq:second_order_expansion_with_esp}
\begin{aligned}
    \log(t) + O(t) = \log(t)-\epsilon_1(t) - \log(\sqrt{\pi}) - \frac{1}{2}\log(-\log(t))& -\frac{1}{2}\log\left( 1+ \frac{\epsilon_1(t)}{-\log(t)} \right) + \\
    & + O\left( \frac{1}{-\log(t)+\epsilon_1(t)}  \right)
\end{aligned}
\end{equation}
Note that by \eqref{eq:s^2_sim_-log(t)}, as $t \rightarrow 0^{+}$
\begin{equation} \label{eq:eps/-log(t)->0}
    \frac{\epsilon_1(t)}{-\log(t)} \rightarrow 0.
\end{equation}
By using \eqref{eq:eps/-log(t)->0} in \eqref{eq:second_order_expansion_with_esp}, we obtain
\begin{equation}\label{eq:def_eps_1}
    \epsilon_1(t) = - \frac{1}{2}\log(-\log(t))- \log(\sqrt{\pi}) + o(1).
\end{equation}
To obtain the following leading terms in the approximation we write
\begin{equation}\label{eq:s^2=first+second+eps_2}
    s^2(t) = -\log(t) - \frac{1}{2}\log(-\log(t))- \log(\sqrt{\pi}) + \epsilon_2(t)
\end{equation}
and repeat the above procedure. We substitute  \eqref{eq:s^2=first+second+eps_2} into \eqref{eq:taylor_exp_exp(-t)_and_erf} and after a little manipulation obtain
\begin{equation}\label{eq:eps_2_big_expansion}
\begin{aligned}
    O(t) = -\epsilon_2(t) & -\frac{1}{2} \log \left( 1 - \frac{1}{2} \frac{\log(-\log(t))}{-\log(t)} - \frac{\log(\sqrt{\pi})}{-\log(t)} + \frac{\epsilon_2(t)}{-\log(t)} \right)- \\
    & -\frac{1}{2} \cdot \frac{1}{-\log(t)} \cdot \frac{1}{1-\frac{1}{2}\frac{\log(-\log(t))}{-\log(t)} - \frac{\log(\sqrt{\pi})}{-\log(t)} + \frac{\epsilon_2(t)}{-\log(t)}} + O((-\log(t))^{2})
\end{aligned}
\end{equation}
Using the fact (by \eqref{eq:def_eps_1}) that $\epsilon_2(t) = o(1)$ and Taylor's expansions for $\log(1+x)$ and $1/(1-x)$, we obtain
\begin{equation*}
\begin{aligned}
    O(t) = -\epsilon_2(t) & - \frac{1}{2}\left(- \frac{1}{2} \frac{\log(-\log(t))}{-\log(t)} + O\left( \frac{1}{-\log(t)} \right)  \right) - \\
    & -\frac{1}{2} \cdot \frac{1}{-\log(t)} \left( 1 + O\left(\frac{\log(-\log(t))}{-\log(t)} \right) \right),
\end{aligned}
\end{equation*}
which yields
\begin{equation}\label{eq:def_esp_2_first_term_expansion}
    \epsilon_2(t) = \frac{\log(-\log(t))}{4(-\log(t))} + O\left( \frac{1}{-\log(t)} \right).
\end{equation}
To obtain the following leading terms in the expansion of $\epsilon_2(t)$, we use \eqref{eq:def_esp_2_first_term_expansion} in \eqref{eq:eps_2_big_expansion} leaving the first term ($-\epsilon_2(t)$) as is:
\begin{equation*}
    \begin{aligned}
        O(t) = -\epsilon_2(t) & -\frac{1}{2} \log \left( 1 - \frac{1}{2} \frac{\log(-\log(t))}{-\log(t)} - \frac{\log(\sqrt{\pi})}{-\log(t)} + O\left(\frac{\log(-\log(t))}{(-\log(t))^2}\right) \right)- \\
        & -\frac{1}{2} \cdot \frac{1}{-\log(t)} \cdot \frac{1}{1-\frac{1}{2}\frac{\log(-\log(t))}{-\log(t)} - \frac{\log(\sqrt{\pi})}{-\log(t)} + O\left(\frac{\log(-\log(t))}{(-\log(t))^2}\right)} + O((-\log(t))^{2})
    \end{aligned}
\end{equation*}
Now, using Taylor's expansions for $\log(1+x)$ and $1/(1-x)$, we obtain
\begin{equation*}
\begin{aligned}
    O(t) = -\epsilon_2(t) & - \frac{1}{2}\left(- \frac{1}{2} \frac{\log(-\log(t))}{-\log(t)} - \frac{\log(\sqrt{\pi})}{-\log(t)} +  O\left(\frac{\log^2(-\log(t))}{(-\log(t))^2}\right)  \right) - \\
    & -\frac{1}{2} \cdot \frac{1}{-\log(t)} \left( 1 + O\left(\frac{\log(-\log(t))}{-\log(t)} \right) \right),
\end{aligned}
\end{equation*}
Hence,
\begin{equation*}
    \epsilon_2(t) = \frac{\log(-\log(t))}{4(-\log(t))} - \frac{\log(e/\sqrt{\pi})}{2(-\log(t))} + O\left( \frac{\log^2(-\log(t))}{(-\log(t))^2} \right).
\end{equation*}
\end{proof}

\begin{corollary}\label{cor:expansion_of_exp(Ks^2)s^(n-1)}
Let $l(t) = -\log(t)$. Then, as $t \rightarrow 0^{+}$,
\begin{equation} \label{eq:expansion_of_exp(Ks(t)^2)s(t)^(n-1)}
\begin{aligned}
    e^{Ks(t)^2}s(t)^{n-1} = e^{Kl(t)}\pi^{-K/2} l(t)^{\frac{n-1-K}{2}} \left(1 - \left(\frac{n-1-K}{4}\right)\frac{\log(l(t))}{l(t)} \right. & - \frac{\log(e^{K/2}\pi^{\frac{n-1-K}{4}})}{l(t)}
     \\
    & \left. + O{\left( \frac{\log^2(l(t))}{l(t)^2}\right)} \right)
\end{aligned}
\end{equation}
\end{corollary}
\begin{proof}
From \autoref{lemmma:Asymptote_of_s(t)} it follows that 
\begin{equation*}
    e^{Ks(t)^2} = e^{Kl(t)}l(t)^{-K/2}\pi^{-K/2}\exp\left(\frac{K\log(l(t))}{4l(t)} - \frac{K\log(e/\sqrt{\pi})}{2l(t)} + O\left( \frac{\log^2(l(t))}{l(t)^2} \right) \right).
\end{equation*}
By using Taylor's expansion for $\exp$ we obtain 
\begin{equation}\label{eq:expansion_of_exp(Ks(t)^2)}
    e^{Ks(t)^2} = e^{Kl(t)}l(t)^{-K/2}\pi^{-K/2}\left(1+\frac{K\log(l(t))}{4l(t)} - \frac{K\log(e/\sqrt{\pi})}{2l(t)} + O\left(\frac{\log^2(l(t))}{l(t)^2}\right) \right).
\end{equation}
Similarly, using \autoref{lemmma:Asymptote_of_s(t)} and binomial expansion, for $s(t)^{n-1}$, we have
\begin{equation}\label{eq:expansion_of_s(t)^(n/2-1/2)}
    (s(t)^2)^{\frac{n-1}{2}} = l(t)^{\frac{n-1}{2}}\left(1- \frac{(n-1)\log(l(t))}{4l(t)}-\frac{(n-1)\log(\sqrt{\pi})}{2l(t)} + O\left(\frac{\log^2(l(t))}{l(t)^2}\right) \right)
\end{equation}
By multiplying the leading terms in \eqref{eq:expansion_of_exp(Ks(t)^2)} and \eqref{eq:expansion_of_s(t)^(n/2-1/2)}, we obtain the desired result.
\end{proof}

\subsubsection*{Step 2}
Let $0<c<1$. We will show that, for large $m$,
    \begin{equation*}
        \int_c^{\infty} e^{Ks(t)^2} s(t)^{n-1} e^{-(m+1)t} dt = O\left(\frac{e^{-c(m+n)}}{m+n}\right).
    \end{equation*}
Let $\erf(s) = e^{-t}$. First, we establish that
\begin{equation}\label{stat:inv_err<Ax}
\text{there exists a positive constant $A$ such that $s(t) = \erf^{-1}(e^{-t}) \leq Ae^{-t}$ for all $t \in [c,\infty)$.}
\end{equation}
Note that \eqref{stat:inv_err<Ax} holds if there exists an $A>0$ such that $\erf^{-1}(x) \leq Ax$ for all $x \in [0,e^{-c}]$.
To prove this, we apply the Mean Value Theorem to $\erf^{-1}$ over $[0,x]$; by the Mean Value Theorem there exists $y\in(0,x)$ such that
\begin{equation}\label{eq:inv_err_MVT}
    \frac{\erf^{-1}(x) - \erf^{-1}(0)}{x-0} = (\erf^{-1})^{'}(y)
\end{equation}
Using the following formula for the derivative of the inverse of the error function \cite[eq (2.4), p.~192]{Amdeberhan2008},
\begin{equation*}
    (\erf^{-1}(x))' = \frac{\sqrt{\pi}}{2} e^{(\erf^{-1}(x))^2},
\end{equation*}
from \eqref{eq:inv_err_MVT}, we obtain
\begin{equation} \label{eq:inv_erf/x=der}
    \frac{\erf^{-1}(x)}{x} = \frac{\sqrt{\pi}}{2} e^{(\erf^{-1}(y))^2}.
\end{equation}
Since $\erf^{-1}$ is an increasing function and $y < x \leq e^{-c}$, \eqref{eq:inv_erf/x=der} gives
\begin{equation*}
    \erf^{-1}(x) \leq \frac{\sqrt{\pi}}{2} e^{(\erf^{-1}(e^{-c}))^2} x,
\end{equation*}
which proves \eqref{stat:inv_err<Ax}.

Now, since $s(t)$ is a monotonically decreasing function with $s(\infty) = 0$, we have\footnote{Over $t \in [c,\infty)$, for $K\geq 0$, $e^{Ks(t)^2}\leq e^{Ks(c)^2}$ and, for $K<0$, $e^{Ks(t)^2}\leq 1$.}
\begin{equation}\label{ineq:e^(Ks^2)<max(1,e^(Ks(c)^2))}
\text{$e^{Ks(t)^2} \leq \max\{1,e^{Ks(c)^2}\}$ for $t \geq c$.}
\end{equation}
Using \eqref{stat:inv_err<Ax} and \eqref{ineq:e^(Ks^2)<max(1,e^(Ks(c)^2))}, we finally obtain
\begin{equation*}
\begin{aligned}
    \int_c^{\infty} e^{Ks(t)^2} s(t)^{n-1} e^{-(m+1)t} dt & \leq A^{n-1} \max\{1,e^{Ks(c)^2}\}  \int_c^{\infty} e^{-(m+n)t} dt \\
    & = A^{n-1} \max\{1,e^{Ks(c)^2}\} \frac{e^{-c(m+n)}}{m+n}.
\end{aligned}
\end{equation*}


\subsubsection*{Step 3}
Let $L(\lambda,\mu,z)$ and $G(\lambda,\mu,z)$ be defined as follows
\begin{equation*}
    L(\lambda,\mu,z) = \int_0^c t^{\lambda - 1} (-\log(t))^{\mu} e^{-zt} dt
\end{equation*}
and
\begin{equation*}
    G(\lambda,\mu,z) = \int_0^c t^{\lambda - 1} (-\log(t))^{\mu} \log(-\log(t)) e^{-zt} dt,
\end{equation*}
where $0<c<1$. The expansion of $e^{Ks(t)^2}s(t)^{n-1}$ in \autoref{cor:expansion_of_exp(Ks^2)s^(n-1)} gives 
\begin{equation}\label{eq:expansion_of_int_0^c}
\begin{aligned}
\int_{0}^{c} e^{Ks(t)^2} s(t)^{n-1} e^{-(m+1)t} dt = & \pi^{-K/2}L\left(1-K,\frac{n-1-K}{2},m+1\right)  \\
& - \pi^{-K/2} \left(\frac{n-1-K}{4}\right) G\left(1-K,\frac{n-3-K}{2},m+1  \right) \\
& - \pi^{-K/2} \log(e^{K/2}\pi^{\frac{n-1-K}{4}}) L\left(1-K,\frac{n-3-K}{2},m+1\right) + \cdots,
\end{aligned}
\end{equation}

The following theorem provides the asymptotic expansion for $L(\lambda,\mu,z)$.
\begin{theorem} \label{thm:expansion_of_L} (see \cite[Theorem 2, p.~70]{Wong2001})
Let $0<c<1$ and let $\lambda$ and $\mu$ be any real numbers with $\lambda > 0$. We have
\begin{equation*}
    L(\lambda,\mu,z) \sim z^{-\lambda}(\log (z))^{\mu} \sum_{r=0}^{\infty} (-1)^r \binom{\mu}{r} \Gamma^{(r)}(\lambda) (\log(z))^{-r}
\end{equation*}
as $z \rightarrow \infty$,
where $\Gamma^{(r)}$ denotes the $r$th derivative of the gamma function.
\end{theorem}
In the following theorem we derive the asymptotic expansion for $G(\lambda,\mu,z)$ based on the proof of \cite[Theorem 2, p.~70]{Wong2001}. 
\begin{theorem}\label{thm:expansion_of_G}
Let $0<c<1$ and let $\lambda$ and $\mu$ be any real numbers with $\lambda > 0$. We have
\begin{equation*}
\begin{aligned}
    G(\lambda,\mu,z) \sim & z^{-\lambda}(\log (z))^{\mu} \log(\log(z)) \sum_{r=0}^{\infty} (-1)^r \binom{\mu}{r} \Gamma^{(r)}(\lambda) (\log(z))^{-r} + \\
    & + z^{-\lambda} (\log(z))^{\mu} \sum_{r=1}^{\infty} a_r \Gamma^{(r)} (\lambda) (\log(z))^{-r},
\end{aligned}
\end{equation*}
as $z \rightarrow \infty$, where 
\begin{equation} \label{eq:def_of_a_r}
    \text{$a_r = -\sum_{i=0}^{r-1} \binom{\mu}{i} \frac{(-1)^i}{r-i}$ for $r = 1,2,\dots$.}
\end{equation}
\end{theorem}
\begin{proof}
With the substitution $u = zt$, we obtain
\begin{equation}\label{eq:G_after_substitution}
\begin{aligned}
    G(\lambda,\mu,z) & = z^{-\lambda}\int_{0}^{cz} u^{\lambda - 1} (\log(z) - \log(u))^{\mu} \log(\log(z)-\log(u)) e^{-u} du \\
    & = z^{-\lambda} (\log(z))^{\mu} \int_{0}^{cz} u^{\lambda -1} \left( 1 - \frac{\log(u)}{\log(z)} \right)^{\mu} \left(\log(\log(z)) + \log\left(1-\frac{\log(u)}{\log(z)}\right)   \right)e^{-u}du \\
    & = z^{-\lambda} (\log(z))^{\mu} (\log(\log(z))G_1 + G_2),
\end{aligned}
\end{equation}
where 
\begin{equation*}
    G_1 = \int_{0}^{cz} u^{\lambda -1} \left( 1 - \frac{\log(u)}{\log(z)} \right)^{\mu} e^{-u}du
\end{equation*}
and
\begin{equation}\label{eq:G_2_defined}
    G_2 = \int_{0}^{cz} u^{\lambda -1} \left( 1 - \frac{\log(u)}{\log(z)} \right)^{\mu}\log\left(1-\frac{\log(u)}{\log(z)}\right) e^{-u}du.
\end{equation}
We first derive the asymptotic expansion for $G_2$, the asymptotic expansion for $G_1$ can then be derived in a similar manner.

Let $N$ be an arbitrary positive integer such that $N+1 \geq \mu$. By Taylor's expansion, 
\begin{equation*}
\begin{aligned}
    \left( 1 - \frac{\log(u)}{\log(z)} \right)^{\mu} & = \sum_{r=0}^{N} (-1)^r\binom{\mu}{r} \left( \frac{\log(u)}{\log(z)} \right)^r + R_{1,N} \\
    \log\left(1-\frac{\log(u)}{\log(z)}\right) & = - \sum_{r=1}^{N} \frac{1}{r} \left( \frac{\log(u)}{\log(z)} \right)^r + R_{2,N}, 
\end{aligned}
\end{equation*}
for all $0<u<cz$, where
\begin{equation*}
\text{$|R_{i,N}| \leq C_{i,N}  \frac{|\log(u)|^{N+1}}{|\log(z)|^{N+1}}$ ($i = 1,2$)}
\end{equation*}
for some fixed constants $C_{1,N},C_{2,N} > 0$. Hence,
\begin{equation}\label{eq:TE_of_(1-logu)log(1-logu)}
 \left( 1 - \frac{\log(u)}{\log(z)} \right)^{\mu} \log\left(1-\frac{\log(u)}{\log(z)}\right) =  \sum_{r=1}^{2N} a_r \left( \frac{\log(u)}{\log(z)} \right)^r + R_{2N},
\end{equation}
for all $0< u < cz$, where $a_r$'s are defined as in \eqref{eq:def_of_a_r} and
\begin{equation*}
|R_{2N}| \leq C_{2N} \frac{|\log(u)|^{2N+1}}{|\log(z)|^{2N+1}}
\end{equation*}
for some fixed $C_{2N} > 0$. By substituting \eqref{eq:TE_of_(1-logu)log(1-logu)} in \eqref{eq:G_2_defined}, we obtain
\begin{equation*}
    G_2 = \sum_{r=1}^{2N} a_r (\log(z))^{-r} \int_0^{cz} u^{\lambda - 1} (\log(u))^r e^{-u} du + r_{2N},
\end{equation*}
where 
\begin{equation*}
    r_{2N} = \int_0^{cz} u^{\lambda -1} e^{-u} R_{2N} du.
\end{equation*}
\citeauthor{Wong2001} showed in \cite[p.~71]{Wong2001} that, as $z \rightarrow \infty$,
\begin{equation*}
    \int_0^{cz} u^{\lambda - 1} (\log(u))^r e^{-u} du = \Gamma^{(r)}(\lambda) + O(e^{-\epsilon c z}),
\end{equation*}
where $\epsilon \in (0,1/2)$. Furthermore,
\begin{equation*}
\begin{aligned}
    |r_{2N}| & \leq C_{2N} |\log(z)|^{-2N-1} \int_{0}^{cz} |u^{\lambda -1} \log(u)^{2N+1} e^{-u}|du \\
    & \leq C_{2N} |\log(z)|^{-2N-1} \int_{0}^{\infty} |u^{\lambda -1} \log(u)^{2N+1} e^{-u}|du
\end{aligned}
\end{equation*}
It can be shown that the latter integral is bounded (see \cite[eq (2.27), p.~71]{Wong2001}; thus, $r_{2N} = O(\log(z)^{-2N-1})$.
Hence, 
\begin{equation}\label{eq:expansion_of_G_2}
    G_2 = \sum_{r=1}^{2N} a_r \Gamma^{(r)}(\lambda) (\log(z))^{-r} + O(\log(z)^{-2N-1}).  
\end{equation}
In a similar manner, one can show that
\begin{equation}\label{eq:expansion_of_G_1}
    G_1 = \sum_{r=0}^{N} (-1)^r \binom{\mu}{r} \Gamma^{(r)}(\lambda) (\log(z))^{-r} + O(\log(z)^{-N-1}).
\end{equation}
Combining \eqref{eq:G_after_substitution}, \eqref{eq:expansion_of_G_2} and \eqref{eq:expansion_of_G_1}, we obtain the desired result.
\end{proof}

\subsubsection*{Conclusions}
\begin{equation}
\begin{aligned}
    J_{m,n}(\Delta) & = \frac{2 \Delta^n}{\Gamma(\frac{n}{2})} \int_0^{\infty} \erf^m(s) s^{n-1} e^{-\Delta^2s^2} ds \\
    & = \frac{\sqrt{\pi}\Delta^n }{\Gamma(\frac{n}{2})} \int_{0}^{\infty} e^{Ks(t)^2} s(t)^{n-1} e^{-(m+1)t} dt \\
    & = \frac{\sqrt{\pi}\Delta^n }{\Gamma(\frac{n}{2})} \int_{0}^{c} e^{Ks(t)^2} s(t)^{n-1} e^{-(m+1)t} dt + O\left(\frac{e^{-c(m+n)}}{m+n}\right). 
\end{aligned}
\end{equation}
By using \autoref{thm:expansion_of_L} and \autoref{thm:expansion_of_G} in \eqref{eq:expansion_of_int_0^c} and substituting $K = 1- \Delta^2$, we obtain \eqref{eq:expansion_of_J}. Note that if $r = 0$ then $n=1$ and $\Delta = 1$ and so $K=0$. In this case, $e^{Ks(t)^2} s(t)^{n-1} = 1$ and direct integration yields $J_{m,1}(1) = 1/(m+1)$.

%% file: appendix_problem_set.tex
\Cref{table: Test set} contains the explicit formula, domain and global minimum of the functions used to generate the high-dimensional test set. The problem set contains 19 problems taken from \cite{AMPGO, Ernesto2005, Bingham2013}. Problems that cannot be solved by BARON are marked with `$^*$'. Problems that will not be solved by KNITRO are marked with `$^\circ$'.

We briefly describe the technique we adapted from Wang et al. \cite{Wang2016} to generate high-dimensional functions with low effective dimensionality, which was first applied to the above test set in \cite{Cartis2020}. Let $\bar{g}(\bar{\mvec{x}})$ be any function from \Cref{table: Test set}; let $d_e$ be its dimension and let the given domain be scaled to $[-1, 1]^{d_e}$. We create a $D$-dimensional function $g(\mvec{x})$ by adding $D-d_e$ fake dimensions to $\bar{g}(\bar{\mvec{x}})$, $ g(\mvec{x}) = \bar{g}(\bar{\mvec{x}}) + 0\cdot x_{d_e+1} + 0 \cdot x_{d_e+2} + \cdots + 0\cdot x_{D}$. We further rotate the function by applying a random orthogonal matrix $\mtx{Q}$ to $\mvec{x}$ to obtain a non-trivial constant subspace. The final form of the function we test is
\begin{equation}\label{eq: f=g(Qx)}
	f(\mvec{x}) = g(\mtx{Q}\mvec{x}).
\end{equation}
Note that the first $d_e$ rows of $\mtx{Q}$ now span the effective subspace $\mathcal{T}$ of $f(\mvec{x})$. 

For each problem in the test set, we generate three functions $f$ as defined in \eqref{eq: f=g(Qx)}, one for each $D = 10$, $100$, $1000$.

\begin{table}[!ht]
	\centering
	\caption{The problem set listed in alphabetical order.}
	\label{table: Test set}
	\begin{tabular} {|L{4.1cm} | C{3cm} | C{3.3cm} | }
		\hline
		Function & Domain  & Global minima  \\ \hline
		1) Beale \cite{Ernesto2005}    & $\mvec{x} \in [-4.5,4.5]^2$ &  $g(\mvec{x}^*) = 0$ \\ \hline
		2) $^*$Branin \cite{Ernesto2005}    & \pbox{20cm}{
			$x_1 \in [-5,10]$ \\ $x_2 \in [0, 15]$} 
		&
		$g(\mvec{x}^*) = 0.397887$ \\ \hline
		
		3) Brent \cite{AMPGO}  & $\mvec{x} \in [-10,10]^2$ & $g(\mvec{x}^*) = 0$  \\ \hline
		
		4) $^\circ$Bukin N.6   \cite{Bingham2013}   & \pbox{20cm}{ $ x_1 \in [-15,-5]$ \\ $x_2 \in [-3,3]$} & $g(\mvec{x}^*) = 0$ \\ \hline
		
		5) $^*$Easom  \cite{Ernesto2005} & $\mvec{x}\in [-100,100]^2$   & $g(\mvec{x}^*) = -1$ \\ \hline
		
		6) Goldstein-Price \cite{Ernesto2005} & $\mvec{x} \in [-2,2]^2 $  & $g(\mvec{x}^*) = 3$ \\ \hline
		
		7) Hartmann 3 \cite{Ernesto2005} & $\mvec{x} \in [0,1]^3$ & $g(\mvec{x}^*) = -3.86278$
		\\ \hline
		
		8) Hartmann 6 \cite{Ernesto2005}  & $\mvec{x} \in [0,1]^6$  & $g(\mvec{x}^*) = -3.32237$
		\\ \hline
		
		9) $^*$Levy \cite{Bingham2013}  & $\mvec{x} \in [-10,10]^4$ & $g(\mvec{x}^*) = 0$  \\ \hline
		
		10) Perm 4, 0.5 \cite{Bingham2013} & $\mvec{x} \in [-4,4]^4$ & $g(\mvec{x}^*) = 0$  \\ \hline
		
		11) Rosenbrock \cite{Bingham2013}   & $\mvec{x} \in [-5,10]^3$ & $g(\mvec{x}^*) = 0$  \\ \hline
		
		12) Shekel $5$ \cite{Bingham2013}  & $\mvec{x} \in [0,10]^4$ & $ g(\mvec{x}^*) = -10.1532$
		\\ \hline
		
		13) Shekel $7$ \cite{Bingham2013}  & $\mvec{x} \in [0,10]^4$ & $g(\mvec{x}^*) = -10.4029$
		\\ \hline
		
		14) Shekel $10$ \cite{Bingham2013} & $\mvec{x} \in [0,10]^4$ & $g(\mvec{x}^*) = -10.5364$
		\\ \hline
		
		15) $^*$Shubert \cite{Bingham2013}  & $\mvec{x} \in [-10,10]^2$ & $g(\mvec{x}^*) = -186.7309$ \\ \hline
		
		16) Six-hump camel \cite{Bingham2013} & \pbox{20cm}{$x_1 \in [-3,3]$ \\ $x_2 \in [-2,2]$} & $g(\mvec{x}^*) = -1.0316$ \\ \hline
		
		17) Styblinski-Tang \cite{Bingham2013}   & $\mvec{x} \in [-5,5]^4$ & $g(\mvec{x}^*) = -156.664$  \\ \hline
		
		18) Trid  \cite{Bingham2013} & $\mvec{x} \in [-25,25]^5$ & $g(\mvec{x}^*) = -30$ \\ \hline
		
		19) Zettl \cite{Ernesto2005}   & $\mvec{x} \in [-5,5]^2$ & $g(\mvec{x}^*) = -0.00379$  \\ \hline
	\end{tabular}
\end{table}